\pdfoutput=1

\documentclass[12pt]{article}
  
\usepackage{amssymb,amsmath,amsthm,geometry}
\usepackage[square,numbers]{natbib}
\usepackage{IEEEtrantools}
\usepackage{mathtools}
\usepackage{tikz}
\usepackage{subcaption}
\usepackage{mathabx}
\usepackage{hyperref}
\usetikzlibrary{calc}
\usepackage{graphicx}
\usepackage{multirow}
\usepackage[title]{appendix}
\newcommand*{\argm}{\text{argmin}}
\newtheorem{theorem}{Theorem}

\newtheorem{corollary}{Corollary}[theorem]

\newtheorem{innercustomgeneric}{\customgenericname}
\providecommand{\customgenericname}{}
\newcommand{\newcustomtheorem}[2]{%
  \newenvironment{#1}[1]
  {%
   \renewcommand\customgenericname{#2}%
   \renewcommand\theinnercustomgeneric{##1}%
   \innercustomgeneric
  }
  {\endinnercustomgeneric}
}

\newcustomtheorem{customthm}{Theorem}
\newcustomtheorem{customlemma}{Lemma}

\newcommand{\sumonn}{\sum_{i=1}^n}

\hypersetup{
    colorlinks=true,
    linkcolor=blue,
    filecolor = blue
, urlcolor = blue,
citecolor = blue
}

\title{The Stein Effect for Fr\'echet Means} 
\author{Andrew McCormack and Peter Hoff \\
Department of Statistical Science \\
Duke University} 
\date{\today}

\begin{document}
\maketitle

\begin{abstract}

The Fr\'echet mean is a useful description of location for a probability distribution on a metric space that is not necessarily a vector space.  This article considers simultaneous estimation of multiple Fr\'echet means from a decision-theoretic perspective, and in particular, the extent to which the unbiased estimator of a Fr\'echet mean can be dominated by a generalization of the James-Stein shrinkage estimator. It is shown that if the metric space satisfies a non-positive curvature condition, then this generalized James-Stein estimator asymptotically dominates the unbiased estimator as the dimension of the space grows. These results hold for a large class of distributions on a variety of spaces - including Hilbert spaces - and therefore partially extend known results on the applicability of the James-Stein estimator to non-normal distributions on Euclidean spaces. Simulation studies on metric trees and symmetric-positive-definite matrices are presented, numerically  demonstrating the efficacy of this generalized James-Stein estimator.  

\smallskip
\noindent {\bf Keywords:} admissibility, empirical Bayes, Hadamard space,  
hierarchical model, 
nonparametric, 
shrinkage. 
\end{abstract}

\section{Introduction}
In his seminal 1948 article, Fr\'echet 
generalized the notion of the mean of a real-valued random variable to a metric space-valued random object \cite{Frechet}. Like the usual mean, the Fr\'echet mean provides a summary of the location of a distribution, from which a notion of Fr\'echet variance may also be defined. Fr\'echet means and variances have been used for statistical analysis of data from non-standard sample spaces, such as spaces of phylogenetic trees, symmetric positive-definite matrices in diffusion tensor imaging, and functional data analysis on Wasserstein spaces, to name a few \cite{Holmes,Pennec,MullerWasserstein}.
In terms of methodological development, \cite{MullerRegression, MullerANOVA} use Fr\'echet means to develop extensions of linear regression and ANOVA that are applicable for metric space-valued data. 
Additionally, substantial effort has gone into studying the convergence properties of sample Fr\'echet means and variances \cite{BhatandPat1, Ginestet, Ziezold}.

This article primarily considers simultaneous estimation of multiple Fr\'echet means, 
and conditions under which a generalized James-Stein shrinkage estimator dominates 
the natural estimator, the unbiased estimator of the Frech\'et mean. Specifically, let $X_1,\ldots, X_n$ be independent random objects taking values in a metric space, with Fr\'echet means 
$\theta_1,\ldots, \theta_n$ respectively, so that $X=(X_1,\ldots, X_n)$ is an 
unbiased estimator of $\theta= (\theta_1,\ldots, \theta_n)$. As shown by 
\cite{Stein}, if $X \sim N_n(\theta, \sigma^2 I)$ with $\sigma^2$ known and $n\geq3$, 
$X$ is dominated by the James-Stein shrinkage estimator $\delta_{JS}(X)$, given by 
\begin{align}
\label{JSEuclidean}
    \delta_{JS}(X) = \left( \frac{\sigma^2(n-2)}{\Vert X  -\psi \Vert^2}\right) \psi +  \left(1 - \frac{\sigma^2(n-2)}{\Vert X- \psi \Vert^2}\right)X,
\end{align}
where $\psi$ is a known shrinkage point. Intuitively, this estimator is obtained by starting from $X$ and ``shrinking'' towards $\psi$ by an amount that is adaptively estimated from the data $X$. The fact that $\delta_{JS}$ dominates $X$ is often interpreted as an indication of how sharing 
information across seemingly unrelated populations can lead to an improved estimator of $\theta_1,\ldots, \theta_n$ with respect to squared error loss summed across all populations. Indeed, the James-Stein estimator may be derived as an empirical Bayes estimator in which $\Vert X- \psi \Vert^2$ provides information about the likely magnitude of $\Vert \theta- \psi \Vert^2$ \cite{Efron}. 

In this article we study a generalization of $\delta_{JS}$ that is applicable for sample spaces that are uniquely geodesic metric spaces, which are metric spaces where there is a unique path of minimum length, or geodesic, between any two points. The estimator of 
 $\theta_1,\ldots, \theta_n$ we consider is a generalization of $\delta_{JS}(X)$ in the sense that
 the resulting estimator of $\theta$ is obtained by traveling from $X$ to the shrinkage point $\psi$ along a geodesic 
 by an amount that is adaptively estimated from $X$. If the geodesics in the metric space have tractable, known forms, then this estimator is simple to compute in practice. As we show, the possibility of 
 a Stein effect, that is, domination of $X$ by this shrinkage estimator, will partly depend on the curvature of the sample 
 space: The Stein effect is generally absent in spaces with positive curvature, and generally present 
 in flat spaces or spaces with negative curvature. These latter two spaces are known as Hadamard spaces \cite{Sturm}, and encompass a wide variety of metric spaces such as the aforementioned spaces of trees, symmetric positive-definite matrices and Wasserstein space on $\mathbb{R}$. Our results show that under some mild conditions, the proposed geodesic James-Stein estimator asymptotically dominates the unbiased estimator. Of note is that the domination results obtained are non-parametric; only moment bounds are placed on the family of distributions under consideration. As a consequence, the geodesic James-Stein estimator is robust, having reasonable performance across a wide range of distributions. Notably, since any Hilbert space is a Hadamard space, all of the results we develop also apply to Euclidean sample spaces.  Previous work generalizing the Stein estimator in Euclidean space  primarily has involved extending domination results to non-normal distributions \cite{Kubokawa}. Typically such distributions are assumed to have some sort of spherical symmetry or exponential family structure which allows for variants of Stein's Lemma to be applied \cite{BrandweinSpherical,Hudson}.  A related focus of research on Stein estimators has been finding estimators that dominate the positive part James-Stein estimator, which is known to be inadmissible \cite{BrownDiffusion, ShaoStrawderman}.

An outline of the remainder of this article is as follows: In Section \ref{Sec2} the concepts of Fr\'echet means, variances and Hadamard spaces are reviewed. Section \ref{Sec3} applies these concepts to the problem of estimating a Fr\'echet mean, and considers randomized, unbiased and minimax estimators. Section \ref{Sec4} provides the core theoretical results of the article, where the geodesic James-Stein estimator is introduced and its risk function for the multi-group estimation problem is investigated. A natural extension of this problem is to place a prior distribution on the Fr\'echet means of each group. This is done in Section \ref{Sec5} where we introduce the possibility of adaptively estimating a shrinkage point. Asymptotic optimality properties of the geodesic James-Stein estimator and the relationship to empirical Bayes estimators are also discussed in this section. Lastly, we demonstrate numerically how the geodesic James-Stein estimator exhibits favorable performance relative to $X$ in simulation studies on the space of symmetric positive-definite matrices and metric tree space.  

\section{Preliminaries}
\label{Sec2}
\subsection{Metric Space Valued Random Objects} 
Let $(\mathcal{X},d)$ be a metric measure space equipped with the Borel $\sigma$-algebra $\mathcal{B}$, induced from the metric topology on $\mathcal{X}$. A metric space valued random object $X$ is a $\mathcal{B}$-measurable function from a probability space $(\mathcal{Y},\mathcal{C},Q)$ into $\mathcal{X}$. The probability distribution $P$ of $X$ on $(\mathcal{X},\mathcal{B})$ is defined as the standard pushforward measure, $P(A) \coloneqq Q(X \in B) = Q\big(X^{-1}(B)\big), \forall B \in \mathcal{B}$.   

Statistical inference for a distribution $P$ is often focused on the estimation of a location of the distribution, and measures of variability about this location. In Euclidean space $\mathbb{R}^n$, the mean of a random variable provides one of the most basic notions of average location or central tendency. In $\mathbb{R}^n$, the integral $\int X dP$ that defines the mean of $X$ depends heavily on the vector space structure of $\mathbb{R}^n$. For example, if $X = \sum_{i = 1}^k x_i I_{A_i}$ is a simple function then $\int X dP = \sum_{i = 1}^k P(A_i)x_i$. This later sum only makes sense because scalar multiplication by the $P(A_i)$'s and vector addition is defined in $\mathbb{R}^n$. When dealing with metric space valued random objects it is no longer possible to define such integrals in general, so a different formulation of measure of central tendency is needed. 

Fr\'echet \cite{Frechet} proposed a generalization of a Euclidean mean that applies to arbitrary metric spaces. The idea is that a mean of $X$ should be the collection of points in $\mathcal{X}$ that are on average the closest to $X$. For $c \geq 1$, the $c$-Fr\'echet mean of $X$, $E_cX$, is defined in terms of the following variational problem:
\begin{align}
    E_cX \coloneqq \underset{x \in \mathcal{X}}{\argm} \; E\big(d(x,X)^c\big).
    \label{frecmndef}
\end{align}
When $\mathcal{X} = \mathbb{R}^n$ with the Euclidean metric, $E_2X$ coincides with the usual Euclidean mean while $E_1X$ is the set of medians of $X$. The existence and uniqueness of the solutions to \eqref{frecmndef} is not guaranteed, so that $E_cX$ is set-valued in general and can even be the empty set. This behaviour is not unfamiliar, as Euclidean medians are not always unique. A simple example of the non-existence of a $2$-Fr\'echet mean is when $X \sim N(0,1)$ on the space $\mathbb{R}/\{0\}$. 

If $E_cX$ is to be meaningful we require that $E\big(d(x,X)^c\big) < \infty$ for at least one $x \in \mathcal{X}$. By the triangle inequality, $d(x,X) \leq d(x,x_0) + d(x_0,X)$, which implies that $E\big(d(x,X)^c\big) < \infty$ for all $x \in \mathcal{X}$. We say that $X \in \mathcal{L}^c(\mathcal{X})$ if $E\big(d(x,X)^c\big) < \infty$ for all $x \in \mathcal{X}$. It should be remarked that this is slightly different than the situation in Euclidean space since a Euclidean mean $E(X)$ exists and is finite as long as $E(\vert X \vert ) < \infty$ or equivalently $E(\vert X - x \vert) < \infty, \forall x \in \mathbb{R}^n$. There is a more general definition of a Fr\'echet mean that accounts for this minor discrepancy, although we do not have any need for this extra generality \cite{Sturm}.

Having defined a mean, it is useful to have a measure describing the spread of $X$ about this mean. The $c$-Fr\'echet variance captures the average $c$-distance of $X$ from its corresponding $c$-Fr\'echet mean. The $c$-Fr\'echet variance of $X$, $V_cX$, is defined as  
\begin{align}
    V_cX \coloneqq \underset{x \in \mathcal{X}}{\inf} E\big(d(x,X)^c\big).
\end{align}
This quantity is always a non-negative real number for $X \in \mathcal{L}^c(\mathcal{X})$. If $X \in \mathbb{R}^n$ with covariance matrix $\Sigma$ then the $2$-Fr\'echet variance of $X$ is $\text{tr}(\Sigma)$, which is the sum of the variances of each component of $X$. As seen from this example, Fr\'echet variances do not capture any information about how the spread of $X$ varies in different  ``directions" in the metric space. Fr\'echet variances only summarize the average squared distance of a random object from its Fr\'echet mean set.

Throughout the remainder of this article we will be primarily concerned with $E_2X$ and $V_2X$ which we shall refer to as the Fr\'echet mean and variance of $X$. If $X$ has distribution $P$ then the notation $E_cP \coloneqq E_cX$ and $V_cP \coloneqq V_cX$ will be used.

\subsection{Hadamard Spaces}
\label{Sec2.2}
A geodesic curve in a metric space $(\mathcal{X},d)$ is a generalization of a straight line segment in $\mathbb{R}^n$. The curve $\gamma:[a,b] \rightarrow \mathcal{X}$, where $-\infty < a < b < \infty$, is a speed $v$ geodesic if $d\big(\gamma(t_1),\gamma(t_0)\big) = v\vert t_1 - t_0 \vert$ for all  $a \leq t_1,t_0 \leq b$. This definition amounts to requiring that the points on the curve $\gamma$ look exactly the same as the points on a corresponding interval in $\mathbb{R}$ with respect to the metric. Thus the map $f:vI \rightarrow \gamma(I)$ defined by $f(s) = \gamma(s/v)$ where $vI = \{vt: t \in I\}$ is an isometry. The length of a curve $\sigma:[a,b] \rightarrow \mathcal{X}$ is defined by $\ell(\sigma) = \sup_{a = x_0 \leq \cdots \leq x_k = b}\sum_{i = 1}^kd\big(\sigma(x_{i}),\sigma(x_{i-1})\big)$ where the supremum is over any finite partition $(x_0,\ldots,x_k)$ of the interval $[a,b]$. The triangle inequality shows that $\sum_{i= 1}^k d\big(\sigma(x_i),\sigma(x_{i-1})\big) \geq d\big(\sigma(b),\sigma(a)\big)$ for any such partition so that $\ell(\sigma) \geq d(\sigma(a),\sigma(b))$. If $\gamma:[c,d] \rightarrow \mathcal{X}$ is a geodesic then $\ell(\gamma) = d\big(\gamma(c),\gamma(d)\big)
$ which shows that for any other curve $\sigma:[a,b] \rightarrow \mathcal{X}$ with $\sigma(a) = \gamma(c)$ and $\sigma(b) = \gamma(d)$ the length of $\gamma$ is no larger than the length of $\sigma$, $\ell(\sigma) \geq \ell(\gamma)$.

A metric space $(\mathcal{X},d)$, is defined to be a geodesic space if for all $x_1,x_0 \in \mathcal{X}$ there exists a geodesic $\gamma:[a,b] \rightarrow \mathcal{X}$ with endpoints, $\gamma(a) = x_0,\gamma(b)  = x_1$. The metric space $\mathcal{X}$ is uniquely geodesic if it is geodesic and any two geodesics $\gamma,\sigma:[a,b] \rightarrow \mathcal{X}$, with $\gamma(a) = \sigma(a),\gamma(b) = \sigma(b)$ are equal \cite{Bridson}. In a uniquely geodesic space where $\gamma:[0,1] \rightarrow \mathcal{X}$ is a geodesic with $\gamma(0) = x$ and $\gamma(1) = y$, the notation $[x,y]_t$ for $t \in [0,1]$ will be used to represent the point $\gamma(t)$. The interpretation of $[x,y]_t$ is that this is the point obtained when travelling $t$ percent of the way along the geodesic that connects $x$ to $y$. Similarly, the expression $[x,y]$  represents the image in $\mathcal{X}$ of the geodesic between $x$ and $y$. 

In a normed vector space $(V,\Vert \cdot \Vert)$, line segments are geodesic in the sense defined above. To see this, if $\gamma:[a,b] \rightarrow V$ is the line segment $\gamma(t) = v_1t + v_0$, then $\Vert \gamma(t_1) - \gamma(t_0) \Vert = \Vert v_1 \Vert \vert t_1 - t_0 \vert$, implying that $\gamma$ is a speed $\Vert v_1 \Vert$ geodesic. Any normed vector space is thus geodesic but may not be uniquely geodesic. In the case where $V$ is an inner product space, $V$ is uniquely geodesic. On a sphere, geodesics are the minor arcs of great circles, which are the shortest paths that connect points on a sphere. The sphere is geodesic but not uniquely geodesic because any two antipodal points can be joined by infinitely many geodesics. It is worth noting that in a Riemannian manifold geodesics are more commonly defined as critical points of the Riemannian length functional. The definition of a geodesic presented here requires that a geodesic be a minimizer of the length functional and so it is more restrictive than the usual definition if $\mathcal{X}$ is a Riemannian manifold.

The curvature of a uniquely geodesic metric space is primarily described in terms of the geometric properties of generalized triangles in the space. Given three points $x,y,z \in \mathcal{X}$ the triangle $\Delta xyz \subset \mathcal{X}$ is defined as the set of points $[x,y] \cup [y,z] \cup [z,x]$. Due to the triangle inequality, given the numbers $d(x,y),d(y,z)d(z,x)$, there exist points $\Tilde{x},\Tilde{y},\Tilde{z}$ in $\mathbb{R}^2$ such that the triangle $\Delta\Tilde{x}\Tilde{y}\Tilde{z} \subset \mathbb{R}^2$ has side lengths $d(x,y),d(y,z)$ and $d(z,x)$. The Alexandrov curvature \cite{Alexandrov} of a metric space compares how the distance from  $[x,y]_t$ to $z$ in $\mathcal{X}$ differs from the distance from $(1-t) \Tilde{x} + t \Tilde{y}$ to $\Tilde{z}$ in $\mathbb{R}^2$ for $t \in [0,1]$. A metric space has negative Alexandrov curvature if $d([x,y]_t,z)$ is no greater than $d([\Tilde{x},\Tilde{y}]_t,\Tilde{z})$ for all $x,y,z \in \mathcal{X}$ while being less than $d([\Tilde{x},\Tilde{y}]_t,\Tilde{z})$ for at least some triplet of points $x,y,z \in \mathcal{X}$ \cite{Bridson}. Positive Alexandrov curvature is defined similarly, while a space with zero Alexandrov curvature has $d([x,y]_t,z) = d([\Tilde{x},\Tilde{y}]_t,\Tilde{z})$ for all $x,y,z \in \mathcal{X}$. These requirements can be visualized as positively curved spaces having triangles with edges that bend outwards and negatively curved spaces having triangles with edges that bend inwards, relative to triangles in $\mathbb{R}^2$. See Figure \ref{comptriangle} for typical examples of generalized triangles in positively and negatively curved spaces. The generalized triangles in Figure \ref{comptriangle} are isometrically embedded in $\mathbb{R}^2$ so that all distances between points are given by Euclidean distance.

A metric space with non-positive curvature satisfies the CAT(0) curvature bound $d([x,y]_t,z) \leq d([\Tilde{x},\Tilde{y}]_t,\Tilde{z})$ for all $x,y,z \in \mathcal{X}$. After expanding $d([\Tilde{x},\Tilde{y}]_t,\Tilde{z})$  in terms of the side lengths of the triangle $\Delta \Tilde{x}\Tilde{y}\Tilde{z}$, the CAT(0) bound is equivalent to
\begin{IEEEeqnarray}{c}
            d([x,y]_t,z)^2 \leq (1-t)d(x,z)^2 + td(y,z)^2 - t(1-t)d(x,y)^2
    \label{cat0ineq}
    \IEEEeqnarraynumspace
    \IEEEeqnarraynumspace
\end{IEEEeqnarray}
for all $x,y,z \in \mathcal{X}$ and $t \in [0,1]$. Hadamard spaces are defined to be \textit{complete}, \textit{uniquely geodesic}, metric spaces that satisfy the \textit{non-positive} or CAT(0) \textit{curvature bound} in \eqref{cat0ineq}.

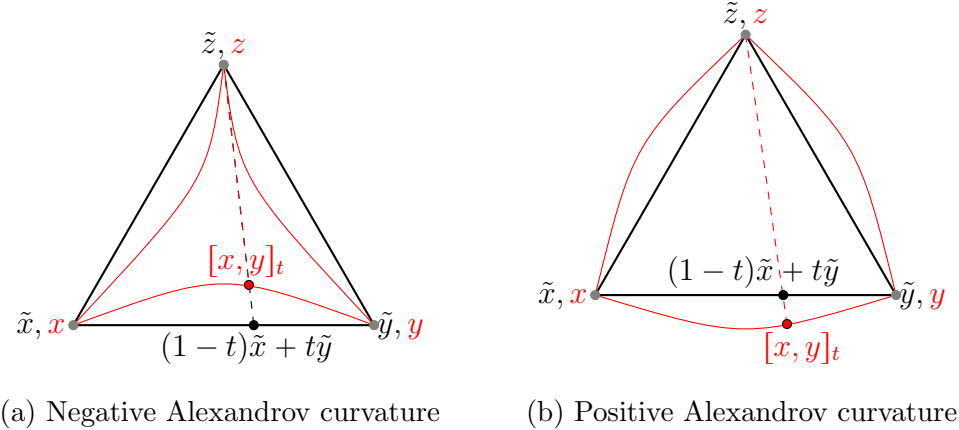
\begin{figure}[h]
    \centering
    \begin{subfigure}[t]{.45\textwidth}
    \centering
    \begin{tikzpicture}[scale = 1]
\draw[thick] (0,0) -- (4,0);
\draw[thick] (0,0) -- (2,{sqrt(12)});
\draw[thick] (2,{sqrt(12)}) -- (4,0);
\draw[red] (0,0) ..controls (2,{sqrt(12)/2 - 1}) .. (4,0);
\draw[red] (0,0) .. controls ({2 - 0.2},{sqrt(12)/2 + 0.2}) .. (2,{sqrt(12)});
\draw[red] (2,{sqrt(12)}) .. controls ({2 + 0.2},{sqrt(12)/2 + 0.2}) .. (4,0);
\draw (-0.43,0) node{$\Tilde{x},\textcolor{red}{x}$};
\draw (4.35,0) node{$\Tilde{y},\textcolor{red}{y}$};
\draw (2,{sqrt(12) + .25}) node{$\Tilde{z},\textcolor{red}{z}$};
\draw[black,dashed] (2,{sqrt(12)}) -- (2.4,0);
\draw[red,dashed] (2,{sqrt(12)}) -- (2.335,.535);
\draw (2.3,-0.3) node{$(1-t)\Tilde{x} + t\Tilde{y}$};
\draw (2.3,.85) node{$\textcolor{red}{[x,y]_t}$};
\draw[fill = black] (2.4,0) circle (.6mm);
\draw[fill = red] (2.335,.535) circle (.6mm);
\draw[fill,gray] (0,0) circle (.6mm);
\draw[fill,gray] (2,{sqrt(12)}) circle (.6mm);
\draw[fill,gray] (4,0) circle (.6mm);
\end{tikzpicture}
\caption{Negative Alexandrov curvature}
\end{subfigure}
\begin{subfigure}[t]{.45\textwidth}
\centering        \begin{tikzpicture}[scale = 1]
\draw[thick] (0,0) -- (4,0);
\draw[thick] (0,0) -- (2,{sqrt(12)});
\draw[thick] (2,{sqrt(12)}) -- (4,0);
\draw[red] (0,0) ..controls (2,{-.6}) .. (4,0);
\draw[red] (0,0) .. controls ({1 - 0.5},{sqrt(12)/2 + 0.3}) .. (2,{sqrt(12)});
\draw[red] (2,{sqrt(12)}) .. controls ({3 + 0.5},{sqrt(12)/2 + 0.3}) .. (4,0);
\draw (-0.43,0) node{$\Tilde{x},\textcolor{red}{x}$};
\draw (4.35,0) node{$\Tilde{y},\textcolor{red}{y}$};
\draw (2,{sqrt(12) + .25}) node{$\Tilde{z},\textcolor{red}{z}$};
\draw[red,dashed] (2,{sqrt(12)}) -- (2.55,-0.385);
\draw (2.1,0.3) node{$(1-t)\Tilde{x} + t\Tilde{y}$};
\draw (2.75,-.7) node{$\textcolor{red}{[x,y]_t}$};
\draw[fill = black] (2.5,0) circle (.6mm);
\draw[fill = red] (2.55,-.385) circle (.6mm);
\draw[fill,gray] (0,0) circle (.6mm);
\draw[fill,gray] (2,{sqrt(12)}) circle (.6mm);
\draw[fill,gray] (4,0) circle (.6mm);
\end{tikzpicture}
\caption{Positive Alexandrov curvature}
\end{subfigure}
\caption{Metric space comparison triangles, \textcolor{red}{$\Delta xyz$} and $\Delta \Tilde{x}\Tilde{y}\Tilde{z}$}
\label{comptriangle}
\end{figure}

The subset of Hadamard spaces that have zero Alexandrov curvature so that \eqref{cat0ineq} holds with equality are geometrically similar to $\mathbb{R}^n$. In this case the triangle $\Delta xyz$ is indistinguishable from its comparison triangle $\Delta \Tilde{x}\Tilde{y}\Tilde{z}$ and thus Euclidean trigonometry will apply to $\Delta xyz$. For example, a version of the Euclidean law of sines or cosines will hold in such a space, and suitably defined interior angles of $\Delta xyz$ will also sum to $\pi$. Any Hilbert space or closed, convex subset thereof is a zero curvature Hadamard space. Consequently, any results that hold for Hadamard spaces will also hold for Hilbert spaces, which is the setting of much of classical statistics.  

The definition of Alexandrov curvature is motivated in part as a generalization of the sectional curvature of a Riemannian manifold. As such, any complete Riemannian manifold with non-positive sectional curvature is a Hadamard space. For example, the saddle surface in $\mathbb{R}^3$ has negative sectional curvature and is a Hadamard space with non-zero curvature. If one draws a triangle of shortest paths on such a surface it will look like the comparison triangle in Figure  \ref{comptriangle}. Another easily visualized example of a Hadamard space with non-zero curvature is a metric tree. Metric trees are weighted graphs that are trees endowed with the shortest path metric. Section \ref{Sec5} goes into further detail about metric tree spaces.  

In a Hilbert space, any closed and convex set $\mathcal{C}$ has the property that there exists a unique projection of any point $x$ onto $\mathcal{C}$ that minimizes the squared distance of $x$ from $\mathcal{C}$. If $\mathcal{C}$ is a closed linear subspace then this follows from the Pythagorean theorem. This result can be generalized to Hadamard spaces as follows: A set $\mathcal{C}$ in a geodesic space is said to  be convex if for all $x,y \in \mathcal{C}$ we have that $[x,y] \subset \mathcal{C}$. The Hadamard space projection theorem of \cite{Bacak} says that for any point $x \in \mathcal{X}$ and closed and convex subset $\mathcal{C}$ of a Hadamard space there exists a unique point $\Pi(x) \in \mathcal{C}$ that satisfies $d\big(x,\Pi(x)\big)^2 = \inf_{y \in \mathcal{C}}d(x,y)^2$. In addition, $\Pi(x)$ satisfies the inequality
\begin{align}
    d(x,z)^2 \geq d\big(z,\Pi(x)\big)^2 + d\big(\Pi(x),x\big)^2, \;\; \forall z \in \mathcal{C}.
     \label{metricproj}
\end{align}
The inequality in \eqref{metricproj} provides a bound on how close $\Pi(x)$ is to $x$ relative to any other point $z \in \mathcal{C}$. When $\mathcal{C}$ is a closed vector subspace of a Hilbert space, \eqref{metricproj} holds with equality and is the Pythagorean theorem. 

We will now show that a Hadamard space $L^2(\mathcal{X})$ of random objects on $\mathcal{X}$ can be constructed in a manner that is analogous to the construction of $L^2(\mathbb{R}^n)$ from $\mathbb{R}^n$. In $L^2(\mathcal{X})$ the inequality \eqref{metricproj} can be applied to obtain a bias-variance inequality. Let $X,Y$ be random objects in $\mathcal{L}^2(\mathcal{X})$. We define a pseudo-metric $\rho$, on $\mathcal{L}^2(\mathcal{X})$ by taking $\rho(X,Y) \coloneqq E\big(d(X,Y)^2\big)^{1/2}$. The space $L^2(\mathcal{X})$ is the set of equivalence classes of random objects in $\mathcal{L}^2(\mathcal{X})$ that are equal almost everywhere, so that $X \sim X'$ if and only if $E\big(d(X,X')^2\big) = 0$. The metric space $\big(L^2(\mathcal{X}\big),\rho)$ is a Hadamard space with geodesics given by $[X,Y]_t(\omega) = [X(\omega),Y(\omega)]_t$. The CAT(0) bound follows by the linearity of expectations while completeness follows in the same way that completeness of $L^2(\mathbb{R}^n)$ follows from the completeness of $\mathbb{R}^n$ \cite{Bacak}. 

The collection of constant almost everywhere random objects $\mathcal{C} \coloneqq \{\theta \in \mathcal{X}\} \subset L^2(\mathcal{X})$ is a closed and convex set in $L^2(\mathcal{X})$. Noticing that ${E_2X = \inf_{\theta \in \mathcal{C}}\rho(X,\theta)^2}$, the projection theorem implies that the Fr\'echet mean $E_2X = \Pi(X)$ exists and is unique \cite{Sturm}. The inequality in \eqref{metricproj} becomes
\begin{align}
    E\big(d(X,\theta)^2\big) \geq d(\theta,E_2X)^2 + E\big(d(E_2X,X)^2\big), \;\; \forall \theta \in \mathcal{X},
    \label{hadbiasvar}
\end{align}
which can be viewed as a bias-variance inequality as follows: If $X$ is used as an estimator for $\theta$ under the loss function $L(\theta,\cdot) = d(\theta,\cdot)^2$ then the term $E\big(d(E_2X,X)^2\big)$ is exactly the Fr\'echet variance $V_2X$ while $d(\theta,E_2X)^2$ can be viewed as the squared bias of $X$. 

Conditional expectations of random objects in a Hadamard space can be defined in a similar manner. Recall that for a $\sigma$-algebra $\mathcal{G} \subset \mathcal{B}$ the conditional expectation of $X \in L^2(\mathbb{R}^n)$ is the projection of $X$ onto the closed vector subspace of $\mathcal{G}$-measurable random variables in $L^2(\mathbb{R}^n)$. Likewise, taking $\mathcal{C} \coloneqq \{Y \in L^2(\mathcal{X}):\sigma(Y) \subset \mathcal{G}\}$ to be the $\mathcal{G}$-measurable random objects in $L^2(\mathcal{X})$, the conditional expectation $E_2(X|\mathcal{G})$, as defined in \cite{Bacak}, is given by 
\begin{align}
 E_2(X|\mathcal{G}) \coloneqq \underset{Y \in \; \mathcal{C}}{\argm} \; E\big(d(X,Y)^2\big). 
 \label{condexpdefinition}
\end{align}
As the set $\mathcal{C}$ is closed and convex $E_2(X|\mathcal{G})$ exists, is unique, and satisfies a version of \eqref{metricproj}. As we will see in the next section, the lack of a vector space structure in $\mathcal{C}$ implies that not all of the familiar properties of Euclidean conditional expectations carry over to Hadamard spaces. 

\section{Estimation of Fr\'echet Means}
\label{Sec3}
We start by considering a general Hadamard space point estimation problem. Let $\mathcal{P} \subset L^2(\mathcal{X})$ be a family of distributions on $\mathcal{X}$ and $g: \mathcal{P} \rightarrow \mathcal{X}$ be a functional defined on $\mathcal{P}$ that is an estimand of interest. For example, $g$ could be the Fr\'echet mean functional $g(P) = E_2P$. Given a single observation of $X \sim P \in \mathcal{P}$, we seek to estimate $g(P)$ under squared distance loss $L\big(g(P),\delta(X)\big) \coloneqq d\big(g(P),\delta(X)\big)^2$ with the corresponding risk function $R(P,\delta) \coloneqq E\big[L\big(g(P),\delta\big)\big]$. 

A function $f: \mathcal{X} \rightarrow \mathbb{R}$ is said to be metrically convex if $f([x,y]_t)$ is convex as a function of $t \in [0,1]$ for any choice of $x,y \in \mathcal{X}$ \cite{Bacak}. The CAT(0) inequality \eqref{cat0ineq} shows that the function $f_z(x) = d(z,x)^2$ is metrically convex for all $z \in \mathcal{X}$. The loss function $d\big(g(P),\delta\big)^2$ is thus metrically convex. This convexity yields behaviour similar to that of convex functions defined on $\mathbb{R}^n$. For instance, a Fr\'echet mean version of Jensen's inequality, $E(d(X,\theta)^2) \geq d(E_2X,\theta)^2$, is an immediate result of \eqref{hadbiasvar}. The metric convexity of the squared distance function is also the key property that allows for the favorable use of the shrinkage estimators considered in the next section.      

First we show that the class of non-randomized estimators of any functional under squared distance loss forms an essentially complete class, an attribute that holds for all convex loss functions on $\mathbb{R}^n$. That is, for any randomized estimator $\delta(X,U)$ of $g(P)$ with $U \sim \text{Unif}[0,1]$ independently of $X$, there exists a non-randomized estimator $\Tilde{\delta}(X)$ that satisfies $R(P,\tilde{\delta}) \leq R(P,\delta)$ for all $P \in \mathcal{P}$. Given a randomized estimator $\delta(X,U)$, take $\tilde{\delta}(X) = E_2\big(\delta(X,U)|\sigma(X)\big)$ as defined in \eqref{condexpdefinition}. Applying the inequality in \eqref{metricproj} yields 
\begin{align}
    R(P,\delta) =  E\big[d\big(\delta,g(P)\big)^2\big] \geq E\big(d(\delta,\tilde{\delta})^2\big) + E\big[d\big(\tilde{\delta},g(P)\big)^2\big] \geq R(P,\tilde{\delta}),
\end{align}
which proves the result. Note that in order for $E_2\big(\delta(X,U)|\sigma(X)\big)$ to be a function of $X$ the metric space $\mathcal{X}$ must be separable \cite{Dudley}.

A version of the Rao-Blackwell theorem can be extended to this setting by similar reasoning. Suppose that a $\sigma$-algebra $\mathcal{G}$ has the property that $E_2\big(\delta(X)|\mathcal{G}\big) = E_2\big(\delta(Y)|\mathcal{G}\big)$ when $X \sim P$ and $Y \sim Q$ for all $ P,Q \in \mathcal{P}$, so that the random object $\tilde{\delta} =  E_2\big(\delta(X)|\mathcal{G}\big)$ is independent of $P \in \mathcal{P}$. Further suppose that a version of $E_2(\delta|\mathcal{G})(\omega)$ can be realized as a function of $X(\omega)$, so that $\tilde{\delta}(X(\omega)) \coloneqq E_2\big(\delta(X)|\mathcal{G}\big)(\omega)$ is an estimator. This second assumption holds in the typical scenario where $\mathcal{X}$ is separable and $\mathcal{G} = \sigma\big(f(X)\big)$ for some measurable function $f$ of $X$. This conditioning on $\mathcal{G}$ reduces the risk of $\delta$ for some $P \in \mathcal{P}$ unless $\tilde{\delta} = \delta, \; a.e \; \mathcal{P}$. It should be noted that the standard definition of sufficiency of $\mathcal{G}$, requiring that $P(A|\mathcal{G}) = Q(A|\mathcal{G})$ for all $ P,Q \in \mathcal{P}$, does not immediately imply that $\tilde{\delta} = E_2\big(\delta|\mathcal{G}\big)$ is independent of the choice of $P$. The reason is that in the case of a Euclidean valued $\delta(X)$, conditional expectations can be approximated by conditional probabilities using the dominated convergence theorem for conditional expectations. The relationship between conditional expectations and conditional probabilities is more complex in the variational formulation of the metric conditional expectation in \eqref{condexpdefinition}. 

From the Rao-Blackwell theorem the Lehmann-Scheff\'e theorem is easily obtained in a Euclidean setting by taking the conditional expectation of an unbiased estimator with respect to a complete sufficient statistic. A metric space point estimator $\delta(X)$ of $g(P)$ is said to be unbiased for the family $\mathcal{P}$ if $E_2\big(\delta(X)\big) = g(P)$ when $X \sim P$ for all $P \in \mathcal{P}$. The main obstacle towards extending Lehmann-Scheff\'e to a metric space case is that the tower rule does not hold in general for conditional Fr\'echet means: If $\mathcal{G} \subset \mathcal{H}$ then it will not always be the case that $E_2\big(E_2\big(\delta|\mathcal{H})|\mathcal{G}\big) = E_2(\delta|\mathcal{G})$ \cite{SturmMartingale}. The reason for this is that $\mathcal{L}^2(\mathcal{G})$ and $\mathcal{L}^2(\mathcal{H})$ do not inherit any Hilbert space structure from $\mathcal{X}$ as they do in the Euclidean case.  The Pythagorean theorem applied to nested vector subspaces cannot in general be applied to $\mathcal{L}^2(\mathcal{H}) \subset \mathcal{L}^2(\mathcal{G})$. It follows that if $\delta$ is unbiased for $g(P)$ then there is no guarantee that $E_2(\delta|\mathcal{G})$ will remain unbiased for $g(P)$. See Appendix \ref{appB} for an explicit example of this phenomenon.    

The problem we will consider for the remainder of this work is the estimation of a Fr\'echet mean, $g(P) = E_2P$, under squared distance loss. Due to the generality of Hadamard spaces we will work with non-parametric families of distributions that only make mild assumptions on the Fr\'echet means and variances of random objects. Parametric alternatives do exist, most notably the Riemannian normal distributions on a Riemannian manifold introduced by Pennec \cite{PennecGauss}. The Riemannian normal distribution can however be challenging to work with as its Fr\'echet variance is in general related in a complex, non-linear way to the scale parameter of the distribution and may even depend on the Fr\'echet mean.  

When working with a large non-parametric family of distributions there may not be many estimators that are unbiased for the entire family. This next result shows that in an unbounded Hadamard space the only unbiased estimator for the family of distributions with a fixed Fr\'echet variance is  $\delta(X) = 
X$. 
\begin{theorem}
\label{umvuthm}
If $\mathcal{X}$ is a Hadamard space with infinite diameter then the unique unbiased estimator of $E_2P$ for the family $\mathcal{P} = \{P: V_2P = \sigma^2\}$ is $\delta(X) = X$. 
\end{theorem}
\begin{proof}
Suppose that $\delta(X)$ is an unbiased estimator for $\mathcal{P}$. For any $x,y \in \mathcal{X}$ let  $P_{xyq}, \; q \in [0,1],$ be the Bernoulli distribution on $\mathcal{X}$ with $P_{xyq}(\{x\}) = q$, $ P_{xyq}(\{y\}) = 1-q$. Fix $x$ and choose a sequence of $y_k$ such that $d(x,y_k) \geq \sqrt{k}$. Such a sequence exists as $\text{diam}(\mathcal{X}) = \infty$. Without loss of generality we can assume that $d(x,y_k) = \sqrt{k}$ since we have that $d(x,[x,y_k]_{\sqrt{k}/d(x,y_k)}) = \sqrt{k}$. A straightforward calculation shows that $E_2P_{xy_kq} = [x,y_k]_{1-q}$ and thus $V_2(P_{xy_kq}) = d(x,y_k)^2q(1-q) = kq(1-q)$. For $k \geq 4\sigma^2$ there exists a $1/2 \leq q_k \leq 1$ such that $V_2(P_{xy_kq_k}) = \sigma^2$. Thus $P_{xy_kq_k} \in \mathcal{P}$ for $k$ large enough, with $q_k \uparrow 1$ as $k \rightarrow \infty$. Now if $X \sim P_{xy_kq_k}$ then $\delta(X) \sim P_{\delta(x)\delta(y_k)q_k}$ so that  $E_2\delta(X) = [\delta(x),\delta(y_k)]_{1-q_k}$. As $\delta(X)$ is unbiased for $E_2X$ we have $[x,y_k]_{1-q_k} = [\delta(x),\delta(y_k)]_{1-q_k}, \forall k$. Taking limits of both sides of this equation gives $x = \lim_{k \rightarrow \infty}[x,y_k]_{1-q_k} = \lim_{k \rightarrow \infty}[\delta(x),\delta(y_k)]_{1-q_k} = \delta(x)$, proving that $\delta(x) = x$ for an arbitrary $x \in \mathcal{X}$.
\end{proof}
We remark that a uniformly minimum Fr\'echet variance unbiased estimator may not minimize the squared distance risk out of the collection of all unbiased estimators. This is due to the Fr\'echet variance and bias only providing a lower bound on the risk in \eqref{hadbiasvar}.  

A different technique for determining properties of estimators in metric spaces is to restrict distributions on $\mathcal{X}$ to subsets of $\mathcal{X}$ that are isometric to Euclidean space and then apply known results for Euclidean spaces. A geodesic line \cite{Bridson} is defined to be a function $\gamma:\mathbb{R} \rightarrow \mathcal{X}$ such that the restriction $\gamma \vert_{[a,b]}$ is a speed $v$ geodesic for any $a<b \in \mathbb{R}$. Geodesic lines look exactly like copies of $\mathbb{R}$ that are embedded in $\mathcal{X}$. 
Using the known result that $\delta(X) = X$ is a minimax estimator for the mean of a normal distribution \cite{Lehmann}, we get the following theorem.
\begin{theorem}
\label{minimaxthm}
If the Hadamard space $\mathcal{X}$ has the property that there exists a geodesic line in $\mathcal{X}$, then $X$ is a minimax estimator of $E_2X$ for the family $\mathcal{P} = \{P:V_2P = \sigma^2\}$. 
\end{theorem}
\begin{proof}
Let $\gamma:\mathbb{R} \rightarrow \mathcal{X}$ be a geodesic line parameterized to have unit speed. Consider the sub-family of distributions $\mathcal{P}^* \subset \mathcal{P}$ where $\mathcal{P}^* \coloneqq \{P_{\theta}: X \sim P_\theta, \; X = \gamma(Y), \; Y \sim N(\theta,\sigma^2), \; \theta \in \mathbb{R}\}.$ That is, the distribution of $P_{\theta}$ is concentrated on the geodesic line $\gamma$ and has a normally distributed coordinate on this geodesic. Take $\Pi:\mathcal{X} \rightarrow \gamma(\mathbb{R})$ to be the projection of points in $\mathcal{X}$ onto the closed and convex set that is the image of $\gamma$ in $\mathcal{X}$, as defined in \eqref{metricproj}. It follows by the projection theorem that for any point $z \in \mathcal{X}/\{\gamma(\mathbb{R})\}$ and $X \in \gamma(\mathbb{R})$ we have $d(X,z)^2 > d\big(X,\Pi(z)\big)^2$. The Fr\'echet mean of $X$ is therefore contained in the image $\gamma(\mathbb{R})$ and must equal $\gamma(\theta)$. Similarly, for any $\gamma(\theta)$ and estimator $\delta(X)$ of $\gamma(\theta)$, the projection theorem implies that ${E\big[d\big(\delta(X),\gamma(\theta)\big)^2\big] \geq E\big[d\big(\Pi(\delta(X)),\gamma(\theta)\big)^2\big]}$ with equality holding if and only if $\delta(X) \in \gamma(\mathbb{R})$ almost surely. This shows that if $\delta(X)$ is an admissible estimator of $E_2X = \gamma(\theta)$ for the sub-family $\mathcal{P}^*$ then $\delta(X) \in \gamma(\mathbb{R})$ almost surely. Along $\gamma$, $d\big(\gamma(t_1),\gamma(t_0)\big) = \vert t_1 - t_0 \vert$, so that  $L(\theta,\delta(X)) = \big(\theta - \gamma^{-1}\big(\delta(X)\big)\big)^2$ for any estimator $\delta(X)$ whose support is contained in the image of $\gamma$. The decision problem of finding a minimax estimator of $E_2X = \gamma(\theta)$ for the sub-family $\mathcal{P}^*$ is equivalent to the problem of finding a minimax estimator of $\theta$ under squared error loss given a sample $Y = \gamma^{-1}(X)$ from the family $\{N(\theta,\sigma^2): \theta \in \mathbb{R}\}$.  The estimator $Y$ is minimax for this normal problem from which it follows that $X$ must be minimax for the sub-family $\mathcal{P}^*$. As $\sup_{P \in \mathcal{P}^*}R(P,X) = \sigma^2 = \sup_{P \in \mathcal{P}}R(P,X)$,  $X$ is minimax for $\mathcal{P}$ \cite{Lehmann}.
\end{proof}
Theorem \ref{minimaxthm} also applies to families of the form $\mathcal{P} = \{P: \sigma_0^2 \leq V_2P \leq \sigma_1^2\}$ because $\sup_{P \in \mathcal{P}^*}R(P,X)= \sup_{P \in \mathcal{P}}R(P,X)$ for such a family. 

In both Theorems \ref{umvuthm} and \ref{minimaxthm} the unboundedness of $\mathcal{X}$ plays a necessary role in ensuring that $X$ is UMVU and minimax respectively. In a bounded metric space there may be some points in the metric space that cannot be the Fr\'echet mean of a distribution with Fr\'echet variance $\sigma^2$. For example, a trivial case of this is where $\mathcal{X} = [0,1]$ and $\sigma^2 = 1/4$. The only distribution with Fr\'echet variance $1/4$ on $\mathcal{X}$ is a Bernoulli($1/2$) distribution. The only possible Fr\'echet mean for $\mathcal{P} = \{P:V_2P = 1/4\}$ is then $1/2$. The estimator $\delta(X) = X$ is unacceptable in such a situation as it has the highest possible risk out of any estimator that could be used. Even if $\sigma^2$ is chosen to be less than $1/2$ the same issue occurs as points that are close to $\{0\}$ and $\{1\}$ cannot be Fr\'echet means of any distribution with variance $\sigma^2$. For instance, $\{0\}$ and $\{1\}$ can only be Fr\'echet means of degenerate point mass distributions. As a result, it is possible for $X$ to be an inadmissible estimator of the Fr\'echet mean for the family $\mathcal{P} = \{P:V_2P = \sigma^2\}$ in a bounded space or an unbounded space that does not contain a geodesic line. 

To resolve this inadmissibility issue it is reasonable to modify $\delta(X) = X$ by projecting it onto the set of points that can be realized as the Fr\'echet mean of a distribution in $\mathcal{P}$ \cite{Marchand}. If it exists, such a projection can be viewed as forcing $X$ into a more favorable region of $\mathcal{X}$. In the next section, shrinkage estimators are examined that push $X$ towards a chosen point in $\mathcal{X}$ that is deemed to be a reasonable initial guess of the Fr\'echet mean. This shrinkage process can be used to partially correct the undesirable behaviour of $\delta(X) = X$ in metric spaces with bounded diameter.

\section{Shrinkage Estimators in Hadamard Spaces}
\label{Sec4}
Suppose that one wishes to estimate the mean of a distribution $P$ on $\mathbb{R}^n$ given one observation $X \sim P$. If it is suspected that $EX$ is close to the point $\psi$ in $\mathbb{R}^n$ then as an alternative to using the estimator $X$ to estimate $EX$ one can instead use the shrinkage estimator $(1-t)X + t\psi = [X,\psi]_t$ for some $t \in [0,1]$. In a Hadamard space the geodesics of the space can be used to define an analogue of this shrinkage estimator. Assume that $X \sim P$ where $V_2P = \sigma^2$ is known, $E_2P = \theta$, and a squared distance loss function is used.  Given a shrinkage point $\psi \in \mathcal{X}$, the estimator $[X,\psi]_t,\; t \in [0,1],$ can be used to estimate the Fr\'echet mean $\theta$.

Even in the absence of strong prior information about $E_2X$, shrinkage estimators can be used to reduce the squared distance risk of the estimator $X$. Applying the CAT(0) bound in \eqref{cat0ineq} to the estimator $[X,\psi]_t$ gives
\begin{IEEEeqnarray}{c}
 E\big(d(\theta,[X,\psi]_t)^2\big) \leq t\sigma^2 + (1-t) d(\theta,\psi)^2 - t(1-t)E\big(d(X,\psi)^2\big).
    \label{cat0riskbd}
    \IEEEeqnarraynumspace
    \IEEEeqnarraynumspace
\end{IEEEeqnarray}
The right hand side of \eqref{cat0riskbd} is a convex, quadratic function of $t$. It is seen that if $t$ is chosen small enough, the right hand side of \eqref{cat0riskbd} is less than $\sigma^2$ and for such a $t$, $R(P,[X,\psi]_t) < R(P,X)$. It is the metric convexity of the squared distance function in a Hadamard space that makes shrinkage estimators on these spaces effective. Another manifestation of the metric convexity that motivates the use of shrinkage estimators is the bias-variance decomposition in \eqref{hadbiasvar}. As long as the distribution of $X$ is non-degenerate, $E\big(d(X,\psi)^2\big) > d(E_2X,\psi)^2$ so that $d(X,\psi)^2$ on average overestimates the squared distance of $\psi$ from $E_2X$. To correct this, the estimator $[X,\psi]_t$ is closer to $\psi$ than $X$ is. 

If it is assumed that the point $\psi$ is given, the central question is how should one go about choosing $t$ in $[X,\psi]_t$. The optimal value of $t$ that minimizes the upper bound of the risk in \eqref{cat0riskbd} is
\begin{align}
\label{optimalshrinkwt}
    \tilde{t} \coloneqq \frac{\sigma^2 + \rho(X,\psi)^2 - d(\theta,\psi)^2}{2\rho(X,\psi)^2},
\end{align}
where we use the notation $\rho(X,\psi)^2 = E\big(d(X,\psi)^2\big)$ with $\rho$ being the metric on the Hadamard space $L^2(\mathcal{X})$ defined in Section \ref{Sec2.2}. We call $\tilde{t}$ the oracle shrinkage weight although it only minimizes the risk upper bound, not the risk function. The Hadamard bias-variance inequality \eqref{hadbiasvar} shows that $\rho(X,\psi)^2 - d(\theta,\psi)^2 \geq \sigma^2$ so that $\tilde{t} \geq \sigma^2/\rho(X,\psi)^2$. Using a plug in estimate for $\rho(X,\psi)^2$, the shrinkage weight ${w(X) \coloneqq 1  \wedge \big(\sigma^2/d(X,\psi)^2\big)}$ serves as an estimate of this lower bound for $\tilde{t}$. In order to use this shrinkage weight, the Fr\'echet variance $\sigma^2$ must be a known quantity. As long as $d(X,\psi)^2$ is sufficiently concentrated around $\rho(X,\psi)^2$ then $w(X)$ will tend to underestimate $\tilde{t}$. This reduces the possibility of overshrinking $X$ when using the estimator $[X,\psi]_{w(X)}$. 

\subsection{Geodesic James-Stein Estimator}
Shrinkage estimators are typically used in a setting where observations from different groups are available and information is shared between groups to improve the estimation of group-specific parameters. A multi-group Fr\'echet mean estimation problem is formulated by first supposing that we have random objects $X = (X_1,\ldots,X_n)$ where each $X_i$ lies in the Hadamard space $(\mathcal{X}_i,d_i)$, has Fr\'echet mean $\theta_i$, a known Fr\'echet variance $\sigma_i^2$, and is independent of the other $X_j$'s. The decision problem we consider for the remainder of this article is the simultaneous estimation of the collection of Fr\'echet means $\theta = (\theta_1,\ldots,\theta_n)$ under the loss function 
$L\big(\theta,\delta(X)\big) = \sum_{i = 1}^n d_i\big(\theta_i,\delta_i(X)\big)^2/n$. This problem formulation is the same as the classical James-Stein estimation problem in the special case when $\mathcal{X}_i = \mathbb{R}$ for each $i$ and $X_i \sim N(\theta_i,\sigma^2)$ independently for $i = 1,\ldots,n$. Notice that like the classical James-Stein problem, there is no relationship assumed between the various $\theta_i$'s and  the $X_i$'s are independent and may not even take values in the same Hadamard space.    

The simultaneous point estimation problem can be viewed as estimating a single point in a larger Hadamard space. The product Hadamard space of the Hadamard spaces $(\mathcal{X}_i,d_i)$ is the set $\mathcal{X}^{(n)} \coloneqq \mathcal{X}_1 \times \cdots \times \mathcal{X}_n$ with the metric $d$ given by $d(x,y) \coloneqq \big( \sum_{i  =1}^n d_i(x_i,y_i)^2/n\big)^{1/2}$ \cite{Bacak}, where the multiplicative factor $n^{-1/2}$ is added to ease notation. Geodesics in $(\mathcal{X}^{(n)},d)$ are given pointwise by $[x,y]_t = ([x_1,y_1]_t,\ldots,[x_n,y_n]_t)$, and the CAT(0) inequality follows from the form of $d(x,y)$. The collection of observations $X = (X_1,\ldots,X_n)$ is a random object in $\mathcal{X}^{(n)}$ with Fr\'echet mean $\theta = (\theta_1,\ldots,\theta_n)$. The simultaneous point estimation problem is to estimate $E_2X = \theta$ under the loss function $L\big(\theta,\delta(X)\big) = d\big(\theta,\delta(X)\big)^2$ which is exactly the Fr\'echet mean estimation problem introduced in Section \ref{Sec3}. The added nuance in this problem is that the independence assumption on the $X_i$'s implies that $X$ must follow a product distribution on $\mathcal{X}^{(n)}$.   

By viewing $X$ as an element of the product Hadamard space $\mathcal{X}^{(n)}$, we can form the shrinkage estimator $[X,\psi]_{w(X)}$ introduced at the beginning of this section. We call $\delta_{JS}(X) \coloneqq [X,\psi]_{w(X)}$ the geodesic James-Stein estimator with shrinkage point ${\psi = (\psi_1,\ldots,\psi_n)}$. The Fr\'echet variance of $X$, which we denote by $\sigma^2$, is $E\big(d(X,\theta)^2\big) = \sum_{i = 1}^n E\big(d_i(X_i,\theta_i)^2\big)/n = \sum_{i = 1}^n \sigma_i^2/n$. The components of $\delta_{JS}(X)$ are thus
\begin{IEEEeqnarray}{c}
        \delta_{JS}(X)_j \coloneqq \bigg(1 - \big(1 \wedge \frac{\sum_{i  =1}^n \sigma_i^2}{\sum_{i = 1}^nd_i(X_i,\psi_i)^2}\big)  \bigg)X_j + \bigg( 1 \wedge \frac{\sum_{i  =1}^n \sigma_i^2}{\sum_{i = 1}^nd_i(X_i,\psi_i)^2}  \bigg)\psi_j.
    \label{jscomponents} 
    \IEEEeqnarraynumspace
    \IEEEeqnarraynumspace
\end{IEEEeqnarray}
In Euclidean space, $\mathcal{X}_i = \mathbb{R}$, the positive-part James-Stein estimator $\delta_+(X)$, for $X \sim N_n(\theta,\sigma^2I)$,  is closely related to $\delta_{JS}(X)$ since $\delta_+(X) = [X,\psi]_{1 \wedge \frac{n-2}{n}\sigma^2/d(X,\psi)^2}$. The only difference between $\delta_+(X)$ and $\delta_{JS}(X)$ is the factor $(n-2)/n$ appearing in the shrinkage weight of $\delta_+(X)$. This factor is a remnant of tailoring $\delta_+(X)$ to a Gaussian $X$. 

\subsection{James-Stein Risk Comparison}
The Gaussian James-Stein estimator dominates $X$ in squared error loss as long as the Gaussian distribution takes values in $\mathbb{R}^n$ with $n \geq 3$ \cite{Stein,James}. Similarly, we will be primarily interested in the behaviour of $R(P,\delta_{JS})$ as the dimension $n$ of the Hadamard space $\mathcal{X}^{(n)}$ increases. In typical applications each $X_i$ takes values in the same Hadamard space $\mathcal{X}$, so that $\mathcal{X}_i = \mathcal{X}$ for all $i$ and $\mathcal{X}^{(n)} = \mathcal{X}^n$. To emphasize the dimension $n$ of the Hadamard space $\mathcal{X}^{(n)}$ that $X$, $\theta$ and $\psi$ lie in, we denote these objects by $X^{(n)},\theta^{(n)}$ and $\psi^{(n)}$. Moreover, when examining how $n$ effects the behaviour of $\delta_{JS}$ it is helpful to assume that we have a sequence of random objects $\{X^{(n)}\}_{n = 1}^{\infty}$ with corresponding Fr\'echet means $\{\theta^{(n)}\}_{n = 1}^{\infty}$, as well as a sequence of shrinkage points $\{\psi^{(n)}\}_{n = 1}^{\infty}$. Note that $X^{(k)}$ and $X^{(n)}$ for $k < n$ may be completely unrelated and similarly for $\psi^{(k)}$ and $\psi^{(n)}$.

An upper bound for the loss function of $\delta_{JS}(X^{(n)})$ can be found by plugging in the expression for $[X^{(n)},\psi^{(n)}]_{w(X^{(n)})}$ into the CAT(0) bound, \eqref{cat0ineq}. Defining $A$ to be the set $\{X^{(n)}:\sigma^2 > d(X^{(n)},\psi^{(n)})^2\}$, which is equal to $\{X^{(n)}:w(X^{(n)}) < 1\}$, it is found that
\begin{align}
    L\big(\theta^{(n)},\delta_{JS}(X^{(n)})\big)  \leq &  \; I_A\bigg[\big(1-w(X^{(n)})\big)\big(d(X^{(n)},\theta^{(n)})^2 - \sigma^2\big) + w(X^{(n)})d(\theta^{(n)},\psi^{(n)})^2 -
    \nonumber
    \\
    & \; w(X^{(n)})\big(1-w(X^{(n)})\big)d(X^{(n)},\psi^{(n)})^2\bigg] + 
     I_{A^c}d(\theta^{(n)},\psi^{(n)})^2 \label{lossbound}
     \\
     = & \; \big[I_A (1-w(X^{(n)}))(d(X^{(n)},\theta^{(n)})^2 - \sigma^2)\big] + \nonumber
     \\
     & \; \big[I_Aw(X^{(n)})d(\theta^{(n)},\psi^{(n)})^2\big] + \big[I_{A^c}d(\theta^{(n)},\psi^{(n)})^2\big] \nonumber
     \\
      \coloneqq & \; (a) + (b) + (c).  \nonumber
\end{align}
Notice that the denominator of $I_Aw(X^{(n)})$ cancels with $d(X^{(n)},\psi^{(n)})^2$ so that

\noindent ${I_Aw(X^{(n)})d(X^{(n)},\psi^{(n)})^2 = I_A \sigma^2}$ which makes \eqref{lossbound} take a reasonably simple form. Heuristically, as $n \rightarrow \infty$ by the law of large numbers we expect $d(X^{(n)},\theta^{(n)})^2 - \sigma^2 \rightarrow 0$ and $d(X^{(n)},\psi^{(n)})^2 -  \rho(X^{(n)},\psi^{(n)})^2 \rightarrow 0$. As a result, the term $(a)$ should vanish and since $E(d(X^{(n)},\psi^{(n)})^2) \geq \sigma^2 + d(\theta^{(n)},\psi^{(n)})^2$ it is expected that $I_{A} \rightarrow 1$ so that $(c)$ vanishes. Furthermore, $w(X^{(n)}) = I_A\sigma^2/d(X^{(n)},\psi^{(n)})^2 + I_{A^c} \rightarrow \sigma^2/\rho(X^{(n)},\psi^{(n)})^2$, which yields the approximate risk bound
\begin{align}
    R(P,\delta_{JS}) \lessapprox  \sigma^2 \frac{d(\theta^{(n)},\psi^{(n)})^2}{\rho(X^{(n)},\psi^{(n)})^2} \leq \sigma^2 \frac{d(\theta^{(n)},\psi^{(n)})^2}{d(\theta^{(n)},\psi^{(n)})^2 + \sigma^2} < \sigma^2 = R(P,X^{(n)}), \label{approxriskbdd}
\end{align}
implying that $\delta_{JS}$ has a lower risk than $X^{(n)}$ under squared distance loss.

Regularity conditions on $d(X^{(n)},\theta^{(n)})^2$ and $d(X^{(n)},\psi^{(n)})^2$ are needed to ensure that these quantities are close enough to their respective means for large $n$. The main challenge of obtaining a domination result that is uniform over all choices of the shrinkage point $\psi^{(n)}$ is that the variance of $d(X^{(n)},\psi^{(n)})^2$ can be bounded below by a term involving $d(\theta^{(n)},\psi^{(n)})$. If the shrinkage point is chosen poorly so that $d(\theta^{(n)},\psi^{(n)})$ is large then the variance of $d(X^{(n)},\psi^{(n)})^2$ will also be large. Restrictions are needed that limit how fast the sequence, $\{d(\theta^{(n)},\psi^{(n)})\}_{n = 1}^{\infty}$, can increase. Despite this, if $\psi^{(n)}$ is chosen to be far away from $\theta^{(n)}$ then $E\big(d(X^{(n)},\psi^{(n)})^2\big)$ will be large which implies that almost no shrinkage will be applied and $\delta_{JS}(X^{(n)}) \approx X^{(n)}$. 

The behaviour of $d(X^{(n)},\theta^{(n)})$ can be controlled by bounding its moments. Given a sequence $m \coloneqq \{m_c\}_{c = 1}^{\infty}$ of positive real numbers, for each $n$ we define the family of probability distributions
\begin{align*}
    \mathcal{P}_m^{(n)} \coloneqq \{P = P_1 \times \ldots \times P_n: & V_2P = \sigma^2,\;  0 <  E\big(d_i(X_i,E_2X_i)^c\big) \leq m_c,
    \\
    & \; X_i \sim P_i, \; c \in \mathbb{N},\; i \in 1,\ldots , n\}.
\end{align*}
The family $\mathcal{P}_m^{(n)}$ is the set of product distributions on $\mathcal{X}^{(n)}$ that have a fixed Fr\'echet variance and have marginal distributions with ``central-moments" that are bounded by the sequence $m$. Recall that the Fr\'echet variance $V_2P$ is $\sumonn E\big(d_i(X_i^{(n)},\theta_i^{(n)})^2\big)/n$, and so it is an average of the Fr\'echet variances of the marginal distributions. In $\mathbb{R}^n$ the family $\mathcal{P}^{(n)}_m$ corresponds to product distributions with $E(\vert X_i^{(n)} - EX_i^{(n)} \vert^c) \leq m_c$ and $\sumonn \text{Var}(X_i^{(n)})/n = \sigma^2$. The condition $E\big(d(X_i^{(n)},E_2X_i^{(n)})^c\big) \leq m_c$ is stronger than $V_c(X_i^{(n)}) \leq m_c$ since $V_c(X_i^{(n)}) \leq E\big(d(X_i^{(n)},E_2X_i^{(n)})^c\big)$. 

The following theorem generalizes the classical Gaussian James-Stein domination result to the large non-parametric family $\mathcal{P}_m^{(n)}$. A mild assumption is needed that constrains how fast $d(\theta^{(n)},\psi^{(n)})^2$ can grow relative to the dimension of the Hadamard space $\mathcal{X}^{(n)}$. It will be shown that this assumption is automatically satisfied if the spaces $\mathcal{X}_i$ have uniformly bounded diameters. At the end of this section we will further prove that $\delta_{JS}$ asymptotically dominates $X^{(n)}$ and has a loss function that is less than $\sigma^2$ with probability tending to one, regardless of how fast $d(\theta^{(n)},\psi^{(n)})^2$ grows.   

\begin{theorem}\label{freqdomthm}
Let $\{a_n\}$ be a sequence with $a_n \rightarrow \infty$ and take $P \in \mathcal{P}_m^{(n)}$ to be any distribution on $\mathcal{X}^{(n)}$ with a Fr\'echet mean $\theta^{(n)}$ that satisfies $ d(\theta^{(n)},\psi^{(n)})^2 \leq n/a_n$.   
There exists an $n^*(m,\{a_n\})$ such that if $n \geq n^*$ then $R(P,\delta_{JS}) < R(P,X^{(n)})$.
\end{theorem}
\begin{proof}
See Appendix \ref{appA} for the proof.
\end{proof}

The main limitation of Theorem \ref{freqdomthm} is that the distribution of $X^{(n)} \in \mathcal{P}_m^{(n)}$ for $n \geq n^*$ must satisfy $d(\theta^{(n)},\psi^{(n)})^2/n \leq a_n^{-1} = o(1)$, 
which is similar to a condition that appears in Brown and Kou \cite{BrownKou} for a heteroskedastic normal model. Although more broadly applicable, this condition is most easily interpreted in terms of a sequence of random objects, $X^{(n)} \sim P^{(n)} \in \mathcal{P}_m^{(n)},\; n \in \mathbb{N}$. For each $n$ choose a shrinkage point $\psi^{(n)}$ and suppose that $d(\theta^{(n)},\psi^{(n)})^2/n \leq a_n^{-1}$ for all $n$. Theorem \ref{freqdomthm} guarantees that there exists an $n^*$ such that $R\big(P^{(n)},\delta_{JS}(X^{(n)})\big) < R\big(
P^{(n)},X^{(n)}\big)$ for all $n \geq n^*$. In particular, if $\lim_n d(\theta^{(n)},\psi^{(n)})^2/n \rightarrow 0$ then one can take $a_n^{-1} = d(\theta^{(n)},\psi^{(n)})^2/n$. Recall that $d(\theta^{(n)},\psi^{(n)})^2$ is an average of squared distances,  $\sumonn d_i(\theta_i^{(n)},\psi_i^{(n)})^2/n$. Therefore $\lim_n d(\theta^{(n)},\psi^{(n)})^2/n \rightarrow 0$ only requires that the average squared distance of the components of $\theta^{(n)}$ and $\psi^{(n)}$ increases at a rate that is slower than linear. Theorem \ref{freqdomthm} also shows that $n^*$ does not depend on the particular sequence of $X^{(n)}$ chosen, rather it only depends on $\{a_n^{-1}\}$ and $m$.  

Instead of starting with a sequence of random objects one can start with a sequence of shrinkage points, $\psi^{(n)}$. A dual way to view Theorem \ref{freqdomthm} is that given a sequence $a_n^{-1}$ and $m$, $\delta_{JS}$ dominates $X^{(n)}$ over the subfamily, $\{P \in \mathcal{P}^{(n)}_m:d(E_2P,\psi^{(n)})^2 \leq na_n^{-1}\}$ of $\mathcal{P}_m^{(n)}$ for $n \geq n^*(m,\{a_n\})$. A special case occurs when the metric spaces $\mathcal{X}_i$ have uniformly bounded diameters, as for a large enough $n$ this subfamily consists of all possible distributions on $\mathcal{X}^{(n)}$. This follows by taking $a_n = \sqrt{n}$ and using the fact that $d(E_2P,\psi^{(n)})^2 \leq \text{diam}(\mathcal{X}^{(n)})^2 < \infty$. Moreover, the central moments $E\big(d_i(X_i^{(n)},\theta_i^{(n)})^c\big)$ on a space with  uniformly bounded diameter cannot be larger than $\text{diam}(\mathcal{X}_i)^c$, which implies the following global domination result:     

\begin{corollary}
\label{boundeddiamcor}
If the Hadamard spaces $\mathcal{X}_i, \; i \in \mathbb{N}$ are all bounded with $\text{diam}(\mathcal{X}_i) \leq D$ for all $i$, then there exists an $n^*(D)$ such that $R(P,\delta_{JS}) < R(P,X^{(n)})$ for any distribution $P$ on $\mathcal{X}^{(n)}$ and any shrinkage point $\psi^{(n)}$, when $n \geq n^*$. 
\end{corollary}

The estimator $X^{(n)}$ is thus inadmissible for estimating the Fr\'echet mean under a squared distance loss when the $\mathcal{X}_i$'s have uniformly bounded diameters and $n$ is large enough. Notably, the dimension $n^*$ in Corollary \ref{boundeddiamcor} is independent of any choices of $\psi^{(n)}$ or $m$. Intuition for Corollary \ref{boundeddiamcor} comes from \eqref{cat0riskbd} where it is seen that there always exists an amount of shrinkage where the shrinkage estimator has lower risk than $X^{(n)}$. Under the uniform boundedness assumption on the $\mathcal{X}_i$'s the shrinkage weight $w(X^{(n)})$ concentrates around $\sigma^2/\rho(X^{(n)},\psi^{(n)})^2$ closely enough for domination to occur independently of the choice of $\psi^{(n)}$.

 Theorem \ref{freqdomthm} and Corollary \ref{boundeddiamcor} are remarkable since very few assumptions are made about the distribution of $X$, apart from assuming that the marginal distributions of $X^{(n)}$ have central moments bounded by $m_c$. On Euclidean spaces the Stein estimator has been considered for certain classes of non-normal distributions \cite{Brandwein, Srist, BrownThesis}. Most results of this type assume that $X^{(n)}$ has an elliptically symmetric distribution where further assumptions are made about various expectations of $X^{(n)}$ that allow variants of Stein's lemma to be applied. When the metric is given by an inner product, Stein's lemma is used to control the term $2\langle X^{(n)} - \theta^{(n)},\delta(X^{(n)}) - X^{(n)} \rangle$ that appears after expanding $R(P,\delta) = \Vert \delta - \theta^{(n)} \Vert^2$. In a general Hadamard space there is no such decomposition of $d(\delta,\theta^{(n)})^2$.  The assumption that the distribution of $X^{(n)}$ is spherically symmetric in $\mathbb{R}^n$ is fairly restrictive since this implies for example that the marginal distribution of each $X_i$ is the same and that $X_i$ and $-X_i$ have the same distribution.
 
 An example of a subfamily of distributions on $\mathbb{R}^n$ that is contained in $\mathcal{P}^{(n)}_m$ is the following location family \cite{Lehmann}:
 Let $F_i^{(n)}, \; i = 1,\ldots,n$ be distributions on $\mathbb{R}$ with mean $0$, variance $\sigma^2$, and central moments bounded by the sequence $\{m_c\}_{c = 1}^{\infty}$. The set of all distributions of random variables of the form $X^{(n)} = \theta^{(n)} + \epsilon^{(n)}$ for any $\theta^{(n)} \in \mathbb{R}^n$ and $\epsilon_i^{(n)} \sim F_i^{(n)}$ is contained in $\mathcal{P}_m^{(n)}$, because the location shifts $\epsilon_i^{(n)} \rightarrow  \theta_i^{(n)} + \epsilon_i^{(n)} $ do not alter any of the central moments. This location family can be restricted further by assuming that for each $n$, $\theta^{(n)}$ is known to lie is some set $\Theta^{(n)}$ with $\text{diam}(\Theta^{(n)}) \leq D$. Theorem \ref{freqdomthm} implies that if  $\psi^{(n)} \in \Theta^{(n)}$ for all $n$, then there exists a dimension $n^*(D,m)$ for which domination of $X^{(n)}$ occurs. Various results similar to this are known for distributions on $\mathbb{R}^n$ with restricted parameter spaces \cite{Marchand}. Immediate generalizations of this location family exist on arbitrary Hadamard spaces by letting the isometry group, instead of the translation group, act on a sequence of fixed distributions with bounded central moments.  

 Theorem \ref{freqdomthm} provides a domination result that applies to a subfamily of $\mathcal{P}_m^{(n)}$ for a finite number of groups. The geodesic James-Stein estimator also dominates $X$ asymptotically over all of $P_m^{(n)}$ as $n \rightarrow \infty$. 
\begin{theorem}
    Let $X^{(n)} \sim P^{(n)} \in \mathcal{P}_m^{(n)}$ for all $n \in \mathbb{N}$. If $d(\theta^{(n)},\psi^{(n)})^2 \rightarrow \infty$ for a sequence of shrinkage points $\{\psi^{(n)}\}_{n = 1}^{\infty}$, then $\limsup_n R\big(P^{(n)},\delta_{JS}(X^{(n)})\big) = \sigma^2.$ It follows from Theorem \ref{freqdomthm} that $\limsup_n R\big(P^{(n)},\delta_{JS}(X^{(n)})\big) \leq \lim_n R(P^{(n)},X^{(n)})$ for any sequence of $\psi^{(n)}$'s. Additionally, for all $\epsilon > 0$,  $\lim_n P\big(L\big(\theta^{(n)},\delta_{JS}(X^{(n)})\big) > \sigma^2 + \epsilon\big) = 0$.
\label{freqasmypthm}
\end{theorem}
\begin{proof}
See Appendix \ref{appA} for the proof.
\end{proof}

Theorem \ref{freqasmypthm} makes explicit the observation that $\delta_{JS}$ behaves similarly to $X^{(n)}$ when the shrinkage point is chosen to be far away from $E_2X^{(n)}$. Consequently, in a simultaneous Fr\'echet mean estimation problem with a large number of groups the geodesic James-Stein estimator has performance that is comparable to, or much better than, the estimator $X^{(n)}$.

The results in this section also apply to estimators of the form $[X^{(n)},\psi^{(n)}]_{\alpha w(X^{(n)})}$ where $\alpha \in (0,1]$. Such estimators apply an amount of shrinkage that is proportional to, but less than $\delta_{JS}$. It follows that $${[X^{(n)},\psi^{(n)}]_{\alpha w(X^{(n)})} = [X^{(n)},[X^{(n)},\psi^{(n)}]_{w(X^{(n)})}]_\alpha = [X^{(n)},\delta_{JS}]_\alpha},$$ from which the convexity of the squared distance function implies that
\begin{align}
d(\theta^{(n)},[X^{(n)},\psi^{(n)}]_{\alpha w(X^{(n)})})^2 \leq (1-\alpha) d(\theta^{(n)},X^{(n)})^2 + \alpha d(\theta^{(n)},\delta_{JS})^2.
    \label{lowervarbdrisk}
\end{align}
The risk of $[X^{(n)},\psi^{(n)}]_{\alpha w(X^{(n)})}$ is therefore no larger than a convex combination of the risk of $X^{(n)}$ and the risk of $\delta_{JS}$. Estimators of this form are useful when the value of $\sigma^2$ that appears in $w(X^{(n)})$ is not known but instead it is known that $\sigma^2$ is bounded below by $\alpha_0 > 0$, so that  $\alpha_0/\sigma^2 \leq 1$. By taking $\alpha = \alpha_0/\sigma^2$ the shrinkage weight $\alpha w(X^{(n)})$ is equal to  $\alpha_0/d(X^{(n)},\psi^{(n)})^2$ when the event $\{X^{(n)}:\sigma^2/d(X^{(n)},\psi^{(n)})^2 \leq 1\}$ occurs. Consequently, the estimator $[X^{(n)},\psi^{(n)}]_{\tilde{w}}$ where $\tilde{w} = 1 \wedge \alpha_0/d(X^{(n)},\psi^{(n)})^2$ will have the same large sample risk properties as $\delta_{JS}$.

\section{Analysis of the Bayes risk of $\delta_{JS}$}
\label{Sec5}

Efron and Morris \cite{Efron} show that the James-Stein estimator may be interpreted as an empirical 
Bayes procedure as follows: If $X^{(n)} \sim N_n(\theta^{(n)},\sigma^2 I)$ and 
the prior distribution for $\theta^{(n)}$ is $\theta^{(n)} \sim N_n(\mu^{(n)}, \tau^2 I)$, then 
the posterior mean estimator of $\theta^{(n)}$ is the linear shrinkage estimator 
 $(1-t)X^{(n)} + t \mu^{(n)}$, with $t = \sigma^2/(\sigma^2 + \tau^2)$. 
 If an appropriate choice of $\tau^2$ is not available, 
 Efron and Morris suggest empirically estimating its value from the data. 
 Specifically, they show that $(n-2)/\sumonn(X_i^{(n)} -\mu_i^{(n)})^2$
 is an unbiased estimator of $1/(\sigma^2 + \tau^2)$ with respect to the 
 marginal distribution of $X$. Plugging this into the expression for $t$ yields the 
 James-Stein estimator $\delta_{JS}$.  
 Whereas Stein's results on risk concerned frequentist risk, that is, risk as a function of $\theta^{(n)}$, 
Efron and Morris obtained results on the Bayes risk, the average frequentist risk with respect to the prior distribution $\theta^{(n)} \sim N_n(\mu^{(n)}, \tau^2 I)$. They showed that not only is $\delta_{JS}$ better than $X^{(n)}$ with respect to Bayes risk, $\delta_{JS}$ is almost as good as the posterior mean estimator, which is Bayes-risk optimal. For any value of $\tau^2$, the Bayes risk of $\delta_{JS}$ approaches that of the optimal posterior mean estimator as $n\rightarrow \infty$. 

In this section, we consider similar results for the geodesic James-Stein estimator. 
We first examine the Bayes risk of the geodesic James-Stein estimator in the case that
the shrinkage point is fixed at $\psi^{(n)}$.  In this case, the Bayes risk is bounded above in terms of the distance between the shrinkage point $\psi^{(n)}$ and the prior Fr\'echet mean of $\theta^{(n)}$. If the dimension $n$ is sufficiently large, $\delta_{JS}$ will have a smaller Bayes risk than $X^{(n)}$.
However, there is no guarantee that the risk of $\delta_{JS}$ will asymptotically approach the minimum Bayes risk as $n\rightarrow \infty$. 
The absence of such a result is not surprising, 
since in general 
the Bayes estimator may not be a geodesic shrinkage estimator of the form $[X^{(n)},\psi^{(n)}]_t$. For example, even for Euclidean sample spaces, Bayes estimators will not generally be linear shrinkage estimators unless the model is an exponential family and 
the prior distribution is conjugate \cite{Diaconis}.
Next, we compare the Bayes risk of $X^{(n)}$ to that of a potentially more 
useful shrinkage estimator, one for which the shrinkage point is empirically 
estimated from the data $X^{(n)}$. This is done in a setting that generalizes the 
simple hierarchical normal model $X^{(n)} \sim N_n(\theta^{(n)}, \sigma^2I)$ and 
$\theta^{(n)} \sim N_n( \tilde \mu 1 , \tau^2 I)$, where $\tilde \mu \in \mathbb R$ and $1$ is 
an $n$-vector of all ones. Empirical Bayes estimation of both 
$\tilde \mu$ and $\tau^2$ is possible since they are common to all elements of $\theta^{(n)}$, 
and therefore, common to all elements of $X^{(n)}$.  We consider an analogous scenario in which the prior Fr\'echet mean of each element of $\theta^{(n)}$ is equal to a common value $\tilde \mu$. Under this assumption, $\tilde{\mu}$ can approximately be estimated by  the sample Fr\'echet mean $\bar{X}^{(n)}$ of $X_1^{(n)},\ldots, X_n^{(n)}$. The resulting estimator $\delta_{JS}$ has a smaller Bayes risk than $X^{(n)}$, where unlike in the frequentist case, this result is global and does not only apply to a sub-family of $\mathcal{P}_{m}^{(n)}$. Recall that the primary difficulty in obtaining a global domination result of $\delta_{JS}$ over $X^{(n)}$ in the frequentist case was that the shrinkage point may be far away from $\theta^{(n)}$. By adaptively choosing the shrinkage point in the Bayesian setting there is no longer this concern as $\Bar{X}^{(n)}$ will be reasonably close to $\theta^{(n)}$ with high probability.

\subsection{Bayes Risk of $\delta_{JS}$}
\label{sec5.1} 

Throughout this section we work with a prior distribution $Q^{(n)} = Q_1^{(n)} \times \cdots \times Q_n^{(n)}$ for the estimand $\theta^{(n)} = (\theta_1^{(n)},\ldots,\theta_n^{(n)})$, so that the components $\theta_i^{(n)}$ of $\theta^{(n)}$ are mutually independent under this prior distribution. Let $\mu^{(n)} \in \mathcal{X}^{(n)}$ be the Fr\'echet mean of $Q^{(n)}$ and take $\tau^2$ to be the Fr\'echet variance of $Q^{(n)}$. Conditional on $\theta^{(n)}$ the distribution $P_{i,\theta_i^{(n)}}^{(n)}$ of $X_i^{(n)}$ is assumed to have Fr\'echet mean $\theta_i^{(n)}$ and Fr\'echet variance $\sigma_i^2$. Furthermore, we assume conditional independence of the $X_i^{(n)}$ given $\theta^{(n)}$ so that this conditional distribution is denoted by $P_{\theta^{(n)}}^{(n)} = P_{1,\theta_1^{(n)}}^{(n)} \times \cdots \times P_{n,\theta_n^{(n)}}^{(n)}$. Lastly we assume some additional moment conditions so that $Q^{(n)} \in \mathcal{P}_{l}^{(n)}$ for some sequence $l = \{l_c\}_{c = 1}^{\infty}$ and $P_\theta^{(n)} \in \mathcal{P}_m^{(n)}$ for every $\theta \in \mathcal{X}^{(n)}$ for some sequence $m = \{m_c\}_{c = 1}^{\infty}$. In summary, the joint distribution of $X$ and $\theta$ has the form
\begin{align}
& \theta^{(n)} \sim Q^{(n)} = Q_1^{(n)} \times \ldots \times Q_n^{(n)} \in \mathcal{P}_l^{(n)}, \;\;  E_2Q^{(n)} = \mu^{(n)}, \;\;  V_2Q^{(n)} = \tau^2, \nonumber
\\
 &   X^{(n)}|\theta^{(n)} \sim P_{\theta^{(n)}}^{(n)} = P_{1,\theta_1^{(n)}}^{(n)} \times \cdots \times P_{n,\theta_n^{(n)}}^{(n)} \in \mathcal{P}_m^{(n)}, \;\;   E_2P_{\theta^{(n)}}^{(n)} = \theta^{(n)},\;\; V_2P_\theta^{(n)} = \sigma^2.
   \label{hierdistassume}
\end{align}
The results of this section remain non-parametric as they apply to any choice $Q^{(n)}$ and $P_{\theta^{(n)}}^{(n)}$ that satisfy \eqref{hierdistassume}. 
 Notice that the model formulation in \eqref{hierdistassume} still does not explicitly posit any relationship between the distributions of the various $(X_i^{(n)},\theta_i^{(n)})$'s. Certain choices of $P_{\theta^{(n)}}^{(n)}$ and $Q^{(n)}$ will however induce similarities between the distributions of the $(X_i^{(n)},\theta_i^{(n)})$'s. For example, the standard Gaussian hierarchical model is encompassed by \eqref{hierdistassume} by taking $P_{\theta^{(n)}}^{(n)} = N_n(\theta^{(n)},\sigma^2 I)$ and $Q^{(n)} = N_n(\mu^{(n)},\tau^2 I)$.

As in Section \ref{Sec4}, the estimation problem of interest is to estimate $\theta^{(n)}$ under squared distance loss where the only known quantities in \eqref{hierdistassume} are $X^{(n)}$ and $\sigma^2$. Theorem \ref{freqdomthm} extends to this setting where a prior distribution is placed on $\theta^{(n)}$ by evaluating the performance of $\delta_{JS}(X^{(n)})$ in terms of its Bayes risk.
\begin{theorem}
\label{Bayesdom1}
Under the distributional assumptions in \eqref{hierdistassume}, suppose that there is a sequence $a_n \rightarrow \infty$ such that  $ d(\mu^{(n)},\psi^{(n)})^2 \leq n/a_n$. There exists an $n^*(m,l,\{a_n\})$ such that if $n \geq n^*$ then the Bayes risk satisfies $E\big(R(P_{\theta^{(n)}}^{(n)},\delta_{JS})\big) < E\big(R(P_{\theta^{(n)}}^{(n)},X^{(n)})\big)$.
\end{theorem}
\begin{proof}
See Appendix \ref{appA} for the proof. 
\end{proof}

The bound on the distance $d(\theta^{(n)},\psi^{(n)})^2$ that appears in Theorem \ref{freqdomthm} is replaced by a bound on $d(\mu^{(n)},\psi^{(n)})^2$ in Theorem \ref{Bayesdom1}. A special sub-model of \eqref{hierdistassume} where the condition $d(\mu^{(n)},\psi^{(n)})^2/n = o(1)$ is easily satisfied is where $\mathcal{X}_i = \mathcal{X}$ for all $i$ and $Q^{(n)}$ has the form $Q^{(n)} = \tilde{Q} \times \cdots \times \tilde{Q}$ for all $n$.  Throughout this section, tildes will be used to denote points, metrics and distributions on $\mathcal{X}$ when $\mathcal{X}^{(n)} = \mathcal{X}^n$ is a Cartesian product of $\mathcal{X}$.  If $\psi^{(n)} = (\tilde{\psi},\ldots,\tilde{\psi})$ is chosen to have identical component-wise entries for all $n$ then $d(\mu^{(n)},\psi^{(n)})^2 = \tilde{d}(\tilde{\mu},\tilde{\psi})^2$ is constant over $n$ and so it is $o(n)$. Using such a sequence of $\psi^{(n)}$'s, Theorem \ref{Bayesdom1} guarantees the existence of an $n^*$ for which $\delta_{JS}$ has a smaller Bayes risk than $X^{(n)}$ for $n \geq n^*$. The dimension that is needed for this smaller Bayes risk is still shrinkage point dependent since it is contingent on the value of $\tilde{d}(\tilde{\mu},\tilde{\psi})^2$. In this case we can write $n^*(m,l,\{a_n\})$ as $n^*\big(m,l,\tilde{d}(\tilde{\mu},\tilde{\psi})\big)$.

Theorem \ref{Bayesdom1} applies to the location family example introduced in the previous section where $X^{(n)} = \theta^{(n)} + \epsilon^{(n)}$. The only modification needed is that $\theta^{(n)}$ is now assumed to have the distribution $\theta_i^{(n)} \sim \tilde{Q} \in \mathcal{P}^{(1)}_l$ independently for $i = 1,\ldots,n$. Even in this specific example, the class of distributions on $\theta^{(n)}$ and $\epsilon^{(n)}$ to which these results hold is very broad. Suppose that the shrinkage point is taken to have equal component-wise entries, $\tilde{\psi}$. The dimension $n^*(m,l,\tilde{d}(\tilde{\mu},\tilde{\psi}))$ needed holds for any mean zero error distribution of $\epsilon^{(n)}$ that is in $\mathcal{P}_m^{(n)}$ with $V_2\epsilon^{(n)} = \sigma^2$. Likewise, $n^*(m,l,\tilde{d}(\tilde{\mu},\tilde{\psi}))$ applies to any distribution $\tilde{Q} \in \mathcal{P}^{(1)}_l$ as long as $\tilde{d}(E_2\tilde{Q},\tilde{\psi}) \leq \tilde{d}(\tilde{\mu},\tilde{\psi})$.

Theorem \ref{freqasmypthm} can similarly be extended to a Bayesian setting. 
\begin{theorem}
\label{Bayesdom2}
Let $X^{(n)} \sim P^{(n)}_{\theta^{(n)}}, \; n \in \mathbb{N}$ and $E_2X^{(n)} = \theta^{(n)} \sim Q^{(n)}, \; n \in \mathbb{N}$ satisfy the distributional assumptions in \eqref{hierdistassume}. If $d\big(\mu^{(n)},\psi^{(n)}\big)^2 \rightarrow \infty$ for a sequence of shrinkage points $\{\psi^{(n)}\}_{n = 1}^{\infty}$, then $\limsup_n E\big(R(P^{(n)}_{\theta^{(n)}},\delta_{JS})\big) = \lim_n E\big(R(P^{(n)}_{\theta^{(n)}},X^{(n)})\big)$. By Theorem \ref{Bayesdom1}, for any sequence of $\psi^{(n)}$'s, $\limsup_n E\big(R(P^{(n)}_{\theta^{(n)}},\delta_{JS})\big) \leq \lim_n E\big(R(P^{(n)}_{\theta^{(n)}},X^{(n)})\big)$, with strict inequality if $d\big(\mu^{(n)},\psi^{(n)}\big)^2/n = o(1)$. Additionally, we have that for all $\epsilon > 0$, $\lim_n P\big(L(\theta^{(n)},\delta_{JS}) > \sigma^2 + \epsilon\big) = 0$.
\end{theorem}
\begin{proof}
See Appendix \ref{appA} for the proof.
\end{proof}

It should be noted that the distributional assumptions in \eqref{hierdistassume} do not constitute a fully Bayesian model since $P_{\theta^{(n)}}^{(n)}$ and the prior distribution $Q^{(n)}$, although constrained, are both left unspecified. By leaving $P_{\theta^{(n)}}^{(n)}$ and $Q^{(n)}$ unspecified the results above can be regarded as part of a robust Bayesian analysis that compares the Bayes risk of $\delta_{JS}$ to $X^{(n)}$ over a wide class of joint distributions for $(X^{(n)},\theta^{(n)})$ \cite{Berger}. A fully Bayesian model can be obtained from \eqref{hierdistassume} if hyper-priors are placed on both $P_{\theta^{(n)}}^{(n)}$ and $Q^{(n)}$. 

\subsection{Bayes Risk for an Adaptively Chosen Shrinkage Point}
\label{Sec5.2}

In scenarios where the distributions of $(X_i^{(n)},\theta_i^{(n)}), \; i = 1,\ldots,n$ are exchangeable it is reasonable to require that an estimator of $\theta^{(n)}$ be equivariant under the permutation of indices. This symmetry consideration suggests that the shrinkage point $\psi^{(n)}$ used in $\delta_{JS}$ should have identical component-wise entries. 

It is intuitively clear that a good choice of $\psi^{(n)}$ should be close to $\theta^{(n)}$ on average.In the proof of Theorem \ref{Bayesdom2}, it was that $
\lim_n E[(a) + (c)] = 0$, for the terms $(a),(c)$ in \eqref{lossbound}. We make the further assumption that for all $n \in \mathbb{N}$,  $Q^{(n)} = \tilde{Q} \times \cdots \times \tilde{Q}$ and $P^{(n)}_{\theta^{(n)}} = \tilde{P}_{\theta_1^{(n)}} \times \cdots \times \tilde{P}_{\theta_n^{(n)}}$. Therefore the joint distribution of $(X_i^{(n)},\theta_i^{(n)})$ is the same for each group. By the definition of $Q^{(n)}$,\; $\mu^{(n)} = (\tilde{\mu},\ldots,\tilde{\mu})$, and if $\psi^{(n)} = (\tilde{\psi},\ldots,\tilde{\psi})$ has identical component-wise entries, this implies
\begin{align}
        \limsup_{n \rightarrow \infty} E\big(R(P_{\theta^{(n)}}^{(n)},\delta_{JS})\big) \leq & \;     \limsup_{n \rightarrow \infty} E\bigg[I_A \frac{d(\theta^{(n)},\psi^{(n)})^2}{d(X^{(n)},\psi^{(n)})^2}\bigg]\sigma^2 = \frac{E\big(d(\theta^{(n)},\psi^{(n)})^2\big)}{E\big(d(X^{(n)},\psi^{(n)}\big)^2)}\sigma^2 \nonumber
        \\
        \leq & \; \frac{E(d\big(\theta^{(n)},\psi^{(n)}\big)^2)}{\sigma^2 + E\big(d(\theta^{(n)},\psi^{(n)}\big)^2)}\sigma^2.
    \label{asymptriskofJS}
\end{align}
The second equality in \eqref{asymptriskofJS} holds since the integrand is uniformly integrable because it is in $L^{1+\epsilon}(\mathbb{R})$ for some $\epsilon > 0$ since $I_A/d(X^{(n)},\psi^{(n)})^2 \leq 1/\sigma^2$. The strong law of large numbers shows that $d(\theta^{(n)},\psi^{(n)})^2 \overset{a.s}{\rightarrow} E\big(d(\theta^{(n)},\psi^{(n)})^2\big)$ and $d(X^{(n)},\psi^{(n)})^2 \overset{a.s}{\rightarrow} E\big(d(X^{(n)},\psi^{(n)})^2\big)$ from which the second equality follows. The last inequality is a result of the Hadamard bias variance inequality \eqref{hadbiasvar} applied to $E\big(d(X^{(n)},\psi^{(n)})^2|\theta^{(n)}\big)$. The upper bound of \eqref{asymptriskofJS} is minimized over $\tilde{\psi}$ when $\tilde{\psi} = \argm_{\tilde{\psi} \in \mathcal{X}} E\big(d(\theta^{(n)},\psi^{(n)}\big)^2) = \argm_{\tilde{\psi} \in \mathcal{X}} E\big(\tilde{d}(\theta_1^{(n)},\tilde{\psi}\big)^2)$. By the definition of $E_2\theta_1^{(n)}$, $\tilde{\psi} = E_2\theta_1^{(n)} = \tilde{\mu}$ is the minimizer of the asymptotic risk upper bound in \eqref{asymptriskofJS}. At this optimal value of $\psi^{(n)}$, the asymptotic Bayes risk of $\delta_{JS}$ is at most $\tau^2/(\sigma^2 + \tau^2)$ percent of the risk of $X^{(n)}$. If either of the inequalities in \eqref{asymptriskofJS} are strict $\delta_{JS}$ may offer an even greater improvement over $X^{(n)}$. 

The preceding discussion makes precise the intuition that $\tilde{\psi}$ should be chosen so that it is close to $\tilde{\mu}$. From the observations $X^{(n)} = (X_1^{(n)},\ldots,X_n^{(n)})$, an estimate of $\tilde{\mu}$ can be obtained by calculating the sample Fr\'echet mean of $X^{(n)}$. The sample Fr\'echet mean, $\widebar{X}^{(n)}$, is the Fr\'echet mean of the empirical distribution of the observations $X_1^{(n)},\ldots,X_n^{(n)}$ so that
\begin{align}
    \widebar{X}^{(n)} \coloneqq \underset{x \in \mathcal{X}}{\argm}\big(\sumonn d(x,X_i^{(n)})^2\big).
\end{align}
In Euclidean space, the sample Fr\'echet mean is simply the sample mean.
Under regularity conditions, 
the sample Fr\'echet mean of an independent and identically distributed sample $\{X_i^{(n)}\}_{i = 1}^n$, converges in $L^2(\mathcal{X})$ to $E_2X_1^{(n)}$ as $n \rightarrow \infty$. Consequently, we propose using the data dependent shrinkage point,  $\tilde{\psi} = \widebar{X}^{(n)}$. It may not, however, be the case that $E_2X_1^{(n)}$ is the asymptotically optimal point $\tilde{\mu}$. The point $\tilde{\mu}$ is defined by $\tilde{\mu} = E_2\theta_1^{(n)} = E_2\big(E_2(X_1^{(n)}|\theta_1^{(n)})\big)$, which is not guaranteed to equal $E_2X_1^{(n)}$ as the tower rule does not always hold in a general Hadamard space (see Appendix \ref{appB}).  

It was shown in Theorem \ref{Bayesdom1} that the dimension needed for $\delta_{JS}$ to outperform $X$, $n^*$, is a function of $m,l$ and $\tilde{d}(\tilde{\mu},\tilde{\psi})$. If $\widebar{X}^{(n)}$ is sufficiently close to $E_2X_1^{(n)}$ then the $n^*$ needed when using this adaptive shrinkage point will approximately be a function of $m,l$ and $\tilde{d}(\tilde{\mu},E_2X_1^{(n)})$. The bias-variance inequality shows that $\tilde{d}(\tilde{\mu},E_2X_1^{(n)})^2 \leq  E\big(\tilde{d}(X_1^{(n)},\tilde{\mu})^2\big)$, while the triangle inequality $\tilde{d}(X_1^{(n)},\tilde{\mu}) \leq \tilde{d}(X_1^{(n)},\theta_1^{(n)}) + \tilde{d}(\theta_1^{(n)},\tilde{\mu})$ can be used to show that  $\tilde{d}(\tilde{\mu},E_2X_1^{(n)})$ can be bounded above entirely in terms of $m$ and $l$. The next theorem makes this reasoning precise and proves the existence of an $n^*(m,l)$ for which the James-Stein estimator with an adaptive shrinkage point has a smaller Bayes risk than $X$.  
\begin{theorem}
\label{adaptthm}
Assume that $X^{(n)} \sim P^{(n)}_{\theta^{(n)}} = \tilde{P}_{\theta_1^{(n)}} \times \cdots \times \tilde{P}_{\theta_n^{(n)}}$ and $\theta^{(n)} \sim Q^{(n)} = \tilde{Q} \times \cdots \times \tilde{Q}$ for all $n \in \mathbb{N}$. If $E\big(\tilde{d}(\widebar{X}^{(n)},E_2X_1^{(n)})^2\big) = O(n^{-1})$ with a multiplicative constant that only depends on $m$ and $l$, then there exists an $n^*(m,l)$ such that for $n \geq n^*$ then $E\big(R(P^{(n)}_{\theta^{(n)}},\delta_{JS})\big) < E\big(R(P^{(n)}_{\theta^{(n)}},X^{(n)})\big)$,  where $\delta_{JS}$ is the adaptive shrinkage estimator given by \eqref{jscomponents} with $\psi_i^{(n)} = \widebar{X}^{(n)}$. Furthermore, the same $n^*$ is valid for any distributions $\tilde{P}_{\theta_i}^{(n)} \in \mathcal{P}^{(1)}_m$ and $\tilde{Q}^{(n)} \in \mathcal{P}^{(1)}_l$.
\end{theorem}
\begin{proof}
See Appendix \ref{appA} for the proof.
\end{proof}
This result demonstrates that by choosing the shrinkage point adaptively there is no longer any concern that  $d(\mu^{(n)},\psi^{(n)})^2$ grows at too fast a rate. The shrinkage point $\widebar{X}^{(n)}$ is on average close enough to $\tilde{\mu}$ so that it is beneficial to shrink $X^{(n)}$ towards $\widebar{X}^{(n)}$. Fixing the conditional distribution $P_{\theta^{(n)}}^{(n)}$, Theorem \ref{adaptthm} shows that $\delta_{JS}$ has a strictly smaller $\mathcal{P}_l^{(n)}$-Bayes risk, $\sup_{Q^{(n)} \in \mathcal{P}^{(n)}_l}E\big(R(P_{\theta^{(n)}}^{(n)},\delta_{JS})\big)$, than $X^{(n)}$ for $n \geq n^*$ \cite{Berger}. 

The condition $E\big(\tilde{d}(\widebar{X}^{(n)},E_2X_1^{(n)})^2\big) = O(n^{-1})$ in Theorem \ref{adaptthm} is not overly restrictive. For example, if $\mathcal{X}$ is a Hilbert space then ${E\big(\tilde{d}(\widebar{X}^{(n)},E_2X_1^{(n)})^2\big) = (\sigma^2 + \tau^2)/n}$. More generally, it is shown in \cite{Schotz} that if $\mathcal{X}$ satisfies the entropy condition 
\\
\noindent $\sqrt{\log\big(N(B_{\alpha}(\mu),r)\big)} \leq cr^t/\alpha^s$ for any $\alpha,r > 0$ and fixed numbers $c,t,s \in \mathbb{R}^+$ with $t =s < 1$ then the desired condition holds with a multiplicative constant that only depends on $m$ and $l$. The number $N(B_{\alpha}(\mu),r)$ is defined as the covering number of the ball of radius $\alpha$ centered at $\mu$ by balls of radius $r$. Many spaces of interest, such as the metric tree space with vertex degrees that are bounded above and edge lengths that are bounded below, will satisfy this covering number condition. In fact, it is not fully necessary that $E\big(\tilde{d}(\widebar{X}^{(n)},E_2X_1^{(n)})^2\big)$ be $O(n^{-1})$ for the conclusion of Theorem \ref{adaptthm} hold; all that is needed is $E\big(\tilde{d}(\widebar{X}^{(n)},E_2X_1^{(n)})^2\big) = o(1)$. However, in such a case the $n^*(m,l)$ needed will also depend on the rate of convergence of $E\big(\tilde{d}(\widebar{X}^{(n)},E_2X_1^{(n)})^2\big)$ to zero.

\subsection{Asymptotic Optimality of $\delta_{JS}$}
\label{Sec5.3}
As mentioned, it is too much to expect that $\delta_{JS}$ asymptotically attain the optimal Bayes risk for a given sampling model, as a Bayes estimator may not take the form of a shrinkage estimator. The asymptotic Bayes risk of $\delta_{JS}$ can instead be compared against the risk of the best possible shrinkage estimator. We define the minimum shrinkage Bayes risk of the model in \ref{Sec5.2} as 
\begin{align*}
    \underset{\tilde{\psi} \in \mathcal{X},\; t \in [0,1]}{\inf} E\big( d([X^{(n)},\psi^{(n)}]_t,\theta^{(n)})^2\big).
\end{align*}
The same derivation used in \eqref{optimalshrinkwt} shows that for a given $\psi = (\tilde{\psi},\ldots,\tilde{\psi)}$ the shrinkage weight that minimizes the CAT(0) upper bound is 
\begin{align}
    \tilde{t} = \frac{\sigma^2 + \rho(X^{(n)},\psi^{(n)})^2 - \rho(\theta^{(n)},\psi^{(n)})^2}{2\rho(X^{(n)},\psi^{(n)})^2}.
    \label{bayesoptimalweight}
\end{align}
As the James-Stein shrinkage weight $w(X)$ converges to  $\sigma^2/\rho(X^{(n)},\psi^{(n)})^2$ almost surely, $\delta_{JS}$ only minimizes the CAT(0) bound asymptotically if $\rho(X^{(n)},\psi^{(n)})^2 - \rho(\theta^{(n)},\psi^{(n)})^2 = \sigma^2$. If $\mathcal{X}$ has negative curvature it is typical that $\rho(X^{(n)},\psi^{(n)})^2 - \rho(\theta^{(n)},\psi^{(n)})^2 > \sigma^2$ so that $\delta_{JS}$ asymptotically performs less shrinkage than is needed to minimize the CAT(0) bound. 

Determining the minimizer of the CAT(0) bound with respect to $\psi$ is more complex. If the above value of $\tilde{t}$ is substituted into the CAT(0) bound, then the resulting expression is
\begin{align*}
    \tilde{t}\sigma^2 + (1-\tilde{t})\rho(\theta^{(n)},\psi^{(n)})^2 - & \tilde{t}(1-\tilde{t})\rho(X^{(n)},\psi^{(n)})^2 = 
    \\
    & \rho(\theta^{(n)},\psi^{(n)})^2 -
     \frac{\big(\rho(\theta^{(n)},\psi^{(n)})^2 + \rho(X^{(n)},\psi^{(n)})^2 -
   \sigma^2\big)^2}{4\rho(X^{(n)},\psi^{(n)})^2}.
\end{align*}
The above expression can also be simplified in the special case when $\rho(X^{(n)},\psi^{(n)})^2 = \rho(\theta^{(n)},\psi^{(n)})^2 + \sigma^2$, where it equals $\sigma^2\rho(\theta^{(n)},\psi^{(n)})^2/(\sigma^2+ \rho(\theta^{(n)},\psi^{(n)})^2)$. In this case it is seen that the optimal choice of $\psi^{(n)}$ is $E_2\theta^{(n)}$ as this minimizes $\rho(\theta^{(n)},\psi^{(n)})^2$. The condition $\rho(X^{(n)},\psi^{(n)})^2 = \rho(\theta^{(n)},\psi^{(n)})^2 + \sigma^2$ is satisfied in any Hilbert space, as this is just the bias-variance decomposition. Furthermore, the CAT(0) bound holds with equality in a Hilbert space so the shrinkage estimator minimizing the Bayes risk is the familiar estimator, $[X^{(n)},E_2\theta^{(n)}]_{\sigma^2/(\sigma^2 + \tau^2)}$. The tower rule also holds in a Hilbert space so $\widebar{X}^{(n)} \rightarrow E_2\theta_1^{(n)}$ in $L^2(\mathcal{X})$. The bound in \eqref{asymptriskofJS} thus shows that $\delta_{JS}$ attains the minimum Bayes shrinkage risk asymptotically in a Hilbert space. For example, in the location family example in $\mathbb{R}^n$, the Bayes risk of the adaptive James-Stein estimator approaches the minimum Bayes risk out of all linear estimators of $\theta^{(n)}$ as $n \rightarrow \infty$.  

Without any additional assumptions on the metric in a Hadamard space with negative Alexandrov curvature, not much can be said about the asymptotic optimality of $\delta_{JS}$. The CAT(0) upper bound may not fully reflect the behaviour of the risk function in such a space.

\section{Numerical Results}
\label{Sec6}
In this section two simulation studies are presented that demonstrate situations in which the performance of the geodesic James-Stein estimator improves considerably over that of the estimator $X$. For the scenarios considered here the $n^*$ needed for $\delta_{JS}$ to have a lower Bayes risk than $X$ appears to be small, so that only a few groups are needed for the geodesic James-Stein estimator to be effective. 
\subsection{Log-Euclidean Metric on Positive-Definite Matrices}
One popular choice of a metric on the space of $k \times k$ symmetric positive-definite matrices $SPD(k)$ is the log-Euclidean metric defined by $d(A,B) = \Vert \log(A) - \log(B)\Vert_F$ where $\log(\cdot)$ is the matrix logarithm and $\Vert \cdot \Vert_F$ is the Frobenius norm. If $A$ has the eigendecomposition $A = U\Lambda U^T$ then $\log(A) = U\log(\Lambda)U^T$ where $\log(\Lambda) = \text{diag}(\log(\lambda_{ii}))$. The log-Euclidean metric is used extensively in diffusion tensor imaging in part because of its ease of computation and invariance properties \cite{Pennec}. Under the log-Euclidean metric, $SPD(k)$ is a Hilbert space and therefore also a Hadamard space. To see this, first note that the matrix logarithm is a bijection from $SPD(k)$ onto the space of symmetric $k \times k$ matrices, $S(k)$. As $S(k)$ is a subspace of the vector space of all $k \times k$ matrices with the Frobenius norm, $(SPD(k),d)$ is isometric to this $k(k+1)/2$ dimensional Hilbert space. Consequently, the log of the sample Fr\'echet mean of a collection of matrices $X_1,\ldots,X_n$ under the log-Euclidean metric is just the arithmetic mean of the log-transformed matrices, namely $\frac{1}{n}\sumonn \log(X_i)$. Converting back to the original coordinates shows that $\widebar{X} = \exp\big(\frac{1}{n}\sumonn \log(X_i)\big)$ where $\exp(\cdot)$ is the matrix exponential. Likewise, the Fr\'echet mean of a random matrix $X$ is $\exp\big(E(\log(X))\big)$ where $E(\log(X))$ is the standard expectation on $S(k)$.

To test the frequentist performance of variants of the geodesic James-Stein estimator we consider the case where $X_i = W_i/k$ with $W_i \sim \text{Wishart}_k(\Psi_i,k), \; i = 1,\ldots,n$ where $k = 3$ and the $\Psi_i$'s vary over the space $SPD(k)$. We are interested in simultaneously estimating $\theta_i = E_2(X_i|\Psi_i)$ for each $i$ from the $X_i$ observations using the geodesic James-Stein estimator. Importantly, the Fr\'echet mean of $X_i|\Psi_i$ is not equal to its Euclidean mean $\Psi_i$. Rather, $\theta_i$ is a non-linear function of $\Psi_i$. The eigenvalues of $\theta_i$ will typically be smaller than that of $\Psi_i$ when $\Psi_i$ is close to diagonal. Heuristically, this is a consequence of Jensen's inequality, for if $Z$ is assumed to be a diagonal random matrix then $\sum_i\exp\big(E(\log(z_{ii}))\big)^2 = \Vert E_2Z \Vert^2_F \leq  \Vert EZ \Vert^2_F = \sum_i E(z_{ii})^2$.  Although the geometry of $(SPD(k),d)$ is easily understood as a Hilbert space, the matrix logarithm transforms matrices in a non-linear manner. The resulting distribution of $\log(X_i)$ for the above Wishart model is decidedly non-Gaussian on the Hilbert space of symmetric matrices, and so the classical theory of James-Stein estimation does not apply in this setting.

Monte Carlo estimation is used to compute the value of the frequentist risk functions, $R(\Psi,X)$ and $R(\Psi,\delta_{JS})$, at a fixed value of $\Psi \coloneqq (\Psi_1,\ldots,\Psi_n)$. As a means of exploring the behaviour of the James-Stein risk function for various choices of $\Psi$ we draw each of the $\Psi_i$'s independently from the diffuse distribution $\Psi_i = W_i/k$ where $W_i \overset{i.i.d}{\sim} \text{Wishart}_k(I,k)$. This is done $100$ times so that the risks of $\delta_{JS}$ and $X$ are evaluated at $100$ different values of $\Psi$. As a distribution over $\Psi$ is involved, this analysis only explores the frequentist risk over the region of $\Psi$'s that occur with medium to high probability. Figure \ref{fig3} shows the proportion the $\Psi$ values where the risk of $\delta_{JS}$ is lower than $X$. Three different choices of the shrinkage point, $\psi = 10I, 100I$ and $\widebar{X}$, are used in $\delta_{JS}$. Figure \ref{fig3} illustrates that as $n$ increases, $\delta_{JS}$ outperforms $X$ for every value of $\Psi$. Under the distribution placed on $\Psi_i$, $EX_i = I$. As the log-Euclidean mean tends to produce matrices with smaller eigenvalues than the Euclidean mean, $\delta_{JS}$ will have better performance for shrinkage points $aI$ that have $a \leq 1$. Consequently, around $n \approx 40$ groups are required in order for $\delta_{JS}$ with $\psi = 100I$ to have a smaller risk than $X$ for every value of $\Psi$ drawn from the diffuse distribution. When $\psi = 10I$ and $\psi = \widebar{X}$, only around $n \approx 15$ and $n \approx 10$ groups are needed respectively. A fewer number of groups are needed since the shrinkage points $\psi = 10I,\widebar{X}$ are on average closer to $E_2X_i$ than $\psi = 100I$ is, and therefore are closer to the $\theta_i$'s on average. 

 \begin{figure}[h]
    \centering
    \begin{subfigure}[t]{.32\textwidth}
    \centering
  \includegraphics[width = \textwidth]{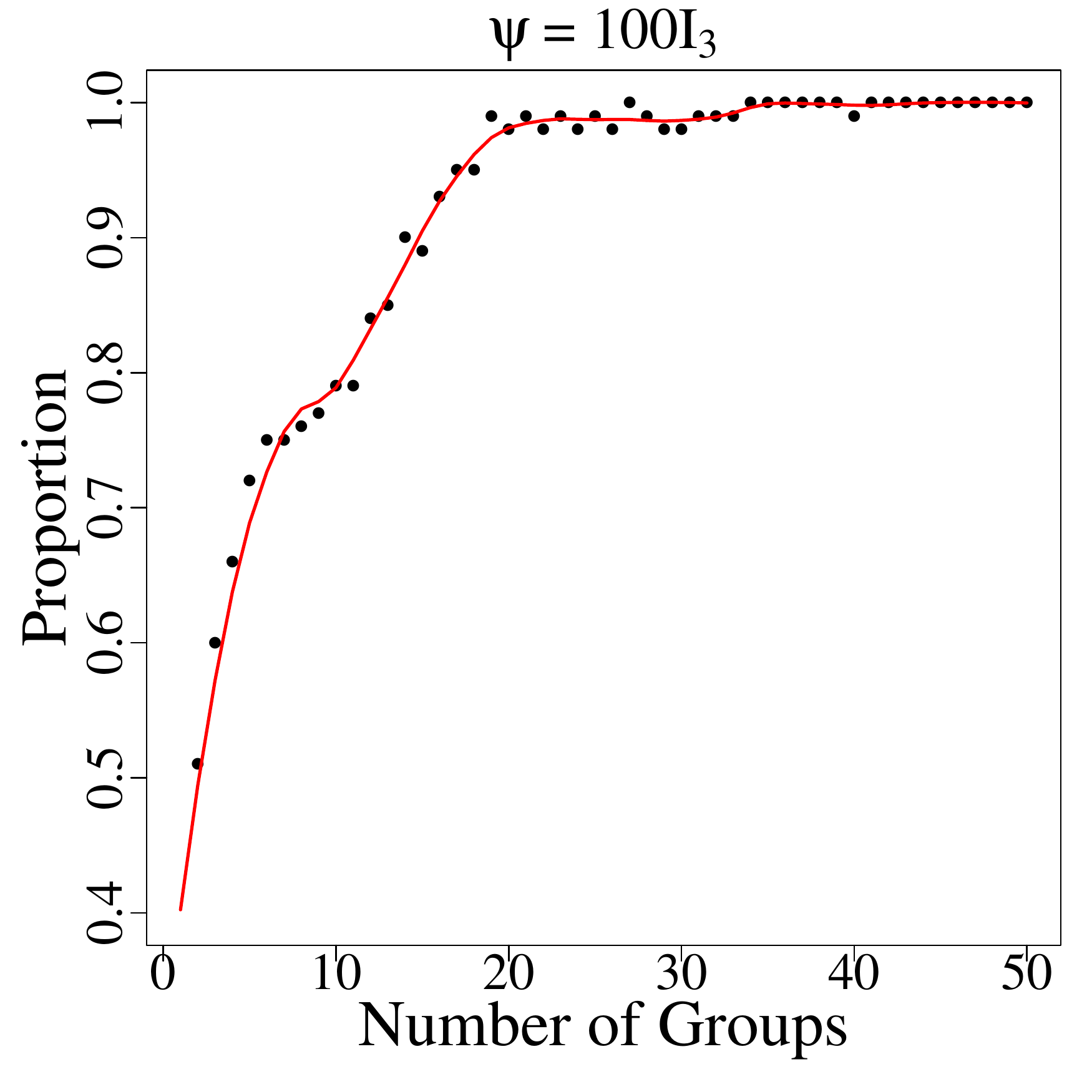}
\end{subfigure}
\begin{subfigure}[t]{.32\textwidth}
\centering        
\includegraphics[width = \textwidth]{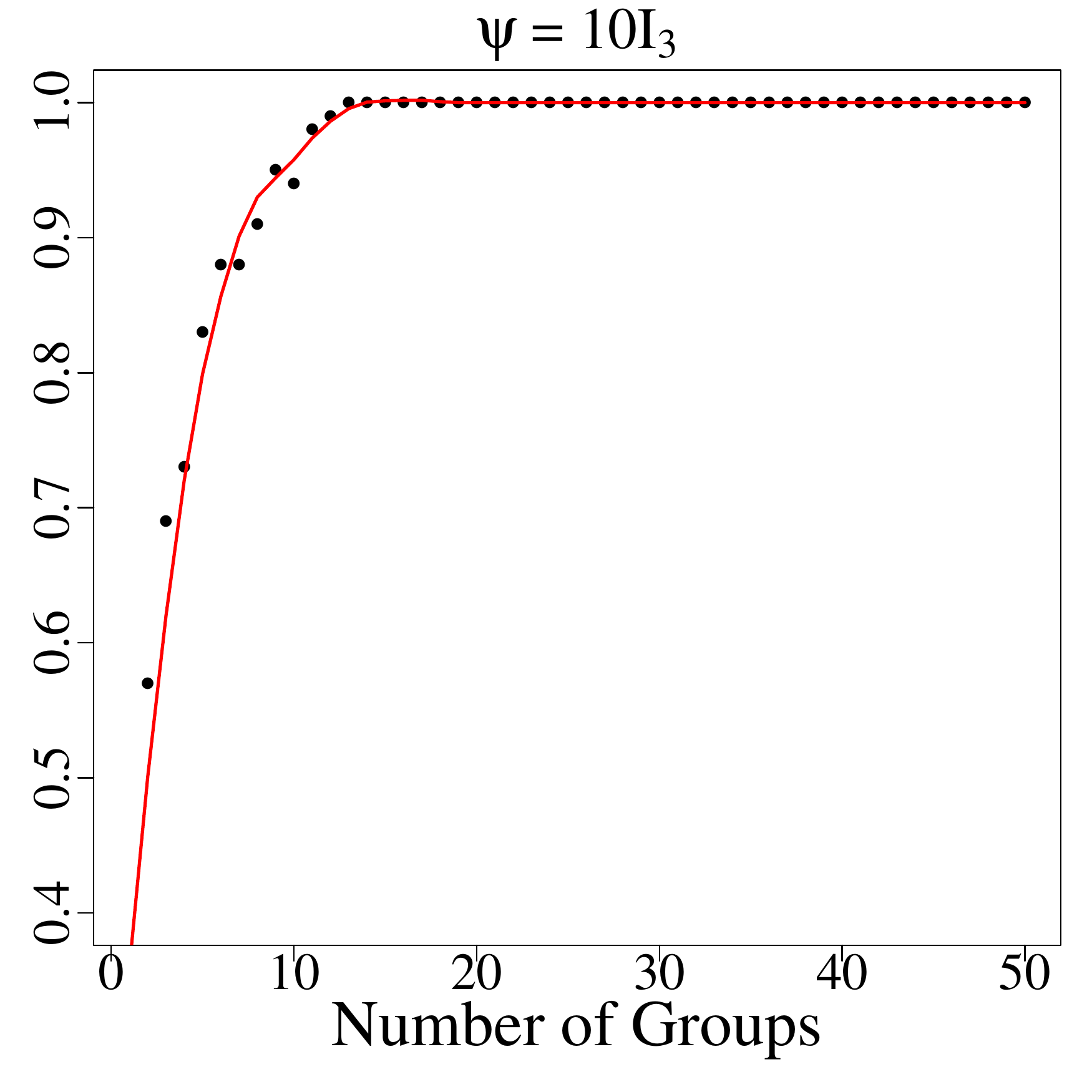}
\end{subfigure}
\begin{subfigure}[t]{.32\textwidth}
\centering        
\includegraphics[width = \textwidth]{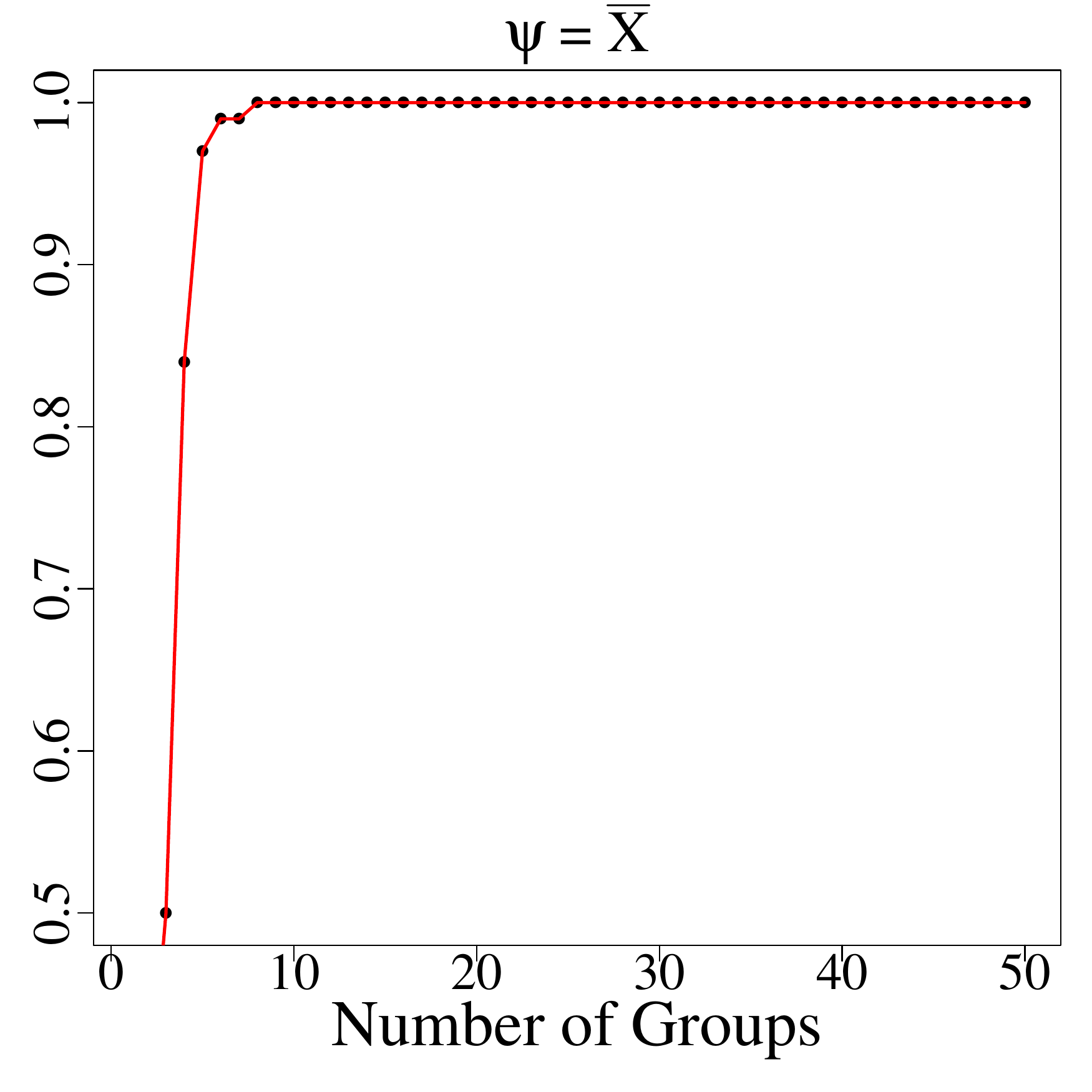}
\end{subfigure}
\caption{Proportion of $\Psi$'s where $\delta_{JS}$ has smaller risk than $X$ }
\label{fig3}
\end{figure}

The Bayes risk of $\delta_{JS}$ is also computed via Monte Carlo estimation for the following hierarchical model
\begin{align}
    (k + \alpha)X_i|\Psi_i \sim  \text{Wishart}_k(\Psi_i,k + \alpha), \;\; k\Psi_i \overset{iid}{\sim} \text{Wishart}_k(I,k), \;\; i = 1,\ldots,n
    \label{wishmodel}
\end{align} 
where $k = 3$ and $\alpha = (0,2,8)$. The added $\alpha$ parameter represents the concentration of $X_i$ about $\Psi_i$, with higher values of $\alpha$ corresponding to a smaller Fr\'echet variance of $X_i|\Psi_i$. In addition to the basic choices, $\psi = .1I,I,10I,100I,\Bar{X}$, of the shrinkage point in $\delta_{JS}$, the Bayes risk is also computed for two other variants of $\delta_{JS}$. The first variant uses the optimal shrinkage point which by the results in Section \ref{Sec5.3} is $\mu \coloneqq E_2\theta$. The second variant is the best shrinkage estimator that uses the same optimal shrinkage point $\mu$ but also uses the fixed, optimal shrinkage weight given by \eqref{bayesoptimalweight}. Note that both the Bayes risk of $X$ and the Bayes risk of the best shrinkage estimator do not depend on $n$.

 \begin{figure}[h]
    \centering
    \begin{subfigure}[t]{.32\textwidth}
    \centering
  \includegraphics[width = \textwidth]{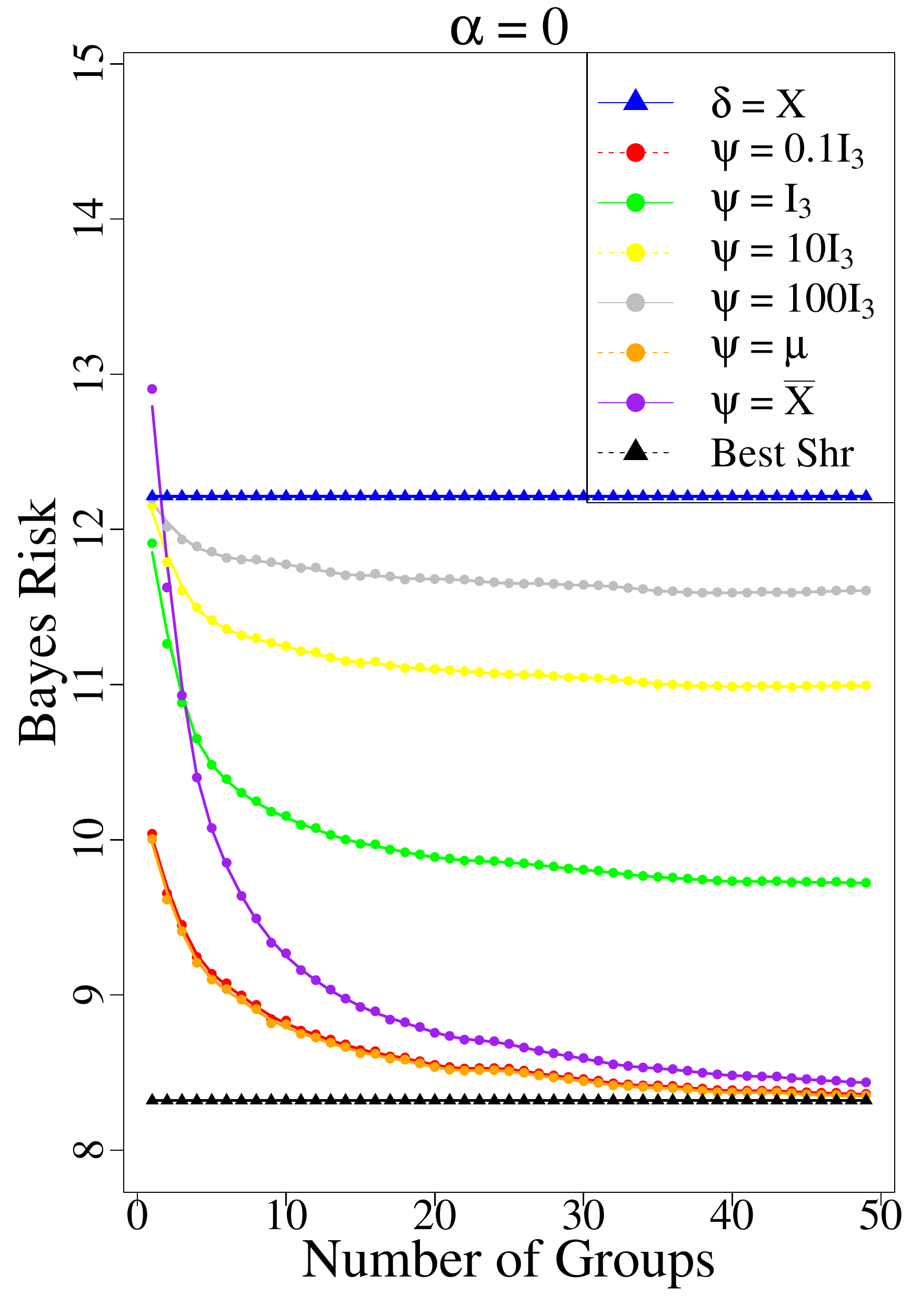}
\end{subfigure}
\begin{subfigure}[t]{.32\textwidth}
\centering        
\includegraphics[width = \textwidth]{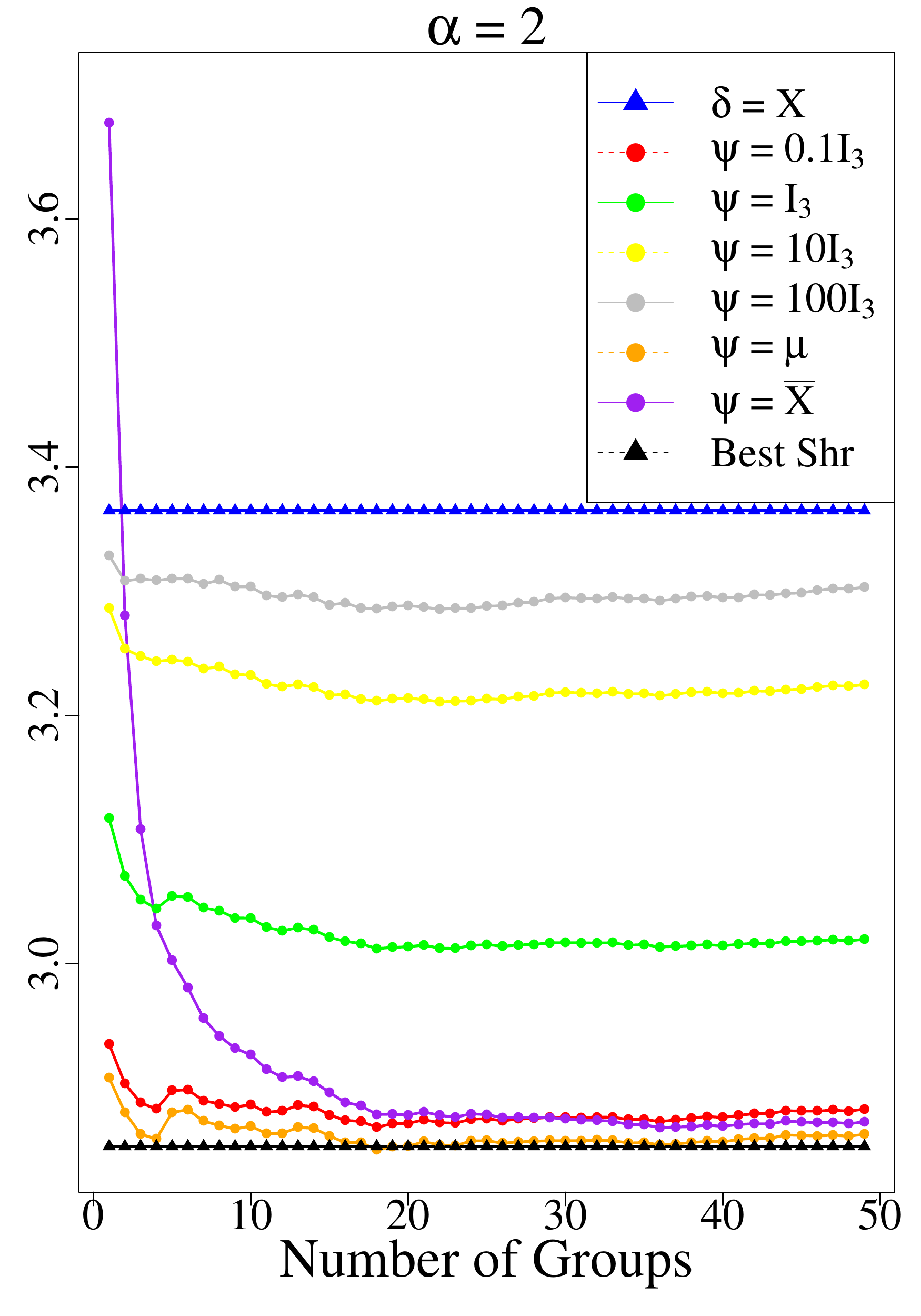}
\end{subfigure}
\begin{subfigure}[t]{.32\textwidth}
\centering        
\includegraphics[width = \textwidth]{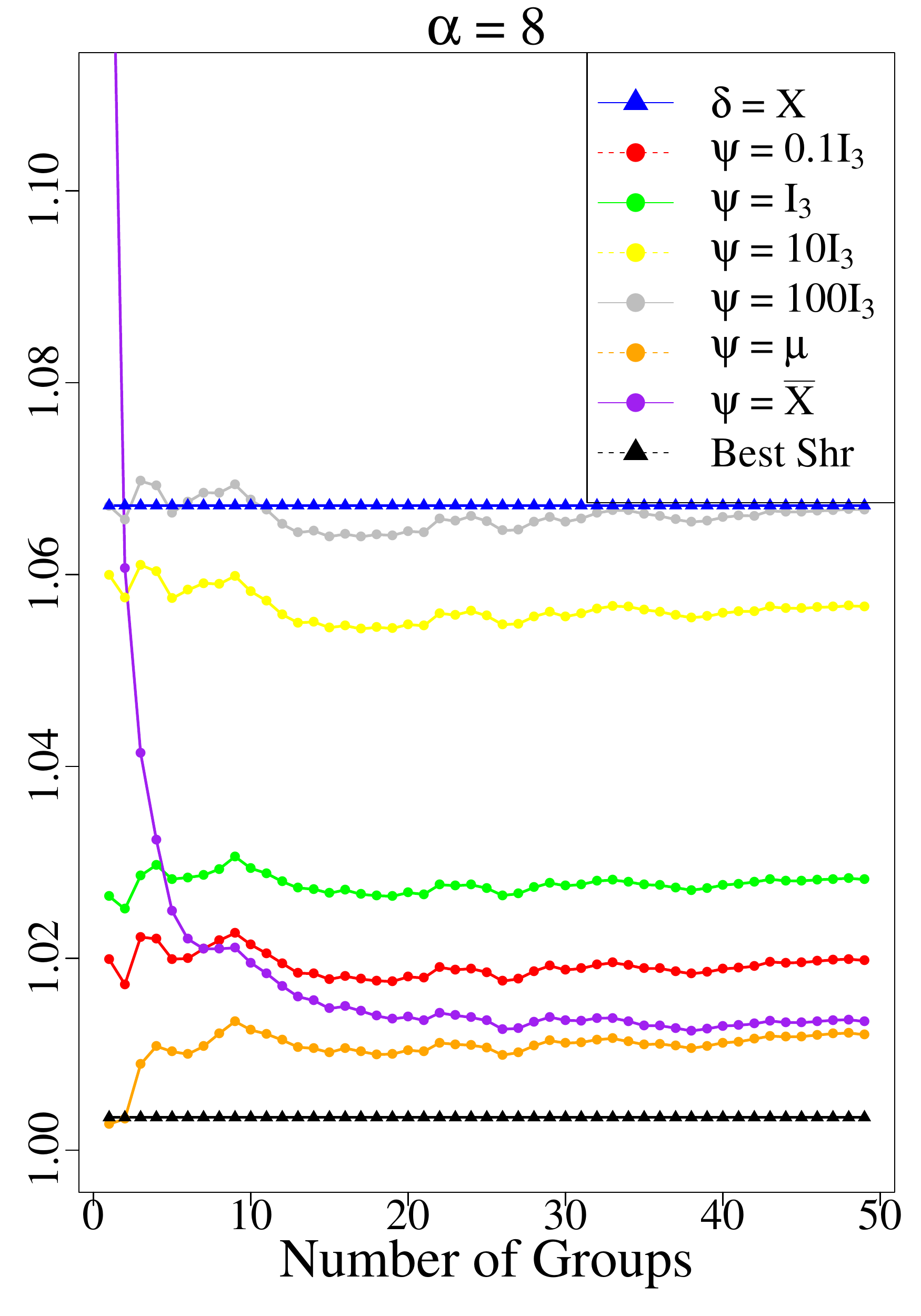}
\end{subfigure}
\caption{Bayes Risk of $\delta_{JS}$ as a function of $n$.}
\label{fig2}
\end{figure}
 
Figure \ref{fig2} illustrates the Bayes risk of the James-Stein estimator as a function of $n$ for various choices of the shrinkage point. It is seen that for shrinkage points that are fixed matrices only a small group size, $n \approx 3$, is needed for $\delta_{JS}$ to have smaller Bayes risk than $X$. The James-Stein estimator with the data-dependent $\psi = \widebar{X}$ performs well, even for a modest number of groups. Its Bayes risk is $75\%, 87\%$ and $95\%$ of the Bayes risk of $X$ for $\alpha = (0,2,8)$ respectively and $n = 10$. Asymptotically, this percentage improvement depends on the ratio of the within-group to the between-group Fr\'echet variance as seen in \eqref{asymptriskofJS}. In addition, its Bayes risk approaches the minimum shrinkage risk, as expected from the discussion Section \ref{Sec5.3}. The estimator that uses the $\psi = \mu$, as its shrinkage point also has a Bayes risk converging to the minimum shrinkage risk. This estimator outperforms the adaptive James-Stein estimator since the optimal shrinkage point is given, unlike in the adaptive James-Stein estimator where $\mu$ has to be estimated by $\bar{X}$.

\subsection{Metric Tree Spaces}
The weighted graph of a tree has the geometry of a Hadamard space under the shortest path metric. As long as the tree has a vertex with degree greater than two the resulting metric tree space has negative Alexandrov curvature. The theoretical results of previous sections  suggest that this negative curvature makes the geodesic James-Stein estimator particularly effective as $\delta_{JS}$ will tend to undershrink relative to the optimal amount of shrinkage. This is corroborated by the numerical results of this section. 

Consider the graph of a tree $\mathcal{T} = (\mathcal{V},\mathcal{E})$ that has an associated weight function, $w:\mathcal{E} \rightarrow \mathbb{R}^+$, on its edges. Draw this tree in $\mathbb{R}^2$ so that every edge $e$ is a straight line with length $w(e)$. The metric tree space $(\mathcal{T}_w,d)$ is the set of points in this drawing. Distances between points in $\mathcal{T}_w$ are given by the shortest paths within the drawing. For example, the distance between vertices is the shortest weighted path between them. More formally, let $\mathbb{R}_e \coloneqq ([0,w(e)],e) \subset \mathbb{R} \times \mathcal{E}$ be intervals in $\mathbb{R}$ tagged by $e \in \mathcal{E}$ and take $\pi_1,\pi_2: \mathcal{E} \rightarrow \mathcal{V}$ to be maps that identify  $\pi_1(e)$ and $\pi_2(e)$ with the two vertices associated with $e$ in an arbitrary order. The metric tree space is the quotient metric space $ \big(\coprod_{e \in \mathcal{E}}  \mathbb{R}_e \coprod \mathcal{V}\big)/ \sim$ where the equivalence relation identifies $(0,e) \in \mathbb{R}_e$ with $\pi_1(e) \in \mathcal{V}$ and $(w(e),e) \in \mathbb{R}_e$ with  $\pi_2(e) \in \mathcal{V}$ \cite{Bridson}. 
Each $\mathbb{R}_e$ in the quotient is equipped with the Euclidean metric. 

To see that the CAT(0) inequality holds in $(\mathcal{T}_w,d)$, if three points $x,y,z$ all lie on the same geodesic so that without loss of generality $z \in [x,y]$ then as $[x,y]$ is isometric to a Euclidean interval the CAT(0) inequality is satisfied. If $x,y,z$ do not all lie on the same geodesic then the comparison triangle looks like the tripod in Figure \ref{tripod} up to  differences in edge lengths from the central vertex. It is visually apparent that this triangle is skinnier than the corresponding Euclidean triangle so the CAT(0) inequality is satisfied.  

The simulations in this section will be performed on the metric tree $\mathcal{X}$ that has countably many vertices, each having degree $3$, where all edges in $\mathcal{X}$ have length one. Suppose that a particle in $\mathcal{X}$ starts at some vertex $\theta_i$, and jumps to adjacent vertices according to a simple symmetric random walk that is run for $k_{\sigma^2}$ iterations. That is, for each step of the random walk, the particle has a $1/4$ probability of moving to one of the $3$ adjacent vertices and a $1/4$ probability of not moving. After observing the positions, $X_i, \; i = 1,\ldots,n$, that $n$ different particles end up at we are tasked with simultaneously estimating the starting position of each particle under squared distance loss. By symmetry considerations, $\theta_i$ is the Fr\'echet mean of $X_i$ in the metric tree $\mathcal{X}$. It is further assumed that $k_{\sigma^2}$ is known, so that the Fr\'echet variance of $X_i$ can explicitly be calculated. A prior distribution is placed on the Fr\'echet means $\theta_i$ so that $(\theta_1,\ldots,\theta_n)$ has the distribution that results from running $n$ independent symmetric random walks each starting at $\mu$ for $k_{\tau^2}$ steps.

The Bayes risk of the geodesic James-Stein estimator is computed by averaging the values of $L(\theta,\delta_{JS}(X))$ over independent samples of $(X,\theta)$ from the distribution described above. Table \ref{metrictreetab} provides the ratio of risks of the James-Stein estimator to the Fr\'echet variance for various values of the shrinkage point and values of $k_{\sigma^2}/k_{\tau^2}$, which is a proxy for the ratio of the within group variance to the between group variance. The value of $k_{\tau^2}$ is fixed at $15$ throughout, while the value of $k_{\sigma^2}$ ranges from $1$ to $30$. A gradient based algorithm, detailed in Appendix \ref{appC},
is used to compute the sample Fr\'echet mean $\widebar{X}$ used in the data-dependent shrinkage estimator. Symmetry considerations show that the oracle shrinkage estimator that minimizes the Bayes risk is $[X,\mu]_{\tilde{t}}$ where $\tilde{t}$ is given by \eqref{bayesoptimalweight}.

The results in Table \ref{metrictreetab} are striking in that only two groups are needed for $\delta_{JS}$ to have a noticeably lower Bayes risk than $X$. For example, when $k_{\sigma^2} = k_{\tau^2} = 15$ the Bayes risk of the adaptive shrinkage estimator is less than half of that of $X$.  Even when the shrinkage point is chosen very poorly so that $d(\psi,\mu) = 32$, the geodesic James-Stein estimator still outperforms $X$. As $k_{\tau^2} = 15$, every possible value of $\theta$ must have $d(\theta,\mu) \leq 15$, so a shrinkage point with $d(\psi,\mu) = 32$ is not even a possible value of any of the $\theta_i$'s. Unlike the log-Euclidean example, there is a sizeable gap between the performance of the oracle shrinkage estimator and the data dependent shrinkage estimator for a modest number of groups. For various choices of $k_i$, the minimum shrinkage risk ranges from $70\%$ to $50\%$ of the adaptive shrinkage risk when $n = 50$. This gap is explained by the fact that the bias-variance inequality $\rho(X,\psi)^2 > \sigma^2 + d(\theta,\psi)^2$ is a strict inequality due to the negative curvature of the space. The shrinkage weight $1 \wedge \sigma^2/\rho(X,\psi)^2$ in $\delta_{JS}$ tends to undershrink relative to the optimal shrinkage estimator.  

\begin{table}
\centering
\begin{tabular}[t]{|l||c|c|c|c|c|c|c|c|}
\hline
    \multirow{2}{*}{$k_{\sigma^2}/k_{\tau^2}$} & \multicolumn{8}{|c|}{Value of $d(\psi,\mu)$} \\
    \cline{2-9}
    & 0 & 1 & 4 & 8 & 16 & 32 & $\psi = \widebar{X}$ & Oracle
    \\
\hline
\hline
1/15 & 0.750 & 0.766 & 0.841 & 0.884 & 0.930 & 0.964 & 0.736 & 0.558\\
\hline
1/3 & 0.569 & 0.592 & 0.656 & 0.728 & 0.821 & 0.877 & 0.624 & 0.305\\
\hline
2/3 & 0.461 & 0.463 & 0.545 & 0.607 & 0.717 & 0.825 & 0.526 & 0.200\\
\hline
1 & 0.373 & 0.381 & 0.472 & 0.538 & 0.646 & 0.766 & 0.445 & 0.160\\
\hline
4/3 & 0.323 & 0.335 & 0.400 & 0.463 & 0.601 & 0.730 & 0.395 & 0.116\\
\hline
5/3 & 0.279 & 0.298 & 0.366 & 0.434 & 0.557 & 0.689 & 0.334 & 0.084\\
\hline
2 & 0.242 & 0.258 & 0.320 & 0.386 & 0.494 & 0.647 & 0.298 & 0.072\\
\hline
\end{tabular}
\caption{Values of $E\big(R(P_\theta,\delta_{JS})\big)/\sigma^2$ for $n = 2$ groups.}
\label{metrictreetab}
\end{table}

The frequentist domination result of Corollary \ref{boundeddiamcor} is applicable here for fixed $k_{\sigma^2},k_{\tau^2}$ if it assumed that the possible starting points, $\theta_i$ of each $X_i$ particle all lie in a bounded set of $\mathcal{T}$. The asymptotic domination result of Theorem \ref{freqasmypthm} applies here without any restrictions on the $\theta_i$'s. Like the classical Gaussian James-Stein result, these results are somewhat counterintuitive. It would appear like the best estimate of the starting positions of several particles that move symmetrically and independently would be the positions where they end up at, $X$. Theorem \ref{freqasmypthm} shows that asymptotically it is possible to do better by using $\delta_{JS}$ even though no relationship is assumed between any of the particles.  

\section{Discussion}
In this article we have primarily considered the risk properties of the geodesic James-Stein estimator for multiple Fr\'echet means. The primary result of this work, Theorem \ref{freqdomthm}, shows that under mild conditions the geodesic James-Stein estimator outperforms $X$ in a simultaneous Fr\'echet mean estimation problem if there are enough groups present and the shrinkage point is reasonably chosen. It is the non-positive Alexandrov curvature of the metric space that forms the foundation of this result, as it implies that the squared distance function is metrically convex.

One may wonder if the results of this article can be extended to arbitrary metric spaces. In general the answer is no. To see this consider the sphere $\mathbb{S}^2 \subset \mathbb{R}^{3}$ with its intrinsic, angular metric. The squared distance metric on the sphere is not metrically convex due to its positive sectional, and thus Alexandrov, curvature. For example, any two points $x,y$ that lie on the equator of the sphere have $d([x,y]_t,N) = d(x,N) = d(y,N)$ for all $t \in [0,1]$ where $N$ is the north pole. As a result, no point of the geodesic $[x,y]$ is closer to $N$ than $x$ itself. A more extreme example on $\mathbb{S}^1$ is presented in Appendix \ref{appB} where for a certain $\psi$ and distribution of $X$, $[X,\psi]_t$ is has a large risk than $X$ for all $t > 0$. As $\mathbb{S}^1$ is compact, Corollary \ref{boundeddiamcor} fails to hold in a general metric space. Shrinkage may still be beneficial under specific circumstances. In the case of a Riemannian manifold, if a distribution is concentrated in a small enough region of the manifold, the effect of curvature on the metric will not be pronounced and results from the Euclidean case will approximately apply. If reliable prior information, suggesting that $E_2X$ is close to $\psi$, is available then the shrinkage estimator $[X,\psi]_t$ will likely have reasonable performance even if the metric space has positive Alexandrov curvature.

Another extension of the geodesic James-Stein estimator presented here would be to cases where $\sigma^2$ is unknown and a plug-in estimator is used for $\sigma^2$ in the expression for the geodesic James-Stein estimator. The theoretical properties of such an estimator are more complex  because multiple observations per group are required to obtain an estimate of $\sigma^2$. A property like the Hadamard bias-variance inequality will no longer be applicable since the sample Fr\'echet means of i.i.d observations may not be unbiased for the underlying Fr\'echet mean. Results from \cite{Gutmann1, Gutmann2} further show that there is no Stein paradox for a family of distributions with finite support. More specifically,  admissible estimators for individual decision problems remain admissible when combined into an estimator for the joint decision problem whose loss function is the sum of the losses for the individual problems. For example, if $X_i \sim Bin(n_i,\theta_i)$ then $(X_1,\ldots,X_n)$ is admissible for estimating $(\theta_1,\ldots,\theta_n)$ under squared error loss because $X_i$ is admissible for estimating $\theta_i$. This shows that Corollary \ref{boundeddiamcor} will not hold in general if $\sigma^2$ is unknown, since the estimator $X$ is admissible in this binomial example. We again remark that $\sigma^2$ does not have to be known exactly in order to use $\delta_{JS}$. Rather, all that is needed is a non-zero lower bound on $\sigma^2$ from which this lower bound can be used in place of $\sigma^2$ in \eqref{jscomponents}. All the theoretical results in in Sections \ref{Sec4} and \ref{Sec5} will apply to the James-Stein estimator that uses such a lower bound, as shown by \eqref{lowervarbdrisk}.

The hierarchical model introduced in Section \ref{Sec4} of this article represents one of the most basic Fr\'echet mean and variance structures possible on metric space valued data. Recent work on Fr\'echet regression \cite{MullerRegression} and geodesic regression \cite{Fletcher} provide examples of reasonable Fr\'echet mean functions of a Euclidean covariate for metric space valued data. In these works the mean functions depend on more general covariates in $\mathbb{R}^k$, rather than just indicator functions of group membership. Another area of recent interest is modelling the joint distributions of random objects on metric spaces. The Bayesian hierarchical model of Section \ref{Sec5} provides a basic example of this, for if multiple observations were obtained within each group, then observations within the same group are more ``correlated" with each other than observations in different groups.  Various notions of covariance on metric spaces have been proposed in \cite{Lyons, MullerCov, Rizzo}. There is substantial scope for the development of parametric and non-parametric models that incorporate these notions of covariance and permit tractable inference. The geodesic James-Stein estimator solves the simple weighted Fr\'echet mean problem, $\delta_{JS,i} = \argm_{z \in \mathcal{X}} \big(1-w(X)\big) d(X_i,z)^2 + w(X) d(\psi,z)^2$. It is anticipated that a typical inferential procedure for estimating the Fr\'echet means of correlated metric space data will result in solving similar weighted sample Fr\'echet mean problems.

\begin{appendices}
\section{Proofs}\label{appA}

\begin{customlemma}{1}
\label{cheblemma}
For $X \in \mathcal{P}_m$ we have $  E\big((d(X,\psi)^2 - \rho(X,\psi)^2)^{2k}\big) \leq  C_k$ where $C_k = O\big(d(\theta,\psi)^{2k}\big)$ and in the case of $k = 1$, $C_1 = O\big(d(\theta,\psi)^2/n\big)$. 

\noindent In particular,  $P\big(\vert d(X,\psi)^2 - \rho(X,\psi)^2)\vert > t\big) \leq C_k/t^{2k}.$
\end{customlemma}
\begin{proof}
 For $k = 1$, we get
\begin{align*}
    E\big[\big(d(X,\psi)^2 - \rho(X,\psi)^2\big)^{2}\big] = \frac{1}{n^2}\sumonn E\big[\big(d_i(X_i,\theta_i)^2 - E[d_i(X_i,\theta_i)^2]\big)^2\big],
\end{align*}
while for $k > 1$ we use the convexity of the function $x \rightarrow x^{2k}$ to get
\begin{align*}
        E\big[\big(d(X,\psi)^2 - \rho(X,\psi)^2\big)^{2k}\big] \leq \frac{1}{n}\sumonn E\big[\big(d_i(X_i,\psi_i)^2 - \rho_i(X_i,\psi_i)^{2}\big)^{2k}\big].
\end{align*}
The triangle and reverse triangle inequalities can be used on both $d_i(X_i,\psi_i)^2$ and $\rho(X_i,\psi_i)^2$. By considering the cases $d_i(X_i,\psi_i)^2 - \rho_i(X_i,\psi_i)^2 > 0$ and $d_i(X_i,\psi_i)^2 - \rho_i(X_i,\psi_i)^2 \leq 0$ this summand can be bounded repeated uses of the triangle inequalities and convexity,
\begin{align*}      E\big[\big(d_i(X_i,&\psi_i)^2 - \rho_i(X_i,\psi_i)^2\big)^{2k}\big]
    \\
     \leq & \; E\big[\big( d_i(X_i,\theta_i)^2 - \sigma_i^2 + 2d_i(X_i,\theta_i)d_i(\theta_i,\psi_i)\big)^{2k}\big]
    +
    \\
    & \; E\big[\big(d_i(X_i,\theta_i)^2 - \sigma_i^2 - 2d_i(X_i,\theta_i)d_i(\theta_i,\psi_i) - 2E[d_i(X_i,\theta_i)d_i(\theta_i,\psi_i)]\big)^{2k}\big]
    \\
    \leq & \; 2^{2k-1}\bigg(E\big[\big(d_i(X_i,\theta_i)^2 - \sigma_i^2\big)^{2k}\big] + 2^{2k}d_i(\theta_i,\psi_i)^{2k}E\big[d_i(X_i,\theta_i)^{2k}\big]\bigg) + 
    \\
    & \; 3^{2k - 1}\bigg(E\big[\big(d_i(X_i,\theta_i)^2 - \sigma_i^2\big)^{2k}\big] +
    2^{2k}d_i(\theta_i,\psi_i)^{2k}\big(E\big[d_i(X_i,\theta_i)^{2k}\big] +
    \\
    & E\big[d_i(X_i,\theta_i)\big]^{2k}\big)\bigg).
\end{align*}
This implies that
\begin{align*}
    E\big[\big(d(X,\psi)^2 - \rho(X,\psi)^2\big)^{2k}\big] \leq & (2^{2k-1} + 3^{2k-1})(m_{4k} + m_2^{2k}) + 
    \\
    & 3^{4k}(m_{2k} + m_1^{2k}) d(\theta,\psi)^{2k},
\end{align*}
from which Chebychev's inequality gives
\begin{align*}
       P\big(\vert d(X,\psi)^2 - \rho(X,\psi)^2)\vert > t\big)  \leq & \frac{1}{t^{2k}} \big[(2^{2k-1} +
       3^{2k-1})(m_{4k} + m_2^{2k}) + 
       \\
       & 3^{4k}(m_{2k} + m_1^{2k}) d(\theta,\psi)^{2k}\big]
       \nonumber
       \\
        \coloneqq & \frac{C_{k}}{t^{2k}}.
\end{align*}
For $k = 1$ this expression can be multiplied by $\frac{1}{n}$.
\end{proof}
The bound in Lemma \ref{cheblemma} is especially useful because if $t = d(\theta,\psi)^2$ then the resulting expression will be $O(d(\theta,\psi)^{-2k})$. The key step which makes this $O\big(d(\theta,\psi)^{-2k}\big)$ rate possible is the use of the triangle and reverse triangle inequalities. This makes it so that $d(\theta_i,\psi_i)$ only has an exponent of $1$ in the decomposition of the expression  $d_i(X_i,\theta_i)^2 - \rho_i(X_i,\theta_i)^2$.    

\begin{customthm}{3}
Let $\{a_n\}$ be a sequence with $a_n \rightarrow \infty$ and take $P \in \mathcal{P}_m^{(n)}$ to be any distribution on $\mathcal{X}^{(n)}$ with a Fr\'echet mean $\theta^{(n)}$ that satisfies $ d(\theta^{(n)},\psi^{(n)})^2 \leq n/a_n$.   
There exists an $n^*(m,\{a_n\})$ such that if $n \geq n^*$ then $R(P,\delta_{JS}) < R(P,X^{(n)})$.
\end{customthm}
\begin{proof}
The proof is split into two parts; the first is the case where $\sup_n d(\theta^{(n)},\psi^{(n)})^2 = M < \infty$ and the second case has $\sup_n d(\theta^{(n)},\psi^{(n)})^2 = \infty$. In the first case we will show the existence of an $N_1(m,M)$ where $R(P,\delta_{JS}) < R(P,X^{(n)})$ for $n \geq N_1$. In the second case it will be shown that there exists an $\Tilde{M}$ and an $\epsilon$ such that $R(P,\delta_{JS}) < R(P,X^{(n)})$ whenever $d(\theta^{(n)},\psi^{(n)})^2 > \Tilde{M}$ and $d(\theta^{(n)},\psi^{(n)})^2/n < \epsilon$. As we have assumed that $d(\theta^{(n)},\psi^{(n)})^2/n \leq a_n^{-1}$ there exists an $N_2(m,\{a_n\})$ such that $a_n^{-1} < \epsilon$ and $d(\theta^{(n)},\psi^{(n)})^2 > \tilde{M}$ for $n \geq N_2$. Then for $n \geq n^*(m,\{a_n\}) \coloneqq \max[N_1(m,\Tilde{M}),N_2(m,\{a_i\})]$ the theorem then follows because the first case applies if $d(\theta^{(n)},\psi^{(n)})^2 \leq \tilde{M}$ and the second applies if $d(\theta^{(n)},\psi^{(n)})^2 > \tilde{M}$. 

\textit{Proof of the bounded case:}
Assume that $\sup_n d(\theta^{(n)},\psi^{(n)})^2 \leq M$ for some $M > 0$. Also assume $d(\theta^{(n)},\psi^{(n)})^2 > 0$ as otherwise the shrinkage estimator clearly outperforms $X^{(n)}$. A bound for $R(P,\delta_{JS})$ is obtained by bounding the expectations of each of the terms $(a),(b),(c)$ in expression \eqref{lossbound}.
The probability $P(A^c) = P\big(d(X^{(n)},\psi^{(n)})^2 < \sigma^2\big)$ is bounded using Chebychev's inequality and the Hadamard bias-variance inequality,
\begin{align}
    P\big(s > d(X^{(n)},\psi^{(n)})^2\big) & = P\big(s - \rho(X^{(n)},\psi^{(n)})^2 > d(X^{(n)},\psi^{(n)})^2 - \rho(X^{(n)},\psi^{(n)})^2\big) \nonumber
    \\
    & \leq P\big(d(\theta^{(n)},\psi^{(n)})^2 + \sigma^2 - s < \rho(X^{(n)},\psi^{(n)})^2 - d(X^{(n)},\psi^{(n)})^2\big) \nonumber
    \\
    & \leq \frac{C_1}{n\big(d(\theta^{(n)},\psi^{(n)})^2 + \sigma^2 - s\big)^2}. \label{simplechebbound}
\end{align}
This inequality holds for all $s < \sigma^2 + d(\theta^{(n)},\psi^{(n)})^2$. The number $C_1$, taken from Lemma \ref{cheblemma}, is independent of $n$ by virtue of $d(\theta^{(n)},\psi^{(n)})^2$ being bounded by $M$.
Taking $s = \sigma^2$, this shows that $E[(c)] = E(I_{A^c})d(\theta^{(n)},\psi^{(n)})^2 \leq C_1/n$. Using \eqref{simplechebbound} again we bound $E[(b)]/d(\theta^{(n)},\psi^{(n)})^2 =  E\big(I_A w(X^{(n)})\big)$ by
\begin{align*}
    E\big(I_A\frac{\sigma^2}{d(X^{(n)},\psi^{(n)})^2}\big)  \leq & P\big(d(\theta^{(n)},\psi^{(n)})^2 + \sigma^2/2 >  d(X^{(n)},\psi^{(n)})^2\big) +
    \\
    & \frac{\sigma^2}{d(\theta^{(n)},\psi^{(n)})^2 + \sigma^2/2}
    \\
     \leq & \frac{4C_1}{n\sigma^4} + \frac{\sigma^2}{d(\theta^{(n)},\psi^{(n)})^2 + \sigma^2/2}.
\end{align*}
The term $E[(a)]$ is handled by the inequality
\begin{align*}
    E\big[I_A\big(1 - w(X^{(n)})\big)\big(d(X^{(n)},\theta^{(n)})^2 - \sigma^2\big)\big] & \leq E[\big(d(X^{(n)},\theta^{(n)})^2 - \sigma^2\big)^2\big]^{1/2}
    \\
    & = \frac{\big(\frac{1}{n}\sumonn E\big[\big(d_i(X_i,\theta_i)^2 - \sigma_i^2\big)^2\big]\big)^{1/2}}{\sqrt{n}}
    \\
    & \leq \frac{(m_4 +  m_2^2)^{1/2}}{\sqrt{n}}.
\end{align*}
Taken together, these inequalities yield the risk upper bound
\begin{align}
    R(P,\delta_{JS}) \leq & \; E[(a)] + E[(b)] + E[(c)] \nonumber
    \\
    \leq &\; \frac{1}{\sqrt{n}} \bigg[\frac{4C_1M}{\sqrt{n}\sigma^4} + \frac{C_1}{\sqrt{n}} + (m_4 + m_2^2)^{1/2}\bigg] + \bigg[\frac{M}{M + \sigma^2/2}\bigg] \sigma^2.
    \label{bddriskbound}
\end{align}
This is less than $\sigma^2$ as long as $n \geq N_1(m,M)$ is large enough so that the $O(n^{-1/2})$ term in \eqref{bddriskbound} is less than $\sigma^2/(2M + \sigma^2)$. 

\textit{Proof of the unbounded case:} Here we assume that  $d(\theta^{(n)},\psi^{(n)})^2 \rightarrow \infty$ but $d(\theta^{(n)},\psi^{(n)})^2/n \rightarrow 0$. By the reasoning in \eqref{approxriskbdd} we expect that the benefit of shrinkage is approximately $\frac{\sigma^2}{\sigma^2 + d(\theta^{(n)},\psi^{(n)})^2}\sigma^2$ which is $O\big(d(\theta^{(n)},\psi^{(n)})^{-2}\big)$. Thus we seek to send all other terms in the risk bound to zero at rates faster than this. We immediately have a bound on $E[(c)] = E\big(I_{A^c}d(\theta^{(n)},\psi^{(n)})^2\big)$ since $P(A^c) \leq  P\big(\vert d(X^{(n)},\psi^{(n)})^2 - \rho(X^{(n)},\psi^{(n)})^2 \vert > d(\theta^{(n)},\psi^{(n)})^2 \big)$ and by taking $k = 3$ in Lemma \ref{cheblemma} we find that $E\big(I_{A^c}d(\theta^{(n)},\psi^{(n)})^2\big) = O\big(d(\theta^{(n)},\psi^{(n)})^{-4}\big)$. 

Rewriting $I_A$ as $1- I_{A^c}$ in the term $E[(a)] = E\big[I_{A}\big(1-w(X^{(n)})\big)\big(d(X^{(n)},\theta^{(n)})^2 - \sigma^2\big)\big]$ we find that
\begin{align*}
E\big[I_{A}\big(1-w(X^{(n)})\big)\big(d(X^{(n)},\theta^{(n)})^2 - \sigma^2\big)\big] = & \;  E\big[I_{A^c}\big(\sigma^2 - d(X^{(n)},\theta^{(n)})^2\big)\big] +
\\
& \; E\big[I_A w(X^{(n)}) \big(\sigma^2 - d(X^{(n)},\theta^{(n)})^2\big)\big].
\end{align*}
By Cauchy-Schwartz,
\begin{align*}
    E\big[I_{A^c}\big(\sigma^2 - d(X^{(n)},\theta^{(n)})^2\big)\big] & \leq \big(P(A^c)E\big[\big(\sigma^2 - d(X^{(n)},\theta^{(n)})^2\big)^2\big]\big)^{1/2}
    \\
    & \leq P(A^c)^{1/2}(m_4 + \sigma^4)^{1/2}.
\end{align*}
This term is $O\big(d(\theta^{(n)},\psi^{(n)})^{-4}\big)$ since $P(A^c)$ is $O\big(d(\theta^{(n)},\psi^{(n)})^{-4}\big)$ by taking $k = 4$ in Lemma \ref{cheblemma}. Next we bound the term $E\big[I_A w(X)\big(\sigma^2 - d(X^{(n)},\theta^{(n)})^2\big)\big]$. Let $B \coloneqq \{X^{(n)} : d(X^{(n)},\psi^{(n)})^2 - \rho(X^{(n)},\psi^{(n)})^2 \geq -d(\theta^{(n)},\psi^{(n)})^2/2\}$ then
\begin{align*}
    E\big[I_A & w(X^{(n)})\big(\sigma^2 -  d(X^{(n)},\theta^{(n)})^2\big)\big]  \
    \\
    & \leq  \sigma^2 E\bigg[\frac{\vert\sigma^2 - d(X^{(n)},\theta^{(n)})^2\vert}{d(X^{(n)},\psi^{(n)})^2 - \rho(X^{(n)},\psi^{(n)})^2 + d(\theta^{(n)},\psi^{(n)})^2 +\sigma^2}I_B\bigg] + \sigma^2 E(I_{B^c})
    \\
    & \leq \frac{2\sigma^2}{d(\theta^{(n)},\psi^{(n)})^2}E(\vert \sigma^2 - d(X^{(n)},\theta^{(n)})^2\vert ) + \sigma^2P(B^c)
    \\
    &  \leq \frac{2\sigma^2(m_4 + m_2^2)^{1/2}}{d(\theta,\psi)^2\sqrt{n}} + \sigma^2P(B^c).
\end{align*}
The first term is $O\big(d(\theta^{(n)},\psi^{(n)})^{-2}n^{-1/2}\big)$ while $P(B^c)$ is $O\big(d(\theta^{(n)},\psi^{(n)})^{-4}\big)$ by Lemma \ref{cheblemma} so the entire expression is $O\big(d(\theta^{(n)},\psi^{(n)})^{{-2.5}}\big)$ by the assumption that $d(\theta^{(n)},\psi^{(n)})^2 = o(n)$. The remaining term, $E[(b)] = \sigma^2 d(\theta^{(n)},\psi^{(n)})^2 E\big(I_A\frac{1}{d(X^{(n)},\psi^{(n)})^2}\big)$, can be decomposed as
\begin{align*}
        E\big(& I_A\frac{1}{d(X,\psi)^2}\big) 
         =\int_{0}^{\sigma^{-2}}P\big(\frac{1}{d(X^{(n)},\psi^{(n)})^2} > t\big)dt
    \\
    & = \int_0^{(d(\theta^{(n)},\psi^{(n)})^2 + \frac{\sigma^2}{2})^{-1}} P\big(\frac{1}{d(X^{(n)},\psi^{(n)})^2} > t\big)dt  
    + 
    \\
    & \;\;\;\; \int_{(d(\theta^{(n)},\psi^{(n)})^2 + \frac{\sigma^2}{2})^{-1}}^{\sigma^{-2}} P\big(\frac{1}{t}  > d(X^{(n)},\psi^{(n)})^2 \big)dt
    \\
    & \leq \frac{1}{d(\theta^{(n)},\psi^{(n)})^2 + \frac{\sigma^2}{2}} \  +
    \\
    &    \int_{(d(\theta^{(n)},\psi^{(n)})^2+ \frac{\sigma^2}{2})^{-1}}^{\sigma^{-2}} P\bigg(\frac{1}{t} - \rho(X^{(n)},\psi^{(n)})^2  > d(X^{(n)},\psi^{(n)})^2 - \rho(X^{(n)},\psi^{(n)})^2\bigg)dt.
\end{align*}
If the second term tends to zero at a rate faster than
 $O\big(d(\theta^{(n)},\psi^{(n)})^{-2}\big)$ then this will complete the proof. By Chebychev's inequality we find that for $t \geq (d(\theta^{(n)},\psi^{(n)})^2 + \sigma^2)^{-1}$,
\begin{align}
    P\big(\frac{1}{t} - \rho(X^{(n)},\psi^{(n)})^2 > d(X^{(n)},& \psi^{(n)})^2  - \rho(X^{(n)},\psi^{(n)})^2\big)  \nonumber
    \\
    & \leq  \frac{E\big[\big(d(X^{(n)},\psi^{(n)})^2 - \rho(X^{(n)},\psi^{(n)})^2\big)^2\big]}{\big(d(\theta^{(n)},\psi^{(n)})^2 + \sigma^2 - 1/t\big)^2}. 
    \label{rationalbound}
\end{align}
The numerator of this expression is $O(n^{-1}d(\theta^{(n)},\psi^{(n)})^{2})$ by Lemma \ref{cheblemma} with $k = 1$. To ease notation let $a = d(\theta^{(n)},\psi^{(n)})^2 + \sigma^2$. Integrating the denominator of \eqref{rationalbound} gives
 \begin{align*}
     \int \frac{t^2}{(at-1)^2}dt & = \frac{1}{a^2}\bigg(t + \frac{2}{a} ln(at-1) - \frac{1}{a(at-1)} \bigg).
 \end{align*}
 The integral of \eqref{rationalbound} becomes
 \begin{align*}
     \int_{(d(\theta^{(n)},\psi^{(n)})^2 +  \frac{\sigma^2}{2})^{-1}}^{\sigma^{-2}} & P\big(\frac{1}{t}  > d(X^{(n)},\psi^{(n)})^2 \big)dt 
     \\
     & \leq \frac{O\big(n^{-1}d(\theta^{(n)},\psi^{(n)})^2\big)}{a^2}\bigg(t + \frac{2}{a} ln(at-1) - \frac{1}{a(at-1)} \bigg) \bigg\vert_{t = (a - \frac{\sigma^2}{2})^{-1}}^{t = \sigma^{-2}}.
 \end{align*}
 It can be checked that the above expression is $O\big(d(\theta^{(n)},\psi^{(n)})^{-2}n^{-1}\big)$. It follows that $\sigma^2 d(\theta^{(n)},\psi^{(n)})^2 E(I_A \frac{1}{d(X^{(n)},\psi^{(n)})^2})$ is $\sigma^2\frac{d(\theta^{(n)},\psi^{(n)})^2}{d(\theta^{(n)},\psi^{(n)})^2 + \sigma^2/2} +  O(n^{-1})$. Putting all of these bounds together shows that $R(\theta,\delta_{JS}) \leq \sigma^2\frac{d(\theta,\psi)^2}{d(\theta,\psi)^2 + \sigma^2/2} + O(n^{-1}) + O\big(d(\theta,\psi)^{-4}\big) + O\big(d(\theta,\psi)^{-2.5}\big)$. It follows that there exists an $\epsilon$ and a $\tilde{M}$ such that $R(P,\delta_{JS}) < \sigma^2$ whenever $d(\theta^{(n)},\psi^{(n)})^2/n < \epsilon$ and $d(\theta^{(n)},\psi^{(n)})^2 > \tilde{M}$ as desired.  
\end{proof}

\begin{customthm}{4}
    Let $X^{(n)} \sim P^{(n)} \in \mathcal{P}_m^{(n)}$ for all $n \in \mathbb{N}$. If $d(\theta^{(n)},\psi^{(n)})^2 \rightarrow \infty$ for a sequence of shrinkage points $\{\psi^{(n)}\}_{n = 1}^{\infty}$, then $\limsup_n R\big(P^{(n)},\delta_{JS}(X^{(n)})\big) = \sigma^2.$ It follows from Theorem \ref{freqdomthm} that $\limsup_n R\big(P^{(n)},\delta_{JS}(X^{(n)})\big) \leq \lim_n R(P^{(n)},X^{(n)})$ for any sequence of $\psi^{(n)}$'s. Additionally, for all $\epsilon > 0$,  $\lim_n P\big(L\big(\theta^{(n)},\delta_{JS}(X^{(n)})\big) > \sigma^2 + \epsilon\big) = 0$.
\end{customthm}
\begin{proof}
From Theorem \ref{freqdomthm}, $E[(a)]= O(n^{-1/2})$ and $E[(c)] = O\big(d(\theta^{(n)},\psi^{(n)})^{-4}\big)$. Therefore,  $\limsup_n R\big(P^{(n)},\delta_{JS}(X^{(n)})\big) = \sigma^2 \limsup_n E\big(I_{A}\frac{d(\theta^{(n)},\psi^{(n})^2}{d(X^{(n)},\psi^{(n)})^2)}\big)$. Applying the reverse triangle inequality to $d(X^{(n)},\theta^{(n)})^2$ gives
\begin{IEEEeqnarray}{rl}
         E\bigg(I_A&\frac{d(\theta^{(n)},\psi^{(n)})^2}{d(X^{(n)},\psi^{(n)})^2}\bigg)  \nonumber
         \\
          \leq & \; E\bigg[ \frac{d(\theta^{(n)},\psi^{(n)})^2}{d(\theta^{(n)},\psi^{(n)})^2 - 2d(\theta^{(n)},\psi^{(n)})d(X^{(n)}, \theta^{(n)}) + d(X^{(n)},\theta^{(n)})^2} \wedge \frac{d(\theta^{(n)},\psi^{(n)})^2}{\sigma^2}\bigg] \nonumber
    \\
     \leq & \; E\bigg[ \frac{d(\theta^{(n)},\psi^{(n)})}{d(\theta^{(n)},\psi^{(n)}) - 2d(X^{(n)},\theta^{(n)}) } \wedge \frac{d(\theta^{(n)},\psi^{(n)})^2}{\sigma^2}\bigg]. \label{asymptbd1}
\end{IEEEeqnarray}

Define $D \coloneqq \{2d(X^{(n)},\theta^{(n)}) > d(\theta^{(n)},\psi^{(n)})^{1/2}\}$ from which an application of Chebychev's inequality shows that 
$P(D)  \leq 2^6m_6 d(\theta^{(n)},\psi^{(n)})^{-3}$. Using this in \eqref{asymptbd1},
\begin{align*}
   E\bigg[& \frac{d(\theta^{(n)},\psi^{(n)})}{d(\theta^{(n)},\psi^{(n)}) - 2d(X^{(n)},\theta^{(n)}) } \wedge \frac{d(\theta^{(n)},\psi^{(n)})^2}{\sigma^2}\bigg]  
   \\
   & \leq E\big[I_{D^c}\frac{d(\theta^{(n)},\psi^{(n)})}{d(\theta^{(n)},\psi^{(n)}) - 2d(X^{(n)},\theta^{(n)}) }\big] +
 P(D) \frac{d(\theta^{(n)},\psi^{(n)})^2}{\sigma^2} 
    \\
     & \leq  \frac{d(\theta^{(n)},\psi^{(n)})}{d(\theta^{(n)},\psi^{(n)}) - d(\theta^{(n)},\psi^{(n)})^{1/2}} +
 \frac{2^6m_6}{\sigma^2d(\theta^{(n)},\psi^{(n)})}
\end{align*}
It follows that $\limsup_n \sigma^2 E\big(I_{A}\frac{d(\theta^{(n)},\psi^{(n)})^2}{d(X^{(n)},\psi^{(n)})^2)}\big) \leq \sigma^2$ as needed.

To show that $\lim_n P\big(L(\theta^{(n)},\delta_{JS}(X^{(n)})\big) > \sigma^2 + \epsilon) = 0$ we split up this probability as
\begin{align*}
    P\bigg[L(\theta^{(n)},\delta_{JS}(X^{(n)})) > \sigma^2 + \epsilon\bigg]  \leq & P\bigg[ (a) + (c) > \frac{\epsilon}{2}\bigg] + 
    \\
    & P\bigg[\sigma^2\big(\frac{d(\theta^{(n)},\psi^{(n)})^2}{d(X^{(n)},\psi^{(n)})^2} - 1\big) > \frac{\epsilon}{2}\bigg]. 
\end{align*}
The limit of the first term tends to zero since the limit of the expectations of $(a)$ and $(c)$ is zero. The second term can be re-written as \begin{align*}
    P\bigg[d(\theta^{(n)},& \psi^{(n)})^2  > (1 + \frac{\epsilon}{2\sigma^2})d(X^{(n)},\psi^{(n)})^2\bigg] \leq 
    \\
    & P\bigg[\frac{\epsilon}{2\sigma^2}\rho(X^{(n)},\psi^{(n)})^2 <
     (1 + \frac{\epsilon}{2\sigma^2})\big(\rho(X^{(n)},\psi^{(n)})^2 - d(X^{(n)},\psi^{(n)})^2\big)\bigg].
\end{align*}
Taking $k =1$ in Lemma \ref{cheblemma}, this  probability is $O(n)$, proving the result. 
\end{proof}

\begin{customthm}{5}
Under the distributional assumptions in \eqref{hierdistassume}, suppose that there is a sequence $a_n \rightarrow \infty$ such that  $ d(\mu^{(n)},\psi^{(n)})^2 \leq n/a_n$. There exists an $n^*(m,l,\{a_n\})$ such that if $n \geq n^*$ then the Bayes risk satisfies $E\big(R(P_{\theta^{(n)}}^{(n)},\delta_{JS})\big) < E\big(R(P_{\theta^{(n)}}^{(n)},X^{(n)})\big)$.
\end{customthm}

\begin{proof}
Conditional on $\theta^{(n)}$, we are able to use the same bounds derived in Theorem \ref{freqdomthm}. Using these bounds and the same proof technique, we first show that if $d(\mu^{(n)},\psi^{(n)})^2 \leq \tilde{M}$ then there exists an $n^*(m,l,\tilde{M})$ with 
$E\big(R(P_{\theta^{(n)}},\delta_{JS})\big) \leq E\big(R(P_{\theta^{(n)}},X)\big)$ whenever $n \geq n^*$. 

 By the risk bound \eqref{bddriskbound} we just need to show that the following quantity can be made to be less than $\sigma^2$
\begin{align*}
        E\big(R(P_{\theta^{(n)}},\delta_{JS})\big) \leq & \frac{1}{\sqrt{n}} \bigg[\frac{4E\big(C_1d(\theta^{(n)},\psi^{(n)})^2\big)}{\sigma^2\sqrt{n}} + \frac{E(C_1)}{\sqrt{n}} + (m_4 + m_2^2)^{1/2}\bigg] + 
        \\
        & E\bigg[\frac{d(\theta^{(n)},\psi^{(n)})^2}{d(\theta^{(n)},\psi^{(n)})^2 + \sigma^2/2}\bigg] \sigma^2.
\end{align*}
Recall that $C_1 = O\big(d(\theta^{(n)},\psi^{(n)})^2/n\big)$, and so both of the terms $E\big(d(\theta^{(n)},\psi^{(n)})^4\big)$ and $E\big(C_1d(\theta^{(n)},\psi^{(n)})^2\big)$ can be bounded above by $ 2^3(\Tilde{M}^2 +  l_4)$. The function

\noindent $d(\theta^{(n)},\psi^{(n)})^2/\big(d(\theta^{(n)},\psi^{(n)})^2 + \sigma^2/2\big)$ is concave and increasing in $d(\theta^{(n)},\psi^{(n)})^2$ so 
\begin{align*}
    E\big(\frac{d(\theta^{(n)},\psi^{(n)})^2}{d(\theta^{(n)},\psi^{(n)})^2 + \sigma^2/2}\big) \sigma^2 &\leq    \frac{E\big(d(\theta^{(n)},\psi^{(n)})^2\big)}{E\big(d(\theta^{(n)},\psi^{(n)})^2\big) + \sigma^2/2} \sigma^2 
    \\
    & \leq \frac{4E\big(d(\theta^{(n)},\mu^{(n)})^2\big) + 4\Tilde{M}}{4E\big(d(\theta^{(n)},\mu^{(n)})^2\big) + 4\Tilde{M} + \sigma^2/2} \sigma^2 < \sigma^2.
\end{align*}
This shows that $E\big(R(P_{\theta^{(n)}},\delta_{JS})\big) < \sigma^2$ for large enough $n$.

Next we want to show that there exists an $\epsilon$ and a $\tilde{M}$ such that $E\big(R(P_{\theta^{(n)}},\delta_{JS})\big) < \sigma^2$ when $d(\mu^{(n)},\psi^{(n)})^2 \geq \tilde{M}$ and $d(\mu^{(n)},\psi^{(n)})^2/n < \epsilon$.  Conditional on $\theta^{(n)}$, the second risk bound found in Theorem \ref{freqdomthm} for an unbounded $d(\theta^{(n)},\psi^{(n)})^2$ shows that \begin{align*}
    R(P_{\theta^{(n)}},\delta_{JS}) \leq \sigma^2 - \frac{c_1}{d(\theta^{(n)},\psi^{(n)})^2} + \frac{c_2}{n} + \frac{c_3}{d(\theta^{(n)},\psi^{(n)})^{2.5}} + \frac{c_4}{d(\theta^{(n)},\psi^{(n)})^4},
\end{align*}
for some positive constants $c_i$, that depend only on the central-moments bounds of $X_i^{(n)}$, $m_j$. The same derivation used in Theorem \ref{freqdomthm} and Lemma \ref{cheblemma} shows that the event $C \coloneqq \{\vert d(\theta^{(n)},\psi^{(n)})^2 - E\big(d(\theta^{(n)},\psi^{(n)})^2\big)\vert >  d(\mu^{(n)},\psi^{(n)})^2/2\}$ has $P(C) = O\big(d(\mu^{(n)},\psi^{(n)})^{-2k}\big)$.  Using Jensen's inequality on $-c_1/d(\theta^{(n)},\psi^{(n)})^2$ we get
\begin{align*}
    E\big(I_{C^c}R(P_{\theta^{(n)}},\delta_{JS})\big)\leq & \sigma^2 - \frac{c_1}{4d(\mu^{(n)},\psi^{(n)})^2 + 4E\big(d(\theta^{(n)},\mu^{(n)})^2\big)} + \frac{c_2}{n} + 
    \\
    & \frac{c_32^{2.5}}{d(\mu^{(n)},\psi^{(n)})^{2.5}} + \frac{c_42^4}{d(\mu^{(n)},\psi^{(n)})^4}.
\end{align*}
Under $C$  we have 
\begin{align*}
    E\big(I_{C}R(P_{\theta^{(n)}},\delta_{JS}) \big) & \leq E\bigg[I_C\max\big(\sigma^2, 4d(\mu^{(n)},\psi^{(n)})^2 + 4d(\mu^{(n)},\theta^{(n)})^2\big)\bigg].
\end{align*}
This term can be made $O\big(d(\mu^{(n)},\psi^{(n)})^{-4}\big)$ by Cauchy-Schwartz and the form of $P(C)$. Thus, $E\big(R(P_{\theta^{(n)}},\delta_{JS})\big) \leq \sigma^2 - O(d(\mu^{(n)},\psi^{(n)})^{-2}) + O(n^{-1}) + O(d(\mu^{(n)},\psi^{(n)})^{-2.5})$ so there exists the desired $\Tilde{M}$ and $\epsilon$. 
\end{proof}

\begin{customthm}{6}
Let $X^{(n)} \sim P^{(n)}_{\theta^{(n)}}, \; n \in \mathbb{N}$ and $E_2X^{(n)} = \theta^{(n)} \sim Q^{(n)}, \; n \in \mathbb{N}$ satisfy the distributional assumptions in \eqref{hierdistassume}. If $d\big(\mu^{(n)},\psi^{(n)}\big)^2 \rightarrow \infty$ for a sequence of shrinkage points $\{\psi^{(n)}\}_{n = 1}^{\infty}$, then $\limsup_n E\big(R(P^{(n)}_{\theta^{(n)}},\delta_{JS})\big) = \lim_n E\big(R(P^{(n)}_{\theta^{(n)}},X^{(n)})\big)$. By Theorem \ref{Bayesdom1}, for any sequence of $\psi^{(n)}$'s, $\limsup_n E\big(R(P^{(n)}_{\theta^{(n)}},\delta_{JS})\big) \leq \lim_n E\big(R(P^{(n)}_{\theta^{(n)}},X^{(n)})\big)$, with strict inequality if $d\big(\mu^{(n)},\psi^{(n)}\big)^2/n = o(1)$. Additionally, we have that for all $\epsilon > 0$, $\lim_n P\big(L(\theta^{(n)},\delta_{JS}) > \sigma^2 + \epsilon\big) = 0$.
\end{customthm}

\begin{proof}
We first show that $\limsup_n E\big(R(P^{(n)}_{\theta^{(n)}},\delta_{JS}) \big) = \sigma^2$ when $d(\mu^{(n)},\psi^{(n)})^2 \rightarrow \infty$. 
It follows from Theorem \ref{freqdomthm} that  $E[(a)|\theta^{(n)}]= O\big(d(\theta^{(n)},\psi^{(n)})^{-2} n^{-1/2}\big)$ and $E[(c)|\theta^{(n)}] = O\big(d(\theta^{(n)},\psi^{(n)})^{-4}\big)$. Defining the event $C$ as $\{ \vert d(\theta^{(n)},\psi^{(n)})^2 - E\big(d(\theta^{(n)},\psi^{(n)})^2\big) \vert > E\big(d(\mu^{(n)},\psi^{(n)})^2\big)/2\}$, we get $E\big(E[(a)|\theta^{(n)}]\big) = E\big(I_CE[(a)|\theta^{(n)}]\big) + E\big(I_{C^c}E[(a)|\theta^{(n)}]\big)$ and we can split $E\big(E[(c)|\theta^{(n)}]\big)$ similarly. By assumption. $d(\mu^{(n)},\psi^{(n)})^2 \rightarrow \infty$ so we get $\limsup_n E\big(I_{C^c}E[(a)|\theta^{(n)}]\big) = 0$ and from  Lemma \ref{cheblemma}, $\limsup_n E\big(I_CE[(a)|\theta^{(n)}]\big)  = 0$. Applying the same reasoning to $E[(c)]$ shows that $\limsup_n \big(E[(a)] + E[(c)] \big)= 0$. The remaining term is in the asymptotic risk $\limsup_n E\big(R(P^{(n)}_{\theta^{(n)}},\delta_{JS})\big)$ is

\noindent $ \limsup_n E[(b)] =  \limsup_n \sigma^2 E\big(I_{A}\frac{d(\theta^{(n)},\psi^{(n)})^2}{d(X^{(n)},\psi^{(n)})^2)}\big)$. By the reverse triangle inequality
\begin{align*}
    E\big(I_A&\frac{d(\theta^{(n)},\psi^{(n)})^2}{d(X^{(n)},\psi^{(n)})^2}\big) 
    \\
     \leq & \; E\bigg[ \frac{d(\theta^{(n)},\psi^{(n)})^2}{d(\theta^{(n)},\psi^{(n)})^2 - 2d(\theta^{(n)},\psi^{(n)})d(X^{(n)},\theta^{(n)}) + d(X^{(n)},\theta^{(n)})^2} \wedge
 \frac{d(\theta^{(n)},\psi^{(n)})^2}{\sigma^2}\bigg]
    \\
     \leq & \; E\bigg[ \frac{d(\theta^{(n)},\psi^{(n)})}{d(\theta^{(n)},\psi^{(n)}) - 2d(X^{(n)},\theta^{(n)}) }\wedge \frac{d(\theta^{(n)},\psi^{(n)})^2}{\sigma^2}\bigg].
\end{align*}
Let $D \coloneqq \{2d(X^{(n)},\theta^{(n)}) > d(\theta^{(n)},\psi^{(n)})^{1/2}\}$ from which Chebychev's inequality yields, 
$P(D|\theta^{(n)}) \coloneqq P\big(2d(X^{(n)},\theta^{(n)}) > d(\theta^{(n)},\psi^{(n)})^{1/2}|\theta^{(n)}\big) \leq 2^6m_6 d(\theta^{(n)},\psi^{(n)})^{-3}$. Using $D$ in the minimum above gives  
\begin{align*}
    E\bigg[& \frac{d(\theta^{(n)},\psi^{(n)})}{d(\theta^{(n)},\psi^{(n)}) - 2d(X^{(n)},\theta^{(n)}) }\wedge \frac{d(\theta^{(n)},\psi^{(n)})^2}{\sigma^2}\bigg\vert \theta^{(n)}\bigg] 
    \\
    \leq & \; E\big(I_{D^c}\frac{d(\theta^{(n)},\psi^{(n)})}{d(\theta^{(n)},\psi^{(n)}) - 2d(X^{(n)},\theta^{(n)}) } \big\vert \theta^{(n)}\big) +
 \frac{d(\theta^{(n)},\psi^{(n)})^2}{\sigma^2} P(D|\theta^{(n)})
    \\
     \leq & \; \frac{d(\theta^{(n)},\psi^{(n)})}{d(\theta^{(n)},\psi^{(n)}) - d(\theta^{(n)},\psi^{(n)})^{1/2}} +
 \frac{2^6m_6}{\sigma^2d(\theta^{(n)},\psi^{(n)})}.
\end{align*}
Lastly, the second term can be split by $I_C + I_{C^c}$ to show that the expectation of this term over $\theta^{(n)}$ is $ O\big(d(\mu^{(n)},\psi^{(n)})^{-1}\big)$. The first term, $1/\big(1-d(\theta^{(n)},\psi^{(n)})^{-1/2}\big)$, is arbitrarily close to $1$ when the event  $E \coloneqq \{d(\theta^{(n)},\psi^{(n)})^2 > M\}$, occurs for a large, fixed $M$. For any choice of $M > 0$,  $P(E^c) = O\big(d(\mu^{(n)},\psi^{(n)})^{-4}\big)$. As the first term is bounded above by $d(\theta^{(n)},\psi^{(n)})^2/\sigma^2$, the expectation of the first term times $I_{E^c}$ tends to zero in the limit. This shows that $\limsup_n \sigma^2 E\big(I_{A}\frac{d(\theta^{(n)},\psi^{(n)})^2}{d(X^{(n)},\psi^{(n)})^2)}\big) = \sigma^2$ as desired.

To show that $\lim_n P\big(L(\theta^{(n)},\delta_{JS}) > \sigma^2 + \epsilon\big) = 0$ we split up this probability as
\begin{align*}
    P\big(L(\theta^{(n)},\delta_{JS}) > \sigma^2 + \epsilon\big) & \leq P\big[(a) + (c) > \frac{\epsilon}{2}\big] +  P\big[\sigma^2\big(\frac{d(\theta^{(n)},\psi^{(n)})^2}{d(X^{(n)},\psi^{(n)})^2} - 1\big) > \frac{\epsilon}{2}\big].
\end{align*}
That the first probability tends to zero follows immediately from the bounds for the expectations of these terms developed above. Conditioning on $\theta^{(n)}$, the second probability can be re-written as \begin{align*}
    P\big[d(\theta^{(n)},&\psi^{(n)})^2  >  (1 + \frac{\epsilon}{2\sigma^2})d(X^{(n)},\psi^{(n)})^2| \theta^{(n)} \big] 
    \\
     & \leq  P[\frac{\epsilon}{2\sigma^2}\rho(X^{(n)},\psi^{(n)})^2 <  (1 + \frac{\epsilon}{2\sigma^2})\big(\rho(X^{(n)},\psi^{(n)})^2 - d(X^{(n)},\psi^{(n)})^2\big)| \theta^{(n)} \big].
\end{align*}
Here $\rho(X^{(n)},\psi^{(n)})^2 = E\big(d(X^{(n)},\psi^{(n)})^2|\theta^{(n)}\big) \geq d(\theta^{(n)},\psi^{(n)})^2$.
Taking $k = 1$ in Lemma \ref{cheblemma}, this probability is $O(n)$ independently of $\theta^{(n)}$, proving the result. 
\end{proof}

\begin{customthm}{7}
Assume that $X^{(n)} \sim P^{(n)}_{\theta^{(n)}} = \tilde{P}_{\theta_1^{(n)}} \times \cdots \times \tilde{P}_{\theta_n^{(n)}}$ and $\theta^{(n)} \sim Q^{(n)} = \tilde{Q} \times 
\cdots \times \tilde{Q}$ for all $n \in \mathbb{N}$. If $E\big(\tilde{d}(\widebar{X}^{(n)},E_2X_1^{(n)})^2\big) = O(n^{-1})$ with a multiplicative constant that only depends on $m$ and $l$, then there exists an $n^*(m,l)$ such that for $n \geq n^*$ then $E\big(R(P^{(n)}_{\theta^{(n)}},\delta_{JS})\big) < E\big(R(P^{(n)}_{\theta^{(n)}},X^{(n)})\big)$,  where $\delta_{JS}$ is the adaptive shrinkage estimator given by \eqref{jscomponents} with $\psi_i^{(n)} = \widebar{X}^{(n)}$. Furthermore, the same $n^*$ is valid for any distributions $\tilde{P}_{\theta_i}^{(n)} \in \mathcal{P}^{(1)}_m$ and $\tilde{Q}^{(n)} \in \mathcal{P}^{(1)}_l$.
\end{customthm}

\begin{proof} 
To ease notation call $E_2X^{(n)} = \omega$. Let $w_1(X^{(n)})$ be the shrinkage weight formed using $\widebar{X}^{(n)}$ as a shrinkage point and $w_2(X^{(n)})$ be the shrinkage weight formed using $\omega$. The proof of Theorem \ref{Bayesdom1} shows that there exists an $\alpha < \sigma^2$ where we have $E\big(d([X^{(n)},\omega]_{w_2(X^{(n)})},\theta^{(n)})^2\big) \leq \alpha < \sigma^2$ for $n \geq N_1(m,l,\tilde{d}(\tilde{\mu},\tilde{\omega}))$, as $\tilde{\omega}$ is fixed and not data dependent. The value of $\alpha$ can be taken to depend only on $m$ and $l$. We want to show that $[X^{(n)},\widebar{X}^{(n)}]_{w_1(X^{(n)})}$ is sufficiently close to $[X^{(n)},\omega]_{w_2(X^{(n)})}$ so that this second estimator also has a lower Bayes risk than $X^{(n)}$. Throughout we drop all $(n)$ superscripts to further ease notation. We have that 
\begin{align*}
    d([X,\widebar{X}]_{w_1},\theta)^2 \leq \bigg(d([X,\widebar{X}]_{w_1},[X,\omega]_{w_1}) + d([X,\omega]_{w_1},[X,\omega]_{w_2}) + d([X,\omega]_{w_2},\theta) \bigg)^2.
\end{align*}
As $E\big(d([X,\omega]_{w_2},\theta)^2\big)$ is less than or equal to  $\alpha$, expanding the above expression and using Cauchy-Schwartz of any cross product terms it will suffice to show that $E\big(d([X,\widebar{X}]_{w_1},[X,\omega]_{w_1})^2\big) \rightarrow 0$ and $E\big(d([X,\omega]_{w_1},[X,\omega]_{w_2})^2\big) \rightarrow 0$ at known rates as $n \rightarrow \infty$. In a Hadamard space, pairs of geodesics have the following convexity property $d([x,y]_t,[w,z]_t) \leq (1-t)d(x,w) + td(y,z)$ \cite{Sturm}. By the assumption that $E\big(d(\widebar{X},\omega)^2\big) = O(n^{-1})$ we have,
\begin{align*}
    E\big(d([X,\widebar{X}]_{w_1},[X,\omega]_{w_1})^2\big)
    & \leq E\big( \big[(1-w_1)d(X,X) + w_1d(\widebar{X},\omega)\big]^2\big)
    \\
    &\leq E\big(d(\widebar{X},\omega)^2\big) = \frac{C}{n} \rightarrow 0. 
\end{align*}
The other term that we wish to show has a limit of zero can be written as
\begin{align*}
    E\big(d([&X,\omega]_{w_1},[X,\omega]_{w_2})^2\big)  
    \\
    = & \; E\big((w_1 - w_2)^2 d(X,\omega)^2\big)
    \\
     = & \; \sigma^4E\bigg[\big(I_{\{d(X,\widebar{X})^2 \geq \sigma^2\}}d(X,\omega)^2 - I_{\{d(X,\omega)^2 \geq \sigma^2\}}d(X,\widebar{X})^2\big)^2d(X,\omega)^{-2}d(X,\widebar{X})^{-4}\bigg]
    \\
     \leq & \; \sigma^4 E\bigg[I_{\{d(X,\widebar{X})^2 \geq \sigma^2\} \cap \{d(X,\omega)^2 \geq \sigma^2\}}\big(d(X,\omega)^2 - d(X,\widebar{X})^2\big)^2d(X,\omega)^{-2}d(X,\widebar{X})^{-4}\bigg] + 
    \\
    & \; \sigma^4 E\bigg[I_{\{d(X,\widebar{X})^2 < \sigma^2\} \cup \{d(X,\omega)^2 < \sigma^2\}} d(X,\omega)^2\bigg]
    \\
     \leq & \;  \sigma^{4}E\bigg[4\sigma^{-4}\big(d(X,\omega) - d(X,\widebar{X})\big)\big)^2\bigg] +
    \\
    & \; \sigma^4 P\bigg(\{d(X,\widebar{X})^2 < \sigma^2\} \cup \{d(X,\omega)^2 < \sigma^2\} \bigg)E\big(d(X,\omega)^4\big)^{1/2}
    \\
     \leq & \;  4E\big(d(\widebar{X},\omega)^2\big) + \sigma^4 P\bigg(\{d(X,\widebar{X})^2 < \sigma^2\} \cup \{d(X,\omega)^2 < \sigma^2\} \bigg)E\big(d(X,\omega)^4\big)^{1/2} 
\end{align*}
The bias-variance inequality shows that
\begin{align*}
    E\big(d(X,\omega)^2\big) = & \; E\big(E\big(d(X,\omega)^2|\theta\big)\big)
    \\
     \geq & \; E\big(E\big(\sigma^2 + d(\theta,\omega)^2|\theta\big)\big)
     \\
     \geq & \; \sigma^2 + \tau^2 + d(\mu,\omega)^2.
\end{align*}
Applying the weak law of large numbers to $P\big(d(X,\omega)^2 < \sigma^2\big) \leq P\big( \vert d(X,\omega)^2 - E\big(d(X,\omega)^2\big) \ \vert > \tau^2\big)$ shows that this term is $O(n^{-1})$. Similarly,
\begin{align*}
    P\big(d(X,\widebar{X})^2 <  \sigma^2 \big) \leq & P\big(\vert d(X,\widebar{X})^2 - d(X,\omega)^2 \vert > \tau^2/2\big) + 
    \\
    & P\big(\vert d(X,\omega)^2 - E\big(d(X,\omega)^2\big)\vert > \tau^2/2\big).
\end{align*}
Chebychev's inequality can be used on the first term with
\begin{align*}
    E\big(\vert d(X,\widebar{X})^2 - d(X,\omega)^2 \vert \big) \leq E\big( d(\widebar{X},\omega)^2)^{1/2}E\big[ \big(d(X,\widebar{X}) + d(X,\omega)\big)^2 \big]^{1/2}.
\end{align*}
To complete the proof that $    E\big(d([X,\omega]_{w_1},[X,\omega]_{w_2})^2\big) = O(n^{-1/2})$ it suffices to show that the terms $E\big(d(X,\omega)^4\big)$ and $E\big(d(X,\widebar{X})^2\big)$ can be bounded above by expressions involving $m$ and $l$. We first bound $d(\mu,\omega)$, \begin{align*}
    d(\mu,\omega)^2 \leq E\big(d(X,\mu)^2\big) \leq 4E\big(E\big(d(X,\theta)^2|\theta\big)\big) + 4E\big(d(\theta,\mu)^2\big) \leq 4\sigma^2 + 4\tau^2.
\end{align*}
The triangle inequality,
\begin{align*}
    d(X,\widebar{X}) \leq d(X,\theta) + d(\theta,\mu) + d(\mu,\omega) + d(\omega,\widebar{X})
\end{align*}
along with the convexity of $x \rightarrow x^c$ can therefore used to bound $E\big(d(X,\widebar{X})^2\big)$ and similarly for $E\big(d(X,\omega)^4\big)$ as needed. 
\end{proof}
Notice that it is not necessary that $E\big(d(\widebar{X},\omega)^2\big)$ be $O(n^{-1})$ in the above proof. As long as $E\big(d(\widebar{X},\omega)^2\big) \rightarrow 0$ the proof will hold. However, the $n^*$ needed will vary depending on the rate at which $E\big(d(\widebar{X},\omega)^2\big)$ tends to $0$.

\section{Counterexamples}\label{appB}

\subsection{The tower rule need not hold in a Hadamard space} 
Consider the metric tree tripod space pictured in Figure \ref{tripod}. The tree is constructed such that the points $A,C,B$ have edge lengths of $1,1$ and $2$ respectively from the central vertex. Points in this space are points along the edges of the graph and the distance between points in the graph is given by the shortest path distance. For example, $d(A,B) = 3$ while $d(A,C) = 2$. It can be checked that this space satisfies the CAT(0) inequality and so is Hadamard. Suppose that $X$ is a random object that is uniformly distributed on $A,B,C$. The Fr\'echet mean of $X$ is the central vertex. Now let $Y$ be the real valued random variable that has $Y = 0$ when $X = A,C$ and $Y = 1$ when $X = B$. Conditional on $Y = 0$ we know that $X$ must equal either $A$ or $C$ with probability $1/2$ each so that $E_2(X|Y = 0)$ is the center vertex, $E_2Y$. If $Y = 1$ then $E_2(X|Y = 1) = B$. Therefore the graph valued random object $E_2(X|Y)$ equals $E_2Y$ with probability $2/3$ and $B$ with probability $1/3$. It follows that $E_2\big(E_2(X|Y)\big) = [E_2X,C]_{1/3} \neq E_2X$ which shows that the tower rule does not hold in this scenario.
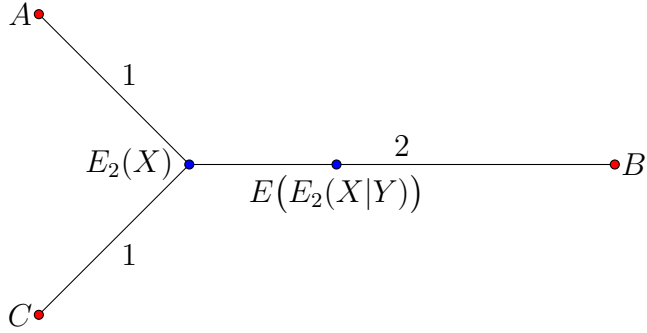
\begin{figure}
    \centering
    
\pgfmathsetmacro{\sqet}{2 + 2*sqrt(8)}
\pgfmathsetmacro{\sqhalf}{(\sqet - 2) /2 + 2}

\begin{tikzpicture}
\draw (0,0) -- (2,2);
\draw (0,4) -- (2,2);
\draw (2,2) -- (\sqet,2);
\draw (1.2,.8) node{$1$};
\draw (1.2,3.2) node{$1$};
\draw (\sqhalf,2.25) node{$2$};
\draw[fill = red] (0,0) circle (.6mm);
\draw[fill = red] (\sqet,2) circle (.6mm);
\draw[fill = red] (0,4) circle (.6mm);
\draw (-.25,4) node{$A$};
\draw (\sqet + 0.25,2) node{$B$};
\draw (-.25,0) node{$C$};
\draw[fill = blue] (2,2) circle (.6mm);
\draw (1.2,2) node{$E_2(X)$};
\draw[fill = blue] (3 + .25*\sqet/2,2) circle (.6mm);
\draw (3 + .25*\sqet/2,2-.4)  node{$E\big(E_2(X|Y)\big)$};
\end{tikzpicture}
    \caption{Tripod space}
    \label{tripod}
\end{figure}

\subsection{Ineffective shrinkage in a positively curved space}
Consider the distribution of $X$ on the circle $\mathbb{S}^1$ where $X = (\cos(\theta),\sin(\theta))$, $\theta \sim Unif[-\pi/2,\pi/2]$ so that the (unique) Fr\'echet mean of $X$ is $(1,0)$. Consider the shrinkage estimator $[X,\psi]_t$ for $\psi \in \mathbb{S}^1$. If $\psi = (\cos(\theta_\psi),\sin(\theta_\psi))$ is chosen such that $\vert \theta_\psi \vert \leq \pi/2$ then this setting is isometric to performing shrinkage estimation on a uniform distribution in $\mathbb{R}$ so that there does exist a $t$ so that $E\big(d([X,\psi]_t,(1,0))^2\big) < E\big(d(X,(1,0))^2\big)$. Conversely, if $\psi = (-1,0)$ is antipodal to $(1,0)$ then for any non-zero amount of shrinkage $t$,  $d\big([X,\psi]_t,(1,0)\big) > d\big(X,(1,0)\big)$. If we then take $X_i, i = 1,...,n$ to be i.i.d with the aforementioned distribution on $\mathbb{S}^1$ we see that the James-Stein estimator $\delta_{JS}$ will necessarily have larger risk than $X \coloneqq (X_1,...,X_n)$ regardless of how large $n$ is taken to be. Note that the circle has constant zero sectional curvature, being locally isometric to $\mathbb{R}$, but has positive Alexandrov curvature.

\section{Fr\'echet Means for Metric Trees}
\label{appC}
In general, the computation of Fr\'echet means can be computationally expensive. For metric trees the situation is straightforward. We provide an efficient gradient descent type algorithm for computing the sample Fr\'echet mean of points lying in a metric tree that can be used to compute a data driven shrinkage point. We assume that the tree has at most $m$ vertices and the maximum degree of each vertex is $D$. For simplicity, all edges are assumed to have weight $1$. The extension of this algorithm to more general weighted trees is straightforward. To ease notation choose an arbitrary root of the tree and represent all vertices of the tree as $v_{(i_1,\ldots,i_k)}, (i_1 ,\ldots,i_k) \in I$ where $k$ is the depth of vertex $v_{(i_1,\ldots,i_k)}$ from the chosen root. All the vertices that are adjacent to the root are identified by $v_{(i_1)}, (i_1) \in I$ and all the vertices distinct from the root that are adjacent to $v_{(i_1)}$ are denoted by $v_{(i_1,i_2)},(i_1,i_2) \in I$ and so on. The set $I$ indexes all possible sequences of unique edges from the root and can be identified with a subset of $\coprod_{k = 1}^m \{1,\ldots,D\}^k $ that has the property that if $(i_1,\ldots,i_{n+1})\in I$ then  $(i_1,\ldots,i_{n}) \in I$. The vertex $v_{\emptyset}$ is taken to represent the root itself. 

The goal of this algorithm is to minimize the Fr\'echet function of the $x_i$'s, defined by $f(v) = \sumonn d(v,x_i)^2$, over all possible points $v \in \mathcal{T}$. The general idea of the algorithm is to start at the vertex $v_{\emptyset}$ and look to see if moving along any edges connected to this vertex reduces the sample Fr\'echet function. This is done by computing the directional derivative of the Fr\'echet function in the direction of each of the finitely many edges that one can move along. If there exists an edge where the directional derivative is negative, move along this edge to the next adjacent vertex, $v_{(i_1)}$. Repeating this process creates a sequence of vertices $v_{\emptyset},v_{(i_1)},v_{(i_1,i_2)},\ldots$ The process terminates at step $k$ when either $v_{(i_1,\ldots,i_k)}$ is found to be optimal or the Fr\'echet function is reduced by moving back to vertex $v_{(i_1,\ldots,i_{k-1})}$ along the edge $[v_{(i_1,\ldots,i_{k-1})},v_{(i_1,\ldots,i_k)}]$. In the later case the sample Fr\'echet mean lies in the interior of the edge $[v_{(i_1,\ldots,i_{k-1})},v_{(i_1,\ldots,i_k)}]$. 

Given points $x_1,\ldots,x_n \in \mathcal{X}$ the root is the sample Fr\'echet mean of these points if and only if
\begin{align}
\label{frectreeinqe}
    \sum_{j = 1}^nd(x_j,[v_{\emptyset},v_{(i_1)}]_{\epsilon})^2 > \sum_{j = 1}^n d(x_j,v_{\emptyset})^2 
\end{align}
for all $(i_1) \in I$ and all small enough $\epsilon$. For each $(i_1) \in I$ let $S_{(i_1)} \coloneqq \{x_j \neq v_{\emptyset} : [x_j,v_{\emptyset}] \cap [v_{(i_1)},v_{\emptyset}] \neq \emptyset\}$. Choose an $\epsilon$ small enough such that $[v_{\emptyset},v_{(i_1)}]_{\epsilon} \cap \{x_1,\ldots,x_n\} \subset \{v_{\emptyset}\}$, then we can rewrite \eqref{frectreeinqe} as
\begin{align*}
    \sum_{(\alpha) \neq (i_1)} \sum_{x_j \in S_{(\alpha)}}(d(x_j,v_{\emptyset}) + \epsilon)^2 + \sum_{x_j = v_{\emptyset}}\epsilon^2 + \sum_{x_j \in S_{(i_1)}}(d(x_j,v_\emptyset) - \epsilon)^2 >  \sum_{j = 1}^n d(x_j,v_{\emptyset})^2 
\end{align*}
Taking derivatives of the left hand side of this equation with respect to $\epsilon$ shows that a necessary and sufficient condition for $v_{\emptyset}$ to be the Fr\'echet mean is that 
\begin{align}
\label{frectreeineq2}
    \sum_{(\alpha) \neq (i_1)} \sum_{x_j \in S_{(\alpha)}}d(x_j,v_{\emptyset}) \geq \sum_{x_j \in S_{(i_1)}}d(x_j,v_\emptyset)
\end{align}
for all $(i_1) \in I$. If $\eqref{frectreeineq2}$ does not hold for some $(i_1)$ then we move to vertex $v_{(i_1)}$. Suppose we have moved to vertex $v_{(i_1,\ldots,i_k)}$ from $v_{(i_1,\ldots,i_{k-1})}$. In a similar fashion, define $S_{(i_1,\ldots,i_k,i_{k+1})} \coloneqq \{x_j \neq v_{(i_1,\ldots i_k)}: [v_{(i_1,\ldots,i_k)},v_{(i_1,\ldots,i_k,i_{k+1})}] \cap [v_{(i_1,\ldots,i_k)},x_j]\neq \emptyset \}$ and $S^*_{(i_1,\ldots,i_k)} \coloneqq \{x_j \neq v_{(i_1,\ldots i_k)}: [v_{(i_1,\ldots,i_k)},v_{(i_1,\ldots,i_{k-1})}] \cap [v_{(i_1,\ldots,i_k)},x_j]\neq \emptyset\}$. Denote, $\sum_{x_j \in S_{(i_1,\ldots,i_k,\alpha)}}d(x_j,v_{(i_1,\ldots,i_k)})$ by  $\mathit{\Sigma}S_{v_{(i_1,\ldots,i_k,\alpha)}}$ and $ \sum_{x_j \in S^*_{(i_1,\ldots,i_k)}}d(x_j,v_{(i_1,\ldots,i_k)})$ by $\mathit{\Sigma}S^*_{(i_1,\ldots,i_k)}$. The vertex $v_{(i_1,\ldots,i_k)}$ is the desired Fr\'echet mean if and only if
\begin{align}
    \sum_{\alpha \neq i_{k+1}} \mathit{\Sigma}S_{v_{(i_1,\ldots,i_k,\alpha)}} + \mathit{\Sigma}S^*_{v_{(i_1,\ldots,i_k)}} & \geq \mathit{\Sigma}S_{v_{(i_1,\ldots,i_k,i_{k+1})}}
    \label{treefrecmn3}
    \\
\sum_{\alpha} \mathit{\Sigma}S_{v_{(i_1,\ldots,i_k,\alpha)}}   & \geq \mathit{\Sigma}S^*_{v_{(i_1,\ldots,i_k)}}
\label{vertexgoback}
\end{align}
both hold for all possible choices of $i_{k+1}$. The algorithm terminates if either both \eqref{treefrecmn3} and \eqref{vertexgoback} hold, or if \eqref{vertexgoback} does not hold. If \eqref{vertexgoback} does not hold then the Fr\'echet mean will be in the interior of $[v_{(i_1,\ldots, i_{k-1})},v_{(i_1,\ldots,i_k)}]$. For each $x_j$ in that is in $[v_{(i_1,\ldots, i_{k-1})},v_{(i_1,\ldots,i_k)}]$ identify $x_j = [v_{(i_1,\ldots, i_{k-1})},v_{(i_1,\ldots,i_k)}]_\epsilon$ with $\epsilon$. If $x_j$ is not in the edge $[v_{(i_1,\ldots, i_{k-1})},v_{(i_1,\ldots,i_k)}]$ and if $v_{(i_1,\ldots,i_{k-1})}$ is closer to $x_j$ than $v_{(i_1,\ldots,i_k)}$ identify $x_j$ with the number $-d(x_j,v_{(i_1,\ldots,i_{k-1})})$. Otherwise identify $x_j$ with $d(v_{(i_1,\ldots,i_{k-1})},v_{(i_1,\ldots,i_k)}) + d(v_{(i_1,\ldots,i_k)},x_j)$. Under these identifications, the sample Fr\'echet mean of the $x$'s is the Euclidean mean of these numbers. That is, if $t_x$ is the Euclidean mean of these numbers then the sample Fr\'echet mean is $[v_{(i_1,\ldots, i_{k-1})},v_{(i_1,\ldots,i_k)}]_{t_x}$. At step $k$ of the algorithm one only needs to compute the sums $\mathit{\Sigma}S_{(i_1,\ldots,i_{k})}$ since $\mathit{\Sigma}S^*_{(i_1,\ldots,i_{k-1})} = \sum_{\alpha \neq i_{k-1}} (\mathit{\Sigma}S_{(i_1,\ldots,i_{k-2},\alpha)} + d(v_{(i_1,\ldots,i_{k-1})},v_{(i_1,\ldots,i_{k})}) \vert S_{(i_1,\ldots,i_{k-2},\alpha)} \vert )$ is known from the previous step. The worst case run time is $O(nm)$ as a sum over $n$ numbers is computed for each vertex visited. The number of vertices visited cannot be any larger than the depth of the tree, $m$. 
\end{appendices}

\bibliography{biblio}

\end{document}